\documentclass[aihp]{imsart}

\RequirePackage{amsthm,amsmath,amsfonts,amssymb,mathrsfs,mathtools,enumerate}
\RequirePackage[numbers]{natbib}
\RequirePackage[colorlinks,citecolor=blue,urlcolor=blue]{hyperref}
\RequirePackage{graphicx}

\startlocaldefs

\DeclareMathOperator{\Var}{Var} \DeclareMathOperator{\dist}{dist}
\DeclareMathOperator{\Dist}{Dist} \DeclareMathOperator{\diam}{diam}

\theoremstyle{plain}
  \newtheorem{theorem}{Theorem}[section]
  \newtheorem{lemma}[theorem]{Lemma}
  \newtheorem{proposition}[theorem]{Proposition}
  \newtheorem{corollary}[theorem]{Corollary}

\theoremstyle{remark}
    \newtheorem{definition}[theorem]{Definition}
    \newtheorem{remark}[theorem]{Remark}

\endlocaldefs

\begin{document}
\begin{frontmatter}

\title{Convergence of limit shapes for 2D near-critical first-passage percolation}
\runtitle{Limit shapes for near-critical FPP}

\begin{aug}
\author[]{\inits{C.-L.}\fnms{Chang-Long} \snm{Yao}\ead[label=e1]{deducemath@126.com}}
\address[]{Academy of Mathematics and Systems Science, CAS, Beijing, China. \printead{e1}}

\end{aug}

\begin{abstract}
We consider Bernoulli first-passage percolation on the triangular
lattice in which sites have 0 and 1 passage times with probability
$p$ and $1-p$, respectively.  For each $p\in(0,p_c)$, let $\mathcal
{B}(p)$ be the limit shape in the classical ``shape theorem'', and
let $L(p)$ be the correlation length.  We show that as $p\uparrow p_c$, the rescaled limit shape $L(p)^{-1}\mathcal {B}(p)$ converges
to a Euclidean disk.  This improves a result of Chayes et al.
[\emph{J. Stat. Phys.} \textbf{45} (1986) 933--951].  The proof
relies on the scaling limit of near-critical percolation established
by Garban et al. [\emph{J. Eur. Math. Soc.} \textbf{20} (2018)
1195--1268], and uses the construction of the collection of
continuum clusters in the scaling limit introduced by Camia et al.
[\emph{Springer Proceedings in Mathematics \& Statistics},
\textbf{299} (2019) 44--89].
\end{abstract}

\begin{abstract}[language=french]
???.
\end{abstract}



\end{frontmatter}

\section{Introduction}
\subsection{The model and main result}\label{s1}
First-passage percolation (FPP) is a stochastic growth model which
was first introduced by Hammersley and Welsh \cite{HW65} in 1965.  For general background on FPP, we refer
the reader to the recent survey \cite{ADH17}.  In this paper, we will focus on FPP defined on the triangular
lattice $\mathbb{T}$, since the proof of our main result relies on the scaling limit of near-critical site percolation
on $\mathbb{T}$ established by Garban, Pete and Schramm \cite{GPS18}, while such result has not been proved for other planar lattices.

The model is defined as follows.  Let $\mathbb{T}$ be the triangular
lattice embedded in $\mathbb{C}$, with site (vertex) set
\begin{equation*}
V(\mathbb{T}):=\{x+ye^{\pi i/3}\in\mathbb{C}:x,y\in\mathbb{Z}\},
\end{equation*}
and bond (edge) set $E(\mathbb{T})$ obtained by connecting all pairs
$u,v\in V(\mathbb{T})$ for which ${|u-v|=1}$, where $|\cdot|$ denotes the Euclidean norm.  We say that $u$
and $v$ are neighbors if $(u,v)\in E(\mathbb{T})$.  A \textbf{path} is a
sequence $(v_0,\ldots,v_n)$ of distinct sites such that $v_{j-1}$
and $v_j$ are neighbors for all $j=1,\ldots,n$.  Let $\{t(v):v\in
V(\mathbb{T})\}$ be an i.i.d. family of nonnegative random variables
with common distribution function $F$.  For a path $\gamma$,
we define its passage time by $T(\gamma):=\sum_{v\in \gamma}t(v)$.
For $A,B\subset V(\mathbb{T})$, the \textbf{first-passage time} from $A$ to $B$ is
defined by
\begin{equation*}
T(A,B):=\inf \{T(\gamma):\gamma \mbox{ is a path from a site in $A$
to a site in $B$}\}.
\end{equation*}
A \textbf{geodesic} from $A$ to $B$ is a path $\gamma$ from $A$ to $B$ such that
$T(\gamma)=T(A,B)$.  If $x,y\in\mathbb{C}$, we define $T(x,y):=T(\{x'\},\{y'\})$,
where $x'$ (resp. $y'$) is the site in $V(\mathbb{T})$ closest to
$x$ (resp. $y$).  Any possible ambiguity can be avoided by ordering
$V(\mathbb{T})$ and taking the site in $V(\mathbb{T})$ smallest for
this order.

In this paper, we concentrate on \textbf{Bernoulli FPP} on
$\mathbb{T}$.  More precisely, for each $p\in [0,1]$, we consider the i.i.d. family $\{t(v):v\in V(\mathbb{T})\}$ of
Bernoulli random variables with parameter $p$, that is, $t(v)=0$ with probability $p$ and $t(v)=1$ with probability $1-p$.
This gives rise to a product probability measure on the set
of configurations $\{0,1\}^{V(\mathbb{T})}$, which is denoted by $\mathbf{P}_p$, the corresponding expectation being $\mathbf{E}_p$.
We refer to a site $v$ with $t(v)=0$ simply as \textbf{open};
otherwise, \textbf{closed}.  So one can view Bernoulli FPP as
Bernoulli site percolation on $\mathbb{T}$ (see, e.g., \cite{BR06,Gri99}
for background on percolation; note that for Bernoulli FPP,
a site is open when it takes the value 0, while in the percolation
literature, a site is open usually means that the site takes the
value 1).  We usually represent it as a random coloring of the faces
of the dual regular hexagonal lattice $\mathbb{H}$, each face centered at
$v\in V(\mathbb{T})$ being blue ($t(v)=0$) or yellow ($t(v)=1$); see Figure \ref{pointtopoint}.
Sometimes we view the site $v$ as the hexagon $H_v$ in $\mathbb{H}$
centered at $v$.

\begin{figure}
\begin{center}
\includegraphics[height=0.35\textwidth]{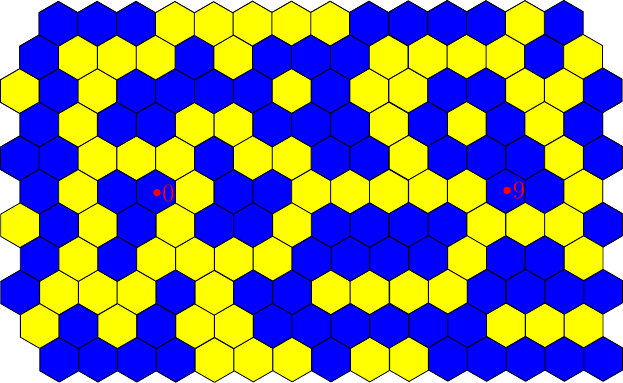}
\caption{Bernoulli FPP on $\mathbb{T}$.  Each
hexagon of the hexagonal lattice
$\mathbb{H}$ represents a site of $\mathbb{T}$, and is
colored blue ($t(v)=0$) or
yellow ($t(v)=1$).  Here, the
first-passage time $T(0,9)=2$.}\label{pointtopoint}
\end{center}
\end{figure}

Suppose $p\in[0,1]$.  It follows from the subadditive ergodic theorem that, for any
$z\in\mathbb{C}$, there is a constant $\mu(p,z)$, such that
\begin{equation}\label{e28}
\lim_{n\rightarrow\infty}\frac{T(0,nz)}{n}=\mu(p,z)\qquad\mbox{$\mathbf{P}_p$-a.s.
and in $L^1$}.
\end{equation}
We call $\mu(p):=\mu(p,1)$ the \textbf{time constant}.  As usual, we
write $a_{0,n}:=T(0,n)$.

It is well known (see, e.g., Kesten's Theorem 6.1 in \cite{Kes86}
for general FPP) that
\begin{equation}\label{e112}
\mu(p)=0\quad\mbox{if and only if}\quad p\geq p_c,
\end{equation}
where $p_c=p_c(\mathbb{T})=1/2$ is the critical point for Bernoulli
site percolation on $\mathbb{T}$.  In this paper, we will focus on
the subcritical case, where $p<p_c$.

The fundamental object of study is
\begin{equation*}
B(t):=\{z\in\mathbb{C}:T(0,z)\leq t\},
\end{equation*}
the set of points reached from the origin 0 within a time $t\geq 0$.  Using (\ref{e28}) and (\ref{e112}), it is easy to deduce that
$\mu(p,z)$ is a norm on $\mathbb{C}$ for each fixed $p<p_c$.  The
unit ball in $\mu$-norm is called the \textbf{limit shape} and will
be denoted by
\begin{equation*}
\mathcal {B}(p):=\{z\in\mathbb{C}:\mu(p,z)\leq 1\}.
\end{equation*}
It is the limit of $B(t)$ in the following sense:  By the famous
Cox-Durrett shape theorem (see, e.g., Theorem 2.17 in \cite{ADH17}),
for each $p<p_c$ and each $\epsilon>0$,
\begin{equation}\label{e1}
\mathbf{P}_p\left[(1-\epsilon)\mathcal {B}(p)\subset
\frac{B(t)}{t}\subset(1+\epsilon)\mathcal {B}(p)\mbox{ for all large
}t\right]=1.
\end{equation}
Moreover, $\mathcal {B}(p)$ is a convex, compact set with non-empty
interior, and has the symmetries of $\mathbb{T}$ that fix the
origin.  Apart from this, little is known about the geometry of
$\mathcal {B}(p)$.

We want to study the asymptotics of $\mu(p)$ and $\mathcal {B}(p)$,
as $p\uparrow p_c$.  For this purpose, we will use a very useful
concept from near-critical percolation: \textbf{correlation length}.
Roughly speaking, the system ``looks like'' critical percolation on
scales smaller than correlation length, while ``notable''
super/sub-critical behavior emerges above this length.  There are
several natural definitions of correlation length, and the
corresponding lengths turn out to be of the same order of magnitude.
See Section \ref{correlation} for two different definitions of it,
denoted by $L(p)$ and $L_{\epsilon}(p)$, respectively.  The former
is defined in terms of the alternating 4-arm events, while the
latter in terms of the box-crossing events.

Chayes, Chayes and Durrett \cite{CCD86} proved that,
\begin{equation}\label{p8}
\mu(p)\asymp L(p)^{-1}\quad\mbox{as }p\uparrow p_c.
\end{equation}
This result together with (\ref{e14}) implies that
$\mu(p)=(p_c-p)^{4/3+o(1)}$ as $p\uparrow p_c$.  We note that
(\ref{p8}) was proved in \cite{CCD86} for bond version of
subcritical Bernoulli FPP on $\mathbb{Z}^2$; the same proof applies
to our site version on $\mathbb{T}$.

Let $\mathbb{U}:=\{z\in\mathbb{C}:|z|=1\}$ denote the unit circle
centered at 0.  For $r>0$ and $z\in\mathbb{C}$, let
$\mathbb{D}_r(z):=\{x\in\mathbb{C}:|x|<r\}$ denote the open
Euclidean disk of radius $r$ centered at $z$, and let
$\overline{\mathbb{D}}_r(z)$ denote its closure.  Write
$\mathbb{D}_r=\mathbb{D}_r(0)$ and $\mathbb{D}=\mathbb{D}_1$.  Our
main result below improves (\ref{p8}), and states that $\mathcal
{B}(p)$ is asymptotically circular as $p\uparrow p_c$.
\begin{theorem}\label{t2}
There exists a constant $\nu>0$, such that
\begin{equation*}
\lim_{p\uparrow p_c}L(p)\mu(p,u)=\nu\quad\mbox{uniformly in
}u\in\mathbb{U}.
\end{equation*}
In particular, when $p\uparrow p_c$, the normalized limit shape
$L(p)^{-1}\mathcal {B}(p)$ tends to the Euclidean disk
$\overline{\mathbb{D}}_{1/\nu}$ in the Hausdorff metric (\ref{e2}).
\end{theorem}
From the above theorem, it is easy to extract the following
corollary (see Section \ref{coro} for its proof):
\begin{corollary}\label{cor1}
Suppose $\epsilon\in(0,1/2)$.  There exists a constant
$\nu_{\epsilon}>0$ depending on $\epsilon$, such that
\begin{equation*}
\lim_{p\uparrow p_c}L_{\epsilon}(p)\mu(p,u)=
\nu_{\epsilon}\quad\mbox{uniformly in }u\in\mathbb{U}.
\end{equation*}
In particular, when $p\uparrow p_c$, the normalized limit shape
$L_{\epsilon}(p)^{-1}\mathcal {B}(p)$ tends to the Euclidean disk
$\overline{\mathbb{D}}_{1/\nu_{\epsilon}}$ in the Hausdorff metric
(\ref{e2}).
\end{corollary}
\begin{remark}
We cannot give the explicit values of the limits in Theorem \ref{t2}
and Corollary \ref{cor1}.  In our proof, the subadditive ergodic
theorem plays a crucial role in showing the existence of the limit
but gives no insight for the exact value of the limit.  See
Section \ref{strategy} for a sketch of the proof.
\end{remark}
\begin{remark}
It is believed that for each fixed $p<p_c$, the limit shape
$\mathcal {B}(p)$ is not a Euclidean disk, since the anisotropy of
$\mathbb{T}$ may persist in the limit.  Although we cannot prove
such a statement for all $p<p_c$, we provide a short argument to
show that this is indeed the case so long as $p$ is sufficiently
small.  For general FPP, a theorem of Cox and Kesten (see
\cite{CK81} or Theorem 2.7 in \cite{ADH17}) states that, the time
constant is continuous under weak convergence of the site-weight
distributions of $t(v)$.  This implies that $\mu(p)$ is continuous in $p$,
since for each $p_0\in[0,1]$, the Bernoulli distribution Ber$(p)$
converges weakly to Ber$(p_0)$ as $p$ tends to $p_0$.  Furthermore,
by Remark 6.18 of \cite{Kes86}, for each $u\in\mathbb{U}$, the
function $\mu(p,u)$ is continuous in $p$, and this continuity is
even uniform in $u$.  Therefore, for any $\epsilon>0$, we can choose
$p\in(0,p_c)$ sufficiently small to make the Hausdorff distance
between $\mathcal {B}(p)$ and $\mathcal {B}(0)$ smaller than
$\epsilon$.  So for $p$ small enough, $\mathcal {B}(p)$ is not a
Euclidean disk since $\mathcal {B}(0)$ is a regular hexagon.  We
want to mention that in high dimensions, Kesten proved that the
limit shape of the Eden model is not a Euclidean ball, and it is
conjectured that this is true for all dimensions $d\geq 2$; see,
e.g., Theorem 6.2 and Question 6.1.1 in \cite{ADH17}.
\end{remark}

\subsection{Strategy of proof}\label{strategy}
Let us outline the proof of Theorem \ref{t2}.  Our method involves
three steps:

\textbf{Step 1.} \emph{Construction of the scaling limit of cluster
ensemble.} Based on the scaling limit results in \cite{GPS13,SS11}
for critical percolation under the quad-crossing topology, Camia,
Conijn and Kiss \cite{CCK19} constructed the scaling limit of (open)
clusters for critical percolation, which is a collection of compact
sets called the ``continuum clusters''.  We extend this result to the
near-critical case by using the approach from \cite{CCK19} and the
scaling limit result for near-critical percolation from \cite{GPS18}.

\textbf{Step 2.} \emph{Study of FPP on the continuum cluster ensemble.}  We
need to define FPP for the scaling limit constructed in Step 1. For
this purpose, we call a finite sequence $\Gamma$ of distinct continuum
clusters a (continuum) chain if any two consecutive clusters in $\Gamma$ touch each other, and let $|\Gamma|$ denote the number of clusters in $\Gamma$.
The first-passage time between
two continuum clusters $\mathcal{C},\mathcal{C}'$ is defined by
\begin{equation*}
T(\mathcal{C},\mathcal{C}'):=\inf\{|\Gamma|-1:\Gamma\mbox{ is a chain from $\mathcal{C}$ to $\mathcal{C}'$}\}.
\end{equation*}
This enables us to define the ``point-to-point'' passage
time $T_{m,n}$ (see Figure \ref{chain}) as the first-passage time between two suitably chosen continuum
clusters contained respectively in the disks $\mathbb{D}(m)$ and
$\mathbb{D}(n)$.  Similarly, for Bernoulli FPP on the rescaled lattice $L(p)^{-1}\mathbb{T}$ we denote by
$T^p_{m,n}$ the first-passage time between two
suitably chosen discrete clusters contained respectively in the disks
$\mathbb{D}(m)$ and $\mathbb{D}(n)$.  Next, we show that under a
coupling such that the percolation configuration on
$L(p)^{-1}\mathbb{T}$ converges almost surely to the quad-crossing scaling limit
as $p\uparrow p_c$, the first-passage time $T^p_{0,n}$ converges to
$T_{0,n}$ in probability.  Furthermore, we use the
subadditive ergodic theorem to obtain a law of large numbers for
$T_{0,n}$, that is, $T_{0,n}/n$ tends to a constant $\nu>0$ almost surely as
$n\rightarrow\infty$.

\textbf{Step 3.} \emph{Convergence of $L(p)\mu(p,u)$ to $\nu$.}  In
Step 2 we have showed that for any fixed $n$, the passage time
$T_{0,n}$ is well-approximated by $T^p_{0,n}$ for all $p<p_c$
sufficiently close to $p_c$.  This is a uniformly ``local approximation'' since $n$ is fixed; we need a uniformly ``global
approximation'': For all $p<p_c$ sufficiently close to $p_c$, the passage
time $T_{0,n}$ is well-approximated by $T^p_{0,n}$ for all large
$n$.  Indeed, we need to show that $L(p)\mu(p)$ tends to $\nu$ as
$p\uparrow p_c$.  We shall implement a standard renormalization argument
to show that $\nu$ is an upper bound for the upper limit of
$L(p)\mu(p)$ (see Figure \ref{renorm}):  The results obtained in Step 2 allow us to get that, with high probability, the point-to-point passage
times for paths constrained in suitably chosen blocks are well-approximated by the corresponding times for paths not constrained, leading to a $K$-dependent bond percolation.  In order to show that $\nu$ is also a lower bound for
the lower limit of $L(p)\mu(p)$, we use a ``block approach'', which
was introduced by Grimmett and Kesten in \cite{GK84} to derive exponential
large deviation bounds for passage times.  Finally, due to the fact that
the near-critical scaling limit is invariant under rotations, it is easy to
show that $L(p)\mu(p,u)\rightarrow\nu$ uniformly in $u\in\mathbb{U}$ as
$p\uparrow p_c$, which completes the proof.

\subsection{Relations to previous works}
In this subsection we wish to review some related works in the
literature. For FPP, most works have focused on the subcritical
regime, where $F(0)<p_c$.  In this case $a_{0,n}$ grows linearly in
$n$, and many results have been proved, such as shape theorems,
fluctuations of first-passage times, geometry of geodesics.  On the
other hand, a number of open problems have been proposed.  For
example, although we know the existence of the time constant, it is
an old question to find a non-trivial explicit distribution for
which we can determine the time constant.  See \cite{ADH17} for a
recent survey.

In the following, we list a few works on Bernoulli FPP (on
$\mathbb{T}$); exact asymptotics for this special model are closely related
to the scaling limits of critical and near-critical percolation.
\begin{itemize}
\item \emph{Critical Bernoulli FPP.}  Chayes et al. \cite{CCD86}
proved that $\mathbf{E}_{p_c}a_{0,n}\asymp\log n$.  Subsequently, Kesten
and Zhang \cite{KZ97} showed that $\Var a_{0,n}\asymp\log n$, and
derived a central limit theorem for $a_{0,n}$.  Using the full scaling limit of
critical percolation (i.e., the conformal loop ensemble CLE$_6$; for the general CLE$_{\kappa}$, $8/3\leq\kappa\leq 8$,
see \cite{She09,SW12})
established by Camia and Newman \cite{CN06}, we obtained a law of
large numbers for $a_{0,n}$ in \cite{Yao14}.  The idea is to
define FPP on CLE$_6$ by using ``loop chains'', similarly to the
cluster chains we used in the present paper; the application of
the subadditive ergodic theorem relies on the scaling invariance of
full-plane CLE$_6$, similarly to the case in the present paper that
the application of this theorem relies on the translation invariance
of the continuum cluster ensemble.  Except for these similarities,
the proof here is much more complicated than that in \cite{Yao14}.

In \cite{Yao18} we improved the result in \cite{Yao14}, giving
explicit limit theorem by identifying the exact values of the
constants appearing in the first-order asymptotics of
$a_{0,n},\mathbf{E}_{p_c}a_{0,n}$ and $\Var_{p_c} a_{0,n}$.  (Recently, Damron, Hanson and Lam
\cite{DHL} have extended this result to general critical FPP.)  Analogous limit theorem for critical Bernoulli FPP
starting on the boundary was established in \cite{JY19} by using
SLE$_6$.  In \cite{Yao19}, we constructed different subsequential
limits for $a_{0,n}$, relying on the large deviation estimates on
the nesting of CLE$_6$ loops.

Moreover, as discussed in Section 4 in \cite{Yao18}, it is expected
that as $t\rightarrow \infty$, the outer boundary of $B(t)$,
properly scaled, converges in distribution to a ``typical'' loop of
full-plane CLE$_6$.
\item \emph{Near-critical Bernoulli FPP: supercritical regime.}  In \cite{Yao19},
relying on the limit theorem for the critical case, we derived exact
first-order asymptotics for $T(0,\mathcal {C}_{\infty})$ together
with its expectation and variance, as $p\downarrow p_c$, where
$\mathcal {C}_{\infty}$ is the infinite cluster with 0-time sites.
\item \emph{Near-critical Bernoulli FPP: subcritical regime.} As mentioned earlier, Chayes et al. \cite{CCD86} proved that
$\mu(p)\asymp L(p)^{-1}$ as $p\uparrow p_c$.  This motivates our
present work.
\end{itemize}

Besides the works on Bernoulli FPP mentioned above, we now describe
several works which are related to the present paper in the sense
that the scaling limits are used to extract geometric information
about the original discrete models.
\begin{itemize}
\item Duminil-Copin \cite{DC13} used the scaling limit of
near-critical percolation established in \cite{GPS18} to show that
the Wulff crystal for subcritical percolation on $\mathbb{T}$
converges to a Euclidean disk as $p\uparrow p_c$.  Roughly
speaking, it was showed that the typical shape of a cluster
conditioned to be large becomes round.  Our strategy of the proof of
Theorem \ref{t2} is rather different from \cite{DC13}.
\item  For percolation on $\mathbb{T}$, we contract each blue cluster into a single vertex and define a new edge
between any pair of blue clusters if there is a yellow hexagon
touching both of them.  Then we obtain a random graph called ``cluster
graph''.  In \cite{Yao19}, we applied the limit theorem for critical
Bernoulli FPP to study the graph distance in the cluster graph at
criticality.  The discrete cluster chains we used in the present
paper are (self-avoiding) paths in the cluster graph.  In the
near-critical regime, we will prove that with high probability the
first-passage time between two large blue clusters is equal to their
graph distance in the cluster graph; the proof uses some techniques
we developed for the cluster graph at criticality in \cite{Yao19}.
\item Based on the scaling limit of
near-critical percolation in \cite{GPS18}, Garban, Pete and Schramm \cite{GPS18b} constructed
the scaling limits of the minimal spanning tree and the invasion
percolation tree.  In the proof they also used the cluster graph,
similarly to that we described above but with some differences; see
Fig. 1 in \cite{GPS18b} and the paragraph above it.  Except for this
similarity, our proof is quite different from that in \cite{GPS18b}.
\end{itemize}

\subsection{Organization of the paper}
The rest of the paper is organized as follows.  In Section \ref{pre}, we introduce the notation and basic definitions, including
correlation length, quad-crossing topology and near-critical scaling limit, and collect
some results about critical and near-critical percolation that will be used.  Section \ref{continuum} is devoted to the
construction of the continuum cluster ensemble, which is the scaling limit of the collection of blue clusters for near-critical percolation.  In Section \ref{fpp} we define FPP on continuum cluster ensembles,
and obtain a law of large numbers for the point-to-point
passage times of this continuum FPP.  We also show that some quantities of the discrete FPP approximate their corresponding quantities
of the continuum FPP as $p\uparrow p_c$.  In Section \ref{global}, based on the results for the continuum FPP together with the asymptotics of the corresponding discrete quantities, we use the renormalization method to prove our main result.

\section{Notation and preliminaries}\label{pre}
Throughout this paper, $C, K, C_j$ or $K_j$ stands for a positive
constant that may change from line to line according to the context.
$\mathbb{N}=\{1,2,\ldots\}$ denotes the set of natural integers and
$\mathbb{Z}_+=\{0\}\cup\mathbb{N}$.  We identify the plane
$\mathbb{R}^2$ with the set $\mathbb{C}$ of complex numbers in the
usual way.

For two positive functions $f$ and $g$ from a set $\mathcal {X}$ to
$(0,\infty)$, we write $f(x)\asymp g(x)$ to indicate that
$f(x)/g(x)$ is bounded away from 0 and $\infty$, uniformly in $x\in
\mathcal {X}$.

A \textbf{circuit} is a path $(v_1,\ldots,v_n)$ with $n\geq 3$, such
that $v_1$ and $v_n$ are neighbors.  Note that the bonds
$(v_1,v_2),\ldots,(v_n,v_1)$ of the circuit form a Jordan curve, and
sometimes the circuit is viewed as this curve.

For a rectangle of the form $R=[x_1,x_2]\times[y_1,y_2]$, we call a
path $(v_0,v_1,\ldots,v_k)$ of $\mathbb{T}$ a \textbf{left-right
(resp. top-bottom) crossing} of $R$, if $v_1,\ldots,v_{k-1}\in R$,
and the line segments $\overline{v_0v_1}$ and
$\overline{v_{k-1}v_k}$ intersect the left and right (resp. top and
bottom) sides of $R$, respectively.  For a left-right crossing of
$R$, if all the sites in it are blue (resp. yellow), we call it a
\textbf{blue (resp. yellow) left-right crossing} of $R$.

For $r>0$, define the box $\Lambda_r:=[-r,r]^2$.  For $0<r<R$, define
the annulus $A(r,R):=\Lambda_R\backslash\Lambda_r$.  For
$z\in\mathbb{C}$ and $\theta\in[0,2\pi]$, we set $\Lambda_r^{\theta}(z):=z+e^{i\theta}\cdot\Lambda_r$ and
$A^{\theta}(z;r,R):=z+e^{i\theta}\cdot A(r,R)$.  Write $\Lambda_r(z):=\Lambda_r^0(z)$ and
$A(z;r,R):=A^0(z;r,R)$.

The so-called arm events play a central role in studying
near-critical percolation.  We write 0 and 1 for ``blue'' and
``yellow'', respectively.  A \textbf{color sequence} $\sigma=(\sigma_1,\sigma_2,\dots,\sigma_k)$ is an
element of $\{0,1\}^k$ with $k\geq 1$, and its length $|\sigma|$ is
$k$.  For ease of notation,
we write
$(\sigma_1\sigma_2\cdots\sigma_k):=(\sigma_1,\sigma_2,\dots,\sigma_k)$.  We say that a color sequence is \textbf{polychromatic} if the
sequence contains at least one 0 and one 1.
We identify two sequences if they are the same up to a cyclic
permutation.  Suppose $\theta\in[0,2\pi]$.  For an annulus $A^{\theta}(z;r,R)$, we denote by $\mathcal
{A}_{\sigma}^{\theta}(z;r,R)$ the event that there exist $|\sigma|$ disjoint
monochromatic paths called \textbf{arms} in $A^{\theta}(z;r,R)$ connecting
the two boundary pieces of $A^{\theta}(z;r,R)$, whose colors are those
prescribed by $\sigma$, when taken in counterclockwise order.  The half-plane arm events are defined similarly:  We denote by
$\mathcal
{A}_{\sigma}^{\theta,+}(z;r,R)$ the event that $\mathcal
{A}_{\sigma}^{\theta}(z;r,R)$ occurs with the arms contained in the half-plane $z+e^{i\theta}\cdot\{u\in\mathbb{C}:\Im u\geq 0\}$.  More generally, given $\vartheta\in(0,2\pi)$, we denote by $\mathcal{A}_{\sigma}^{\theta,\vartheta}(z;r,R)$ the wedge arm event that $\mathcal
{A}_{\sigma}^{\theta}(z;r,R)$ occurs with the arms contained in the wedge $W(z;\theta,\vartheta):=z+e^{i\theta}\cdot\{u\in\mathbb{C}:\arg(u)\in[-\vartheta/2,\vartheta/2]\}$.

For ease of notation, we write $\mathcal {A}_{\sigma}(z;r,R):=\mathcal {A}_{\sigma}^{0}(z;r,R)$ when $\theta=0$, and write
$\mathcal {A}_{\sigma}(r,R):=\mathcal {A}_{\sigma}(0;r,R)$ when $z=0$.  Analogous abbreviations apply to $\mathcal
{A}_{\sigma}^{\theta,+}(z;r,R)$ and $\mathcal
{A}_{\sigma}^{\theta,\vartheta}(z;r,R)$.  For any $r\geq 1$, we let $\mathcal
{A}_{\sigma}(z;r,r)$ be the entire sample space.  Write  $\mathcal {A}_4:=\mathcal {A}_{(0101)}$ and $\mathcal
{A}_6:=\mathcal {A}_{(011011)}$.

\subsection{Correlation length}\label{correlation}
First, let us define the \textbf{correlation length} $L(p)$ that we
mainly work with in this paper: For any $p<p_c$, we set
\begin{equation*}
L(p):=\inf\left\{R\geq 1:R^2\mathbf{P}_{p_c}[\mathcal
{A}_4(1,R)]\geq\frac{1}{p_c-p}\right\},
\end{equation*}
where we let $L(p)=1/2$ if the above set is empty.  We chose to work with $L(p)$
because our proof is based on Corollary 1.7 of \cite{GPS18} (see Theorem \ref{t6} below) which gives the scaling limit of
near-critical percolation on $L(p)^{-1}\mathbb{T}$ as $p\uparrow p_c$.

Now we introduce another definition of correlation length: For each
$\epsilon\in(0,1/2)$ and $p<p_c$, let
\begin{equation*}
L_{\epsilon}(p):=\inf\{R\geq 1: \mathbf{P}_p[\mbox{there is a blue
left-right crossing of $[0,R]^2$}]\leq\epsilon\}.
\end{equation*}
The particular choice of $\epsilon$ is not important in the above
definition.  Indeed, for any $\epsilon,\epsilon'\in(0,1/2)$, we have $L_{\epsilon}(p)\asymp L_{\epsilon'}(p)$; see, e.g., Corollary 37 of \cite{Nol08}.  We refer
the interest reader to Section 7 of \cite{Nol08} for three natural
definitions of correlation length (called ``characteristic length''
there), including $L_{\epsilon}(p)$.

Kesten (see Proposition 34 of \cite{Nol08})
proved that
\begin{equation*}
(p_c-p)L_{\epsilon}(p)^2\mathbf{P}_{p_c}[\mathcal
{A}_4(1,L_{\epsilon}(p))]\asymp 1\qquad\mbox{as $p\uparrow p_c$}.
\end{equation*}
This together with the quasi-multiplicativity (\ref{e71}) and the
inequality (\ref{e9}) at $p_c$ implies that for any
$\epsilon\in(0,1/2)$,
\begin{equation*}
L(p)\asymp L_{\epsilon}(p)\qquad\mbox{as $p\uparrow p_c$.}
\end{equation*}

It is well known that $L(p)\rightarrow +\infty$ when $p\uparrow p_c$.  For convenience, we will take the convention that
$L(p_c)=L_{\epsilon}(p_c)=+\infty$.  In the following, the
expression ``for any $1\leq R\leq L(p)$'' must be interpreted as ``for any $R\geq 1$'' when $p=p_c$.  For $R\geq 1$, define
\begin{equation*}
p(R):=\inf\{p:L(p')\geq R\mbox{ for all $p'\in [p,p_c]$}\}.
\end{equation*}
Note that $p(R)\in(0,p_c)$ for all $R\geq 1$, and $p(R)\uparrow p_c$ as $R\rightarrow\infty$.

\subsection{Classical results for planar percolation}\label{classic}
We assume that the reader is familiar with the FKG inequality, the
BK (van den Berg-Kesten) inequality, Reimer's inequality
\cite{Rei00}, and the RSW (Russo-Seymour-Welsh) technology.  Here we
collect some classical results in critical and near-critical percolation which
will be used.  See, e.g., Section 2.2 in
\cite{BKN18}, Section 2.2 in \cite{BN21}, and \cite{Kes87,Nol08,SW01,Wer09}.
\begin{enumerate}[(i)]
\item \emph{RSW bounds}.  For any $k,K\geq 1$, there exists a
constant $\delta=\delta(k,K)>0$, such that for all $p\in(p(1),p_c]$ and
$1\leq n\leq KL(p)$,
\begin{equation*}
\mathbf{P}_p[\mbox{there is a blue left-right crossing of $[0,kn]\times[0,n]$}]\geq\delta.
\end{equation*}
\item \emph{A-priori bounds on 1-arm events}.  There exist constants
$\lambda_1,\lambda_1'>0$, such that for any $K\geq 1$, there are
constants $C_1(K),C_2(K)>0$, such that for all $p\in(p(1),p_c]$ and
$1\leq r<R\leq KL(p)$,
\begin{equation}\label{e78}
C_1\left(\frac{r}{R}\right)^{\lambda_1'}\leq \mathbf{P}_p[\mathcal
{A}_{(0)}(r,R)]\leq\mathbf{P}_p[\mathcal {A}_{(1)}(r,R)]\leq
C_2\left(\frac{r}{R}\right)^{\lambda_1}.
\end{equation}
\item \emph{Quasi-multiplicativity}.  For any color sequence $\sigma$ and $K\geq 1$, there
exist constants $C_1(|\sigma|,K)$, $C_2(|\sigma|,K)>0$, such that for all $p\in(p(1),p_c]$ and $1\leq r_1<r_2<r_3\leq KL(p)$,
\begin{equation}\label{e71}
C_1\mathbf{P}_p[\mathcal {A}_{\sigma}(r_1,r_2)]\mathbf{P}_p[\mathcal
{A}_{\sigma}(r_2,r_3)]\leq \mathbf{P}_p[\mathcal
{A}_{\sigma}(r_1,r_3)]\leq C_2\mathbf{P}_p[\mathcal
{A}_{\sigma}(r_1,r_2)]\mathbf{P}_p[\mathcal {A}_{\sigma}(r_2,r_3)].
\end{equation}
\item \emph{Stability for arm events near criticality}.
For any color sequence $\sigma$ and $K\geq 1$, there exist constants
$C_1(|\sigma|,K)$, $C_2(|\sigma|,K)>0$, such that for all $p\in(p(1),p_c)$
and $1\leq r<R\leq KL(p)$,
\begin{equation}\label{e7}
C_1\mathbf{P}_{p_c}[\mathcal
{A}_{\sigma}(r,R)]\leq\mathbf{P}_p[\mathcal {A}_{\sigma}(r,R)]\leq
C_2\mathbf{P}_{p_c}[\mathcal {A}_{\sigma}(r,R)].
\end{equation}
\item  \emph{Exponential decay with respect to $L(p)$}.  There exist constants $C_1,C_2>0$, such that for all
$p\in(p(1),p_c)$ and $R\geq L(p)$,
\begin{equation}\label{e40}
\mathbf{P}_p[\mbox{$\exists$ a yellow circuit surrounding 0 in
$A(R,2R)$}]\geq 1-C_1\exp\left(-C_2\frac{R}{L(p)}\right).
\end{equation}
This follows from FKG and (7.23) in \cite{Nol08}. (See also (17) of \cite{Yao19} for the corresponding inequality at $p>p_c$.)
\item \emph{Lower bound on the 4-arm exponent}. There exist constants $\lambda_4,C>0$, such that for all $p\in(p(1),p_c]$ and $1\leq r<R\leq L(p)$,
\begin{equation}\label{e9}
\mathbf{P}_p[\mathcal {A}_4(r,R)]\geq
C\left(\frac{r}{R}\right)^{2-\lambda_4}.
\end{equation}
\item  \emph{Upper bound on the 6-arm exponent}.
There is a $\lambda_6>0$, such that for any $K\geq 1$ and
polychromatic color sequence $\sigma$ with $|\sigma|=6$, there
exists a constant $C(K)>0$, such that for all $p\in(p(1),p_c]$ and
$1\leq r<R\leq KL(p)$,
\begin{equation}\label{e67}
\mathbf{P}_p[\mathcal {A}_{\sigma}(r,R)]\leq
C\left(\frac{r}{R}\right)^{2+\lambda_6}.
\end{equation}
\item \emph{Upper bounds on half-plane 2-arm and wedge 3-arm events}.
For any $K\geq 1$ and color sequences $\sigma_2,\sigma_3$ with
$|\sigma_2|=2$ and $|\sigma_3|=3$, there exists a constant $C(K)>0$,
such that for all $p\in(p(1),p_c]$, $1\leq r<R\leq KL(p)$ and
$\theta\in[0,2\pi]$,
\begin{align}
&\mathbf{P}_p\left[\mathcal {A}_{\sigma_2}^{\theta,+}(r,R)\right]\leq
C\left(\frac{r}{R}\right),\label{e79}\\
&\mathbf{P}_p\left[\mathcal {A}_{\sigma_3}^{\theta,+}(r,R)\right]\leq
C\left(\frac{r}{R}\right)^2.\label{e72}
\end{align}
See, e.g., Lemma 6.8 in \cite{Sun11} and Appendix A in \cite{LSW02} for the critical case.  This combined with the stability for arm events near criticality gives the above inequalities for the near-critical case.  Moreover, given any fixed $\vartheta\in(0,2\pi)$ and $\epsilon\in(0,1)$, there exists a constant $C(\vartheta,\epsilon)>0$ such that for all $p\in(p(1),p_c]$, $1\leq r<R\leq L(p)$ and $\theta\in[0,2\pi]$,
\begin{equation}
\mathbf{P}_p\left[\mathcal {A}_{\sigma_3}^{\theta,\vartheta}(r,R)\right]\leq
C\left(\frac{r}{R}\right)^{2\pi/\vartheta-\epsilon}.\label{e23}
\end{equation}
The conformal invariance of the scaling limit gives $\mathbf{P}_{p_c}[\mathcal {A}_{\sigma_3}^{\theta,\vartheta}(r,R)]=(\mathbf{P}_{p_c}[\mathcal {A}_{\sigma_3}^{\theta,\pi}(r,R)])^{\pi/\vartheta+o(1)}$, which combined with $(\ref{e72})$ implies the above inequality in the critical case.  Then
the stability for (wedge) arm events near criticality yields (\ref{e23}).
\item \emph{Upper bounds on arm events with monochromatic color sequence}.
For any polychromatic color sequence $\sigma$, there exist
$\epsilon,C>0$ (depending on $|\sigma|$), such that for all
$p\in(p(1),p_c]$ and $1\leq r<R\leq L(p)$,
\begin{equation}\label{e92}
\mathbf{P}_p[\mathcal {A}_{(\underbrace{1\cdots
1}_{|\sigma|})}(r,R)]\leq
C\left(\frac{r}{R}\right)^{\epsilon}\mathbf{P}_p[\mathcal
{A}_{\sigma}(r,R)].
\end{equation}
This inequality at $p_c$ (which was used in \cite{Yao19}
as Lemma 1) follows from the proof of Theorem 5 in \cite{BN11}; see
in particular Step 1 of the proof.   Combining the inequality at
$p_c$ and (\ref{e7}) yields (\ref{e92}).
\item When $p\uparrow p_c$,
\begin{equation}\label{e14}
L(p)=(p_c-p)^{-4/3+o(1)}.
\end{equation}
\end{enumerate}
\subsection{Space of quad-crossings}\label{quad}
To describe the scaling limits of planar percolation, Schramm and
Smirnov introduced the quad-crossing topology in \cite{SS11}.  Let
us briefly recall this topology in this subsection.  We shall use the
notation and definitions from \cite{GPS18}.

When taking the scaling limit of percolation on the whole plane, it
is convenient to compactify $\mathbb{C}$ into $\hat{\mathbb{C}}:=
\mathbb{C}\cup\{\infty\}$ (i.e., the Riemann sphere) as follows.
First, we replace the Euclidean metric with a distance function
$\Delta(\cdot,\cdot)$ defined on $\mathbb{C}\times\mathbb{C}$ by
\begin{equation*}
\Delta(x,y):=\inf_{\varphi}\int(1+|\varphi(s)|^2)^{-1}ds,
\end{equation*}
where the infimum is over all smooth curves $\varphi(s)$ joining $x$
with $y$, parameterized by arc length $s$.  This metric is equivalent to the Euclidean metric
in bounded regions.  Then, we add a single point $\infty$ at
infinity to get the compact space $\hat{\mathbb{C}}$ which is
isometric, via stereographic projection, to the two-dimensional
sphere.

Let $D\subset\hat{\mathbb{C}}$ be open.  A \textbf{quad} in the
domain $D$ can be considered as a homeomorphism $Q$ from $[0,1]^2$
into $D$.  The space of all quads in $D$, denoted by $\mathcal
{Q}_D$, can be equipped with the following metric:
\begin{equation*}
d(Q_1,Q_2):=\inf_{\phi}\sup_{z\in\partial[0,1]^2}|Q_1(z)-Q_2(\phi(z))|,
\end{equation*}
where the infimum is over all homeomorphisms $\phi:
[0,1]^2\rightarrow[0,1]^2$ which preserve the four corners of the
square.  A \textbf{crossing} of a quad $Q$ is a connected closed
subset of $Q([0,1]^2)$ that intersects both $Q(\{0\}\times[0,1])$
and $Q(\{1\}\times[0,1])$.  From the point of view of crossings,
there is a natural partial order on $\mathcal {Q}_D$: We write
$Q_1\leq Q_2$ if any crossing of $Q_2$ contains a crossing of $Q_1$.
Furthermore, we write $Q_1<Q_2$ if there are open neighborhoods
$\mathcal {N}_i$ of $Q_i$ for $i\in\{1,2\}$, such that $N_1\leq N_2$
holds for any $N_i\in\mathcal {N}_i$.  We say that a subset
$\mathcal {S}\subset\mathcal {Q}_D$ is \textbf{hereditary} if,
whenever $Q\in \mathcal {S}$ and $Q'\in\mathcal {Q}_D$ satisfies
$Q'<Q$, we have $Q'\in \mathcal {S}$.  The collection of all closed
hereditary subsets of $\mathcal {Q}_D$ will be denoted by
$\mathscr{H}_D$.

By introducing a natural topology, $\mathscr{H}_D$ can be made into
a compact metric space.  Indeed, let
\begin{align*}
&\boxminus_{Q}:=\{\mathcal {S}\in\mathscr{H}_D: Q\in\mathcal {S}\}\qquad\mbox{for any }Q\in\mathcal {Q}_D,\\
&\boxdot_{\mathcal {U}}:=\{\mathcal {S}\in\mathscr{H}_D: \mathcal
{U}\cap\mathcal {S}=\emptyset\}\qquad\mbox{for any open subset
}\mathcal {U}\subset\mathcal {Q}_D.
\end{align*}
Then we endow $\mathscr{H}_D$ with the topology $\mathscr{T}_D$
which is the minimal topology that contains every $\boxminus_{Q}^c$
and $\boxdot_{\mathcal {U}}^c$ as open sets.  It was proved in
\cite{SS11}, Theorem 1.13, that for any nonempty open subset
$D\subset\hat{\mathbb{C}}$, the topological space
$(\mathscr{H}_D,\mathscr{T}_D)$ is a compact metrizable Hausdorff
space.  In particular, $(\mathscr{H}_D,\mathscr{T}_D)$ is a Polish
space.  Furthermore, for any dense $\mathcal {Q}_0\subset\mathcal
{Q}_D$, the events $\{\boxminus_{Q}:Q\in\mathcal {Q}_0\}$ generate
the Borel $\sigma$-field of $\mathscr{H}_D$.  An arbitrary metric
generating the topology $\mathscr{T}_D$ will be denoted by
$d_{\mathscr{H}}$. The above compactness property implies that (see
Corollary 1.15 of \cite{SS11}), the space of Borel probability
measures of $(\mathscr{H}_D,\mathscr{T}_D)$, equipped with the weak*
topology is a compact metrizable Hausdorff space.

When $D=\hat{\mathbb{C}}$, we write
$\mathscr{H}:=\mathscr{H}_{\hat{\mathbb{C}}}$ and $\mathscr{T}:=\mathscr{T}_{\hat{\mathbb{C}}}$.  With a slight abuse of notation, when we
refer to $Q$ as a subset of $\hat{\mathbb{C}}$ in the following, we consider its
range $Q([0,1]^2)\subset\hat{\mathbb{C}}$.

Note that any discrete percolation configuration
$\omega_p^{\eta}:=\eta\omega_p$ on $\eta\mathbb{T}$, considered as a
union of blue hexagons in the plane, naturally induce an element in
$\mathscr{H}$: the set of all quads for which $\omega_p^{\eta}$
contains a crossing.  By a slight abuse of notation, we will still
denote by $\omega_p^{\eta}$ the point in $\mathscr{H}$ corresponding
to the percolation configuration $\omega_p^{\eta}$.  It follows that
$\omega_p^{\eta}$ induces a probability measure on $\mathscr{H}$,
denoted by $\mathbf{P}_p^{\eta}$.

\subsection{Near-critical scaling limit}\label{s2}
As in \cite{GPS18}, we define the following near-critical parameter
scale: For $\eta>0$ and $\lambda\in\mathbb{R}$, we set
\begin{equation*}
p_{\lambda}(\eta):=p_c+\lambda\frac{\eta^2}{\alpha_4^{\eta}(\eta,1)},
\end{equation*}
where $\alpha_4^{\eta}(r,R)$ stands for the probability of the
alternating 4-arm event in $A(r,R)$ for critical site percolation on
$\eta\mathbb{T}$.

Recall that for each $p\in [0,1]$, $\omega_p$ stands for Bernoulli site
percolation on $\mathbb{T}$ with intensity $p$, and $\mathbf{P}_p$
stands for the law of $\omega_p$.  The following are two natural ways to define near-critical
percolation on rescaled lattices:
\begin{itemize}
\item For $\eta>0$ and $\lambda\in\mathbb{R}$, let $\omega_{p_{\lambda}(\eta)}^{\eta}$ denote the percolation
configuration on $\eta\mathbb{T}$ with intensity
$p_{\lambda}(\eta)$, and let $\mathbf{P}_{p_{\lambda}(\eta)}^{\eta}$
denote the law of $\omega_{p_{\lambda}(\eta)}^{\eta}$.
\item For $p<p_c$, let $\omega_p^{L(p)^{-1}}=L(p)^{-1}\omega_p$ denote the percolation
configuration on $L(p)^{-1}\mathbb{T}$ with intensity $p$, and let
$\mathbf{P}_p^{L(p)^{-1}}$ denote the law of $\omega_p^{L(p)^{-1}}$. (Such near-critical
percolation for $p>p_c$ can be defined analogously.)
\end{itemize}
Note that for $\mathbf{P}_{p_{\lambda}(\eta)}^{\eta}$, the intensity
$p_{\lambda}(\eta)$ is a function of the mesh size $\eta$, while for
$\mathbf{P}_p^{L(p)^{-1}}$, the mesh size $L(p)^{-1}$ is a function
of the intensity $p$.  As discussed in Section \ref{quad}, we also
view $\omega_{p_{\lambda}(\eta)}^{\eta}$ (resp.
$\omega_p^{L(p)^{-1}}$) as an element in $\mathscr{H}$, and view
$\mathbf{P}_{p_{\lambda}(\eta)}^{\eta}$ (resp.
$\mathbf{P}_p^{L(p)^{-1}}$) as the probability measure on
$\mathscr{H}$ induced by $\omega_{p_{\lambda}(\eta)}^{\eta}$ (resp.
$\omega_p^{L(p)^{-1}}$).

The following theorem states that for fixed $\lambda$, the
near-critical percolation $\omega_{p_{\lambda}(\eta)}^{\eta}$ has a
scaling limit as $\eta\rightarrow 0$.
\begin{theorem}[Theorem 1.4 and Corollary 10.5 in \cite{GPS18}]\label{t4}
Fix $\lambda\in \mathbb{R}$.  As $\eta\rightarrow 0$, the
near-critical percolation $\omega_{p_{\lambda}(\eta)}^{\eta}$
converges in law in $(\mathscr{H},d_{\mathscr{H}})$ to a limiting
random percolation configuration, denoted by $\omega^{\infty}(\lambda)$.
Moreover, the law of $\omega^{\infty}(\lambda)$, denoted by
$\mathbf{P}^{\infty,\lambda}$, is invariant under translations and
rotations.
\end{theorem}
Note that our definition of $\omega_{p_{\lambda}(\eta)}^{\eta}$ is
slightly different from $\omega_{\eta}^{nc}(\lambda)$ in Theorem 1.4 of \cite{GPS18} (see Section 1.2 in
\cite{GPS18} for the definition), but this makes no essential difference; the scaling
limit $\omega^{\infty}(\lambda)$ in Theorem \ref{t4} is denoted by
$\omega_{\infty}^{nc}(2\lambda)$ in \cite{GPS18}.  We want to
mention that in \cite{AS17,BKN18,DC13}, the authors also used the scaling limit result for
$\omega_{p_{\lambda}(\eta)}^{\eta}$.

Similarly to Theorem \ref{t4} on
$\omega_{p_{\lambda}(\eta)}^{\eta}$, the following theorem states
that the near-critical percolation $\omega_p^{L(p)^{-1}}$ also has a
scaling limit as $p\uparrow p_c$, which is a key input for the proof
of our main result.
\begin{theorem}[Corollary 1.7 of \cite{GPS18}]\label{t6}
As $p\uparrow p_c$, the near-critical percolation
$\omega_p^{L(p)^{-1}}$ converges in law in
$(\mathscr{H},d_{\mathscr{H}})$ to $\omega^{\infty}(-1)$.
\end{theorem}

Theorems \ref{t4} and \ref{t6} state that the near-critical
percolation measures converge.  Moreover, the convergence of
quad-crossing probabilities also holds:
\begin{lemma}\label{l15}
Fix $\lambda\in\mathbb{R}$.  Let $D\subset \mathbb{C}$ be a bounded
domain.  Let $Q\in\mathcal {Q}_{D}$.  We have
\begin{equation}\label{e80}
\lim_{\eta\rightarrow
0}\mathbf{P}_{p_{\lambda}(\eta)}^{\eta}[\boxminus_{Q}]=\mathbf{P}^{\infty,\lambda}[\boxminus_{Q}].\qquad\mbox{((9.2)
in \cite{GPS18})}
\end{equation}
Moreover,
\begin{equation}\label{e81}
\lim_{p\uparrow
p_c}\mathbf{P}_p^{L(p)^{-1}}[\boxminus_{Q}]=\mathbf{P}^{\infty,-1}[\boxminus_{Q}].
\end{equation}
\end{lemma}
The proof of (\ref{e80}) in \cite{GPS18} uses Theorem \ref{t4}
and the proof of Corollary 5.2 in \cite{SS11}, relying on the RSW
estimates for near-critical percolation; the same proof also works for
(\ref{e81}) by using Theorem \ref{t6} and the proof of Corollary 5.2
in \cite{SS11}.

\subsection{Proof of Corollary \ref{cor1}}\label{coro}
In this subsection, we describe how to derive Corollary \ref{cor1} from Theorem \ref{t2}.

For $R>0$, denote by $\boxminus_R$ the left-right crossing event in the quad $[0,R]^2$.  As in the discrete model,
we define a notion of correlation length for $\omega^{\infty}(\lambda)$: Given $\epsilon\in(0,1/2)$ and $\lambda<0$, define
\begin{equation*}
L_{\epsilon}^{\infty}(\lambda):=\inf\{R>0:
\mathbf{P}^{\infty,\lambda}[\boxminus_R]\leq\epsilon\}.
\end{equation*}

For $\lambda\in\mathbb{R}$ and $R>0$,  by using (\ref{e80}) we
can define
\begin{equation*}
f(\lambda,R):=\lim_{\eta\rightarrow
0}\mathbf{P}_{p_{\lambda}(\eta)}^{\eta}[\boxminus_R]=\mathbf{P}^{\infty,\lambda}[\boxminus_R].
\end{equation*}
By (25) in \cite{AS17} and the argument below it, we
know that for fixed $R>0$, $f(\lambda,R)$ is absolutely continuous and strictly
increasing in $\lambda$, and it satisfies $f(\lambda,R)\in(0,1)$ and
\begin{equation*}
\lim_{\lambda\rightarrow-\infty}f(\lambda,R)=0\quad\mbox{and}\quad\lim_{\lambda\rightarrow\infty}f(\lambda,R)=1.
\end{equation*}
Furthermore, it is clear that $f(0,R)=1/2$ for all $R>0$, by the
well-known duality $f(-\lambda,R)=1-f(\lambda,R)$.  By Corollary
10.5 of \cite{GPS18}, for any scaling parameter $\rho>0$,
$\rho\cdot\omega^{\infty}(\lambda)$ has the same law as
$\omega^{\infty}(\rho^{-3/4}\lambda)$.  Therefore,
\begin{equation*}
f(\rho\lambda,R)=f(\lambda,\rho^{4/3}R).
\end{equation*}
(See (26) in \cite{AS17} for general quad.)  The above argument implies that $f(\lambda,R)$ is absolutely
continuous and strictly decreasing in $R$, and for any fixed
$\lambda<0$,
\begin{equation*}
\lim_{R\rightarrow
\infty}f(\lambda,R)=0\quad\mbox{and}\quad\lim_{R\rightarrow
0}f(\lambda,R)=1/2.
\end{equation*}
This implies that for any fixed $\lambda<0$ and $\epsilon\in(0,1/2)$,
\begin{equation}\label{e82}
L_{\epsilon}^{\infty}(\lambda)=\mbox{the unique $R$ such that
$f(\lambda,R)=\epsilon$};
\end{equation}
furthermore, $f(\lambda,R)>\epsilon$ when
$R<L_{\epsilon}^{\infty}(\lambda)$ and $f(\lambda,R)<\epsilon$ when
$R>L_{\epsilon}^{\infty}(\lambda)$.

It is well known that $L(p)\asymp L_{\epsilon}(p)$ as $p\uparrow
p_c$ (see Section \ref{correlation}).  The following is a refinement
of this result.
\begin{lemma}\label{l11}
For any fixed $\epsilon\in(0,1/2)$, we have
\begin{equation*}
\lim_{p\uparrow p_c}\frac{L_{\epsilon}(p)}{L(p)}=
L_{\epsilon}^{\infty}(-1).
\end{equation*}
\end{lemma}
\begin{proof}
By (\ref{e81}), for any $R>0$, we have
\begin{equation*}
\lim_{p\uparrow p_c}\mathbf{P}_p^{L(p)^{-1}}[\boxminus_R]=f(-1,R).
\end{equation*}
Then using (\ref{e82}) and the statement below it, we have
\begin{equation*}
\lim_{p\uparrow
p_c}\mathbf{P}_p^{L(p)^{-1}}[\boxminus_{L_{\epsilon}^{\infty}(-1)}]=f(-1,L_{\epsilon}^{\infty}(-1))=\epsilon,
\end{equation*}
and furthermore,
\begin{align*}
&\lim_{p\uparrow p_c}\mathbf{P}_p^{L(p)^{-1}}[\boxminus_R]=f(-1,R)>\epsilon\mbox{ when }R<L_{\epsilon}^{\infty}(-1),\\
&\lim_{p\uparrow
p_c}\mathbf{P}_p^{L(p)^{-1}}[\boxminus_R]=f(-1,R)<\epsilon\mbox{
when }R>L_{\epsilon}^{\infty}(-1).
\end{align*}
By the definition of $L_{\epsilon}(p)$ we have
\begin{equation*}
\frac{L_{\epsilon}(p)}{L(p)}=\inf\left\{R\geq L(p)^{-1}: \mathbf{P}_p^{L(p)^{-1}}[\boxminus_R]\leq\epsilon\right\}.
\end{equation*}
Since $L(p)\rightarrow\infty$ as $p\uparrow
p_c$, we conclude by the above argument that $\lim_{p\uparrow p_c}L_{\epsilon}(p)/L(p)=L_{\epsilon}^{\infty}(-1)$.
\end{proof}
\begin{proof}[Proof of Corollary \ref{cor1}]
Theorem \ref{t2} combined with Lemma \ref{l11} yields Corollary
\ref{cor1}.
\end{proof}

\subsection{Basic properties of Bernoulli FPP}
We say that a finite set $D\subset V(\mathbb{T})$ is simply
connected if the union of the hexagons $H_v,v\in D$, is simply
connected.  For a simply connected set $D$ of sites, we denote by
$\partial^- D$ its inner site boundary, that is, the set of sites of
$D$ that are adjacent to some site of $V(\mathbb{T})\backslash D$.
We call a simply connected subset $D$ of $\mathbb{T}$ a
discrete Jordan set if $\partial^- D$ is a circuit.  A
\textbf{discrete quad} is a discrete Jordan set $D$ together with
four distinct sites $v_1,v_2,v_3,v_4$ of $\partial^- D$, appearing
in this order as $\partial^- D$ is traversed counterclockwise. Given
a discrete quad $(D;v_1,v_2,v_3,v_4)$, we define the \textbf{arc}
$(v_kv_{k+1})$ to be the path from $v_k$ to $v_{k+1}$ (with
$v_5=v_1$) in $\partial^- D$ as $\partial^- D$ is traversed
counterclockwise.

For two disjoint finite sets $S,S'\subset V(\mathbb{T})$, we
say that a path $\gamma$ in $V(\mathbb{T})\backslash(S\cup S')$ separates $S$ from $S'$ if
any path from $S$ to $S'$ must intersect $\gamma$.  For $V_0\subset V(\mathbb{T})$ and $S,S'\subset V_0$, define
\begin{equation*}
T(S,S')(V_0):=\inf\{T(\gamma):\gamma\mbox{ is a path from a
site in $S$ to a site in $S'$ and
$\gamma\subset V_0$}\}.
\end{equation*}

The following topological or combinatorial properties of passage times are very useful in studying Bernoulli FPP on $\mathbb{T}$.
\begin{proposition}\label{p4}
Consider Bernoulli FPP on $\mathbb{T}$ with parameter $p\in(0,1)$.
The following statements hold.
\begin{enumerate}[(i)]
\item Let $(D;v_1,v_2,v_3,v_4)$ be a discrete quad.  We have
\begin{align*}
T((v_1v_2),(v_3v_4))(D)=\mbox{the maximal number of disjoint yellow paths from $(v_2v_3)$ to $(v_4v_1)$ in $D$}.
\end{align*}\label{i1}
\item Let $\mathcal {C}$
and $\mathcal {C}'$ be two fixed distinct finite blue clusters.  Then almost surely
\begin{equation*}
T(\mathcal {C},\mathcal {C}')=\mbox{the maximal number of disjoint
yellow circuits separating $\mathcal {C}$ from $\mathcal {C}'$}.
\end{equation*}
This implies that, almost surely, there exist $T(\mathcal {C},\mathcal
{C}')$ disjoint yellow circuits separating $\mathcal {C}$ from
$\mathcal {C}'$, such that any geodesic from $\mathcal {C}$ to
$\mathcal {C}'$ must intersect each of these circuits in exactly one
site. \label{i2}
\end{enumerate}
\end{proposition}
\begin{proof}
(\ref{i1}) is a quad version of property (i) for the annulus
passage times in Proposition 2 of \cite{Yao19}, and its proof is essentially the same as the proof of
that property (see, e.g., the first part of the proof of Proposition 2.4 in \cite{Yao18}), which is omitted here.

(\ref{i2}) is a slightly general version of (i) and (ii) in
Proposition 2 of \cite{Yao19}; their proofs are the same when $\mathcal {C}$ surrounds
$\mathcal {C}'$ or vice versa.  We now prove the
case where $\mathcal {C}$ does not surround $\mathcal {C}'$ and vice
versa.  Let $N$ denote the maximal number of disjoint yellow
circuits separating $\mathcal {C}$ from $\mathcal {C}'$.  The inequality
$T(\mathcal {C},\mathcal {C}')\geq N$ is trivial, since if there are $N$
disjoint yellow circuits separating $\mathcal {C}$ from $\mathcal
{C}'$ then any path connecting $\mathcal {C}$ and $\mathcal {C}'$
must intersect each of these circuits.

It remains to show that $T(\mathcal {C},\mathcal {C}')\leq N$ almost
surely.  Suppose that there exist almost surely $N$ disjoint yellow circuits
separating $\mathcal {C}$ from $\mathcal {C}'$, with $N_1$ of them
surrounding $\mathcal {C}$ and $N_2=N-N_1$ of them surrounding $\mathcal
{C}'$.  Write $\mathcal {C}_0:=\mathcal {C}$ and $\mathcal
{C}_0':=\mathcal {C}'$.  For $1\leq k\leq N_1$, let $\mathcal {C}_k$
denote the innermost yellow circuit surrounding $\mathcal
{C}_{k-1}$; for $1\leq k\leq N_2$, let $\mathcal {C}_k'$ denote the
innermost yellow circuit surrounding $\mathcal {C}_{k-1}'$.  It is
clear that $\mathcal {C}_k$, $1\leq k\leq N_1$ and $\mathcal
{C}_k'$, $1\leq k\leq N_2$ are $N$ disjoint yellow circuits
separating $\mathcal {C}$ from $\mathcal {C}'$.  We claim that almost surely either there is
a blue path $\gamma$ such that its starting site has a
neighbor $v_{N_1}\in\mathcal {C}_{N_1}$ and its ending site has a
neighbor $v'_{N_2}\in\mathcal {C}_{N_2}'$, or there is a $v_{N_1}\in\mathcal
{C}_{N_1}$ and a $v'_{N_2}\in\mathcal {C}_{N_2}'$ such that they are
neighbors; in the latter case we let $\gamma=\emptyset$.  Suppose the claim does not hold.
Then it is easy to see that there is a.s. a yellow circuit separating
$\mathcal {C}_{N_1}$ from $\mathcal {C}_{N_2}'$, and thus there are
a.s. $N+1$ disjoint yellow circuits separating $\mathcal {C}$ from
$\mathcal {C}'$, a contradiction.  Next, (since $\mathcal {C}_{N_1}$
is the innermost circuit surrounding $\mathcal {C}_{N_1-1}$) we take
a path $\gamma_{N_1}$ from $v_{N_1}$ to a site which has a neighbor
$v_{N_1-1}\in\mathcal {C}_{N_1-1}$, with all sites of
$\gamma_{N_1}\backslash\{v_{N_1}\}$ being blue, and then take a path
$\gamma_{N_1-1}$ from $v_{N_1-1}$ to a site which has a neighbor
$v_{N_1-2}\in\mathcal {C}_{N_1-2}$, with all sites of
$\gamma_{N_1-1}\backslash\{v_{N_1-2}\}$ being blue, and so on.  The
process stops after $N_1$ steps, and $\gamma_{1}$ is a path from
$v_1\in\mathcal {C}_1$ to a site in $\mathcal {C}$, with all sites
of $\gamma_1\backslash\{v_1\}$ being blue.  Similarly we take paths
$\gamma_{N_2}',\gamma_{N_2-1}',\ldots,\gamma_1'$.  Then we
concatenate the paths
$\gamma_1,\ldots,\gamma_{N_1},\gamma,\gamma_{N_2}',\ldots,\gamma_1'$
to obtain a path $\Gamma$ from $\mathcal {C}$ to $\mathcal {C}'$,
such that $T(\Gamma)=N$, which implies $T(\mathcal {C},\mathcal
{C}')\leq N$ almost surely.
\end{proof}

In the rest of this paper, we mainly work on $\mathbf{P}_p^{L(p)^{-1}}$.
For notational convenience, we write ${\eta=\eta(p):=L(p)^{-1}}$.

Suppose that $w,h>0$.  For the rectangle
$[0,w]\times[0,h]$, define the \textbf{line-to-line passage time}
$l_{w,h}^p$ by
\begin{equation*}
l_{w,h}^p=l_{w,h}^p(\omega_p^{\eta}):=\inf\{T(\gamma):\mbox{$\gamma$ is a left-right crossing
of $[0,w]\times [0,h]$ in $\eta\mathbb{T}$}\}.
\end{equation*}
More generally, the line-to-line passage time $l_{w,h}^{p,\theta}(z)$ corresponding to the
rectangle $z+e^{i\theta}([0,w]\times [0,h])$, with $z\in
\mathbb{C}$ and $\theta\in[0,2\pi]$, is defined
similarly as above.  Note that $l_{w,h}^{p,0}(0)=l_{w,h}^p$.

\begin{lemma}\label{l24}
There exist constants $C_1,C_2>0$ and $K\geq 2$, such that for all
$p\in(p(10),p_c)$, $2\eta\leq h\leq 1$,
${w\geq h}$, ${x\geq Kw/h}$, ${z\in\mathbb{C}}$ and
$\theta\in[0,2\pi]$,
\begin{equation*}
\mathbf{P}_p^{\eta}\left[l_{w,h}^{p,\theta}(z)\geq x\right]\leq
C_1\exp(-C_2x).
\end{equation*}
\end{lemma}
\begin{proof}
For simplicity, we shall show the lemma in the case $\theta=0,z=0$
and $h=1$; the proof extends immediately to the general case.

Suppose $w\geq 1$ and $p\in(p(10),p_c)$.  Let
$D_w^{\eta}$ be the largest discrete Jordan set of
$\eta\mathbb{T}$ in $[0,w]\times[0,1]$.  Let $v_1,v_2,v_3,v_4\in
\partial^- D_w^{\eta}$ be four sites closest to the four points
$(0,1),(0,0),(w,0),(w,1)$, respectively.  Then we get a discrete
quad $(D_w^{\eta};v_1,v_2,v_3,v_4)$.  It is easy to see that
\begin{equation*}
l_{w,1}^p\leq T((v_1v_2),(v_3v_4))(D_w^{\eta})+4.
\end{equation*}
Using (\ref{i1}) of Proposition \ref{p4} and the above inequality,
we have
\begin{align*}
l_{w,1}^p-4&\leq\mbox{the maximal number of disjoint yellow paths
from $(v_2v_3)$ to $(v_4v_1)$ in $D_w^{\eta}$}\\
&\leq\mbox{the maximal number of disjoint yellow top-bottom
crossings of $[0,w]\times[1/4,3/4]$}\\
&:=\widetilde{l_{w,1}^p}.
\end{align*}
The argument in the following is analogous to Step 3 of the proof of
Theorem 5 in \cite{BN11}.  Observe that any yellow top-bottom
crossing of $[0,w]\times[1/4,3/4]$ must either cross a rectangle in
$\{[j/2,j/2+1]\times[1/4,3/4]: j=0,1,\ldots,\lfloor 2w\rfloor-1\}$
from top to bottom, or a square in $\{[j/2,j/2+1/2]\times[1/4,3/4]:
j=0,1,\ldots,\lfloor 2w\rfloor-1\}$ from left to right.  Therefore,
if $\widetilde{l_{w,1}^p}\geq x$, then there are $\lfloor x\rfloor$
disjoint yellow top-bottom crossings of $[0,w]\times[1/4,3/4]$, each
one crossing a rectangle from the two families of rectangles above.
By RSW, there exists a universal constant $\delta\in(0,1)$, such
that for any rectangle from the two families of rectangles above,
the probability of the event that this rectangle has a yellow
top-bottom or left-right crossing is bounded above by $\delta$. This
combined with the BK inequality implies that
\begin{equation*}
\mathbf{P}_p^{\eta}\left[\widetilde{l_{w,1}^p}\geq
x\right]\leq\sum_{n_1+n_2+\cdots+n_{2\lfloor 2w\rfloor}=\lfloor
x\rfloor}\delta^{\lfloor x\rfloor}=\binom{\lfloor x\rfloor+2\lfloor 2w\rfloor}{2\lfloor
2w\rfloor}\delta^{\lfloor
x\rfloor},
\end{equation*}
where $n_1,n_2,\ldots,n_{2\lfloor 2w\rfloor}\in\mathbb{Z}_+$.  Then
using Stirling's formula we obtain that, there exist constants
$C_3,C_4>0$ and ${K_1\geq 2}$, such that for all $w\geq 1,x\geq K_1w$ and
$p\in(p(10),p_c)$,
\begin{equation*}
\mathbf{P}_p^{\eta}\left[\widetilde{l_{w,1}^p}\geq x\right]\leq
C_3\exp(-C_4x),
\end{equation*}
which implies the desired result for $l_{w,1}^p$ since $l_{w,1}^p\leq
\widetilde{l_{w,1}^p}+4$.
\end{proof}

We will also consider annulus-crossing times.  A path
$(v_0,v_1,\ldots,v_k)$ of $\eta\mathbb{T}$ is called a crossing of
$A(z;r_1,r_2)$ if $v_1,\ldots,v_{k-1}\in
A(z;r_1,r_2)$ and the line segments $\overline{v_0v_1}$ and
$\overline{v_{k-1}v_k}$ intersect the two boundary pieces of
$A(z;r_1,r_2)$, respectively.  Let
\begin{equation*}
X^p(z;r_1,r_2):=\inf\{T(\gamma):\gamma\mbox{ is a crossing of $A(z;r_1,r_2)$ in $\eta\mathbb{T}$}\}.
\end{equation*}
The next lemma is an annulus analog of Lemma \ref{l24}, and is a near-critical analog of Corollary 2.3 in
\cite{Yao14} for the critical case.
\begin{lemma}\label{l10}
There exist constants $K_1,K_2>0$ and $K\geq 2$, such that for all
$p\in(p(10),p_c)$, ${\eta\leq r_1\leq
r_2/2\leq 2}$, ${x\geq K\log_2(r_2/r_1)}$ and $z\in\mathbb{C}$,
\begin{equation*}
\mathbf{P}_p^{\eta}[X^p(z;r_1,r_2)\geq x]\leq
K_1\exp(-K_2x).
\end{equation*}
\end{lemma}
\begin{proof}
The proof is very similar to the proof of Lemma
\ref{l24}, so we only give a sketch.  By using (i) of
Proposition 2 in \cite{Yao19}, proving the lemma boils down to estimate the
maximal number of disjoint yellow circuits surrounding $z$ in $A(z;r_1,r_2)$.  Then
one can use the argument in Step 3 of the proof of
Theorem 5 in \cite{BN11} to obtain it, based on
BK inequality and near-critical RSW.
\end{proof}

The square lattice has site set $\mathbb{Z}^2$ and bond set
$E(\mathbb{Z}^2)$ obtained by connecting all pairs
$u,v\in\mathbb{Z}^2$ for which ${|u-v|=1}$.  In a standard abuse
of notation, we write $\mathbb{Z}^2$ to denote this graph.  Let
$\mathbf{P}_{\mathbb{Z}^2,p}^{site}$ (resp.
$\mathbf{P}_{\mathbb{Z}^2,p}^{bond}$) denote the Bernoulli site
(resp. bond) percolation measure on $\mathbb{Z}^2$ with parameter
$p$, defined similarly as the measure $\mathbf{P}_p$ on
$\mathbb{T}$.  Here we adapt the usual setting for Bernoulli FPP: Let
each site (resp. bond) of $\mathbb{Z}^2$ take the value 0 (open) with probability $p$,
and take the value 1 (closed) with probability $1-p$.

In the following theorem, we will compare locally dependent fields with
Bernoulli percolation measures.  A family $Y=\{Y_v: v\in
\mathbb{Z}^2\}$ of random variables is called $k$-dependent if any
two sub-families $\{Y_v: v\in A\}$ and $\{Y_v: v\in A'\}$ are
independent whenever the graph distance between $v$ and $v'$ is
larger than $k$ for all $v\in A$ and $v'\in A'$.  We denote by $Z^p=\{Z^p_v:v\in \mathbb{Z}^2\}$ an i.i.d. family of Bernoulli random variables which has the law $\mathbf{P}_{\mathbb{Z}^2,p}^{site}$.
\begin{theorem}[Theorem 7.65 of \cite{Gri99}]\label{t8}
Let $k\in\mathbb{N}$.  There exists a nondecreasing function $\pi:
[0,1]\rightarrow [0,1]$ satisfying $\pi(\delta)\rightarrow 1$ as
$\delta\rightarrow 1$ such that the following assertion holds.  If
$Y=\{Y_v:v\in \mathbb{Z}^2\}$ is a $k$-dependent family of random
variables satisfying
\begin{equation*}
\mathbf{P}[Y_v=1]\geq\delta\quad\mbox{for all $v\in\mathbb{Z}^2$},
\end{equation*}
then we have the stochastic domination:
$Y\geq_{st}Z^{1-\pi(\delta)}$.
\end{theorem}

The following proposition is a site version of Theorem 2.3 in
\cite{Kes03} in the case $d=2$.  (Note that Theorem 2.3 in \cite{Kes03} is a special
case of Proposition 5.8 in \cite{Kes86}.)
\begin{proposition}[Theorem 2.3 of \cite{Kes03}]\label{p1}
If $p<p_c^{site}(\mathbb{Z}^2)$,  then there are constants
$\epsilon,C_1,C_2>0$ depending on $p$, such that for all $n\in\mathbb{N}$,
\begin{equation*}
\mathbf{P}_{\mathbb{Z}^2,p}^{site}\left[
\begin{aligned}
&\mbox{there exists a path starting from 0 with at least }\\
&\mbox{$n$ sites and fewer than $\epsilon n$ closed sites}
\end{aligned}
\right]\leq C_1\exp(-C_2n).
\end{equation*}
\end{proposition}

Let $\omega_{\mathbb{Z}^2}^{bond}$ denote the bond percolation
configuration on $\mathbb{Z}^2$.  The \textbf{chemical distance}
$D(u,v)(\omega_{\mathbb{Z}^2}^{bond})$ between two sites $u$ and $v$
in $\mathbb{Z}^2$ is defined by
\begin{equation*}
D(u,v)(\omega_{\mathbb{Z}^2}^{bond}):=\inf\{\mbox{the
number of bonds of $\gamma$}:\gamma\mbox{ is a closed path
connecting $u$ and $v$}\}.
\end{equation*}
If $u$ and $v$ are not in the same closed cluster, we set $D(u,v)=\infty$.  The following proposition is a
corollary of Theorem 1.4 in \cite{GM07}.
\begin{proposition}\label{p7}
For any $\epsilon>0$, there exists $p_0=p_0(\epsilon)\in (0,1/2)$,
such that for all $p\in[0,p_0]$ and all large $n$ (depending on
$\epsilon$),
\begin{equation*}
\mathbf{P}_{\mathbb{Z}^2,p}^{bond}[D(0,n)\leq (1+\epsilon)n]\geq
1-\epsilon.
\end{equation*}
\end{proposition}
\begin{proof}
Given two sites $u,v\in\mathbb{Z}^2$, we denote by
$u\leftrightarrow_c v$ the event that there is a closed path
connecting $u$ and $v$, and by $v\leftrightarrow_c \infty$ the event
that $v$ is in an infinite closed cluster.  Let
$\theta(p):=\mathbf{P}_{\mathbb{Z}^2,p}^{bond}[0\leftrightarrow_c
\infty]$.

It is well known that there is a.s. a unique infinite closed cluster
when $0\leq p<1/2=p_c^{bond}(\mathbb{Z}^2)$ and there is a.s. no
infinite closed cluster when $1/2\leq p\leq 1$; see e.g.
\cite{Gri99}.  From this and the FKG inequality, we obtain that for
all $n\in\mathbb{N}$,
\begin{equation}\label{e29}
\mathbf{P}_{\mathbb{Z}^2,p}^{bond}[0\leftrightarrow_c
n]\geq\mathbf{P}_{\mathbb{Z}^2,p}^{bond}[0\leftrightarrow_c\infty,n\leftrightarrow_c\infty]\geq
\theta(p)^2.
\end{equation}
It is also well known that $\theta(p)$ is a continuous function of
$p$ on the interval $[0,1]$ (for bond percolation on
$\mathbb{Z}^2$); see e.g. \cite{Gri99}.  Moreover, it is clear that
$\theta(0)=1$.  These facts and (\ref{e29}) imply that
\begin{equation}\label{e68}
\mathbf{P}_{\mathbb{Z}^2,p}^{bond}[0\leftrightarrow_c n]\rightarrow
1\quad\mbox{uniformly in $n$ as $p\rightarrow 0$}.
\end{equation}
By Theorem 1.4 of \cite{GM07}, for each $\epsilon>0$, there exists
$p_1(\epsilon)\in(0,1/2)$, such that for every $p\in
[0,p_1(\epsilon))$,
\begin{equation}\label{e69}
\limsup_{n\rightarrow\infty}\frac{\log\mathbf{P}_{\mathbb{Z}^2,p}^{bond}[0\leftrightarrow_c
n,D(0,n)\geq (1+\epsilon)n]}{n}<0.
\end{equation}
Combining (\ref{e68}) and (\ref{e69}), we obtain the desired result.
\end{proof}

\section{Scaling limit of near-critical percolation clusters}\label{continuum}
The proof of our main result relies on the scaling limit of near-critical percolation clusters.
We start by setting notation in Section \ref{notation}.  By using the scaling limit results in \cite{GPS18}
and the approach from \cite{CCK19}, we construct the scaling limit of the collection
of blue clusters for near-critical percolation in Section \ref{whole}.
\subsection{Setting and notation}\label{notation}
For ease of notation, we write $\omega^{\infty}:=\omega^{\infty}(-1)$ and $\mathbf{P}^{\infty}:=\mathbf{P}^{\infty,-1}$, where $\omega^{\infty}(-1)$ and $\mathbf{P}^{\infty,-1}$ are defined in Theorem \ref{t4}.  Recall that $\eta=\eta(p):=L(p)^{-1}$.

Consider Bernoulli percolation on $\eta\mathbb{T}$ with $p\in(p(1),p_c)$.  We shall view each site in $\eta\mathbb{T}$ as its
corresponding (topologically closed) regular hexagons in
$\eta\mathbb{H}$.  \textbf{Clusters} are connected components of blue
or yellow hexagons.  In Section \ref{continuum} we consider
$\omega_p^{\eta}$ as a union of blue clusters, and construct the limit of large clusters in $\omega_p^{\eta}$ as
$p\uparrow p_c$.  For this purpose, we need to introduce some
notation borrowed from \cite{CCK19}, with small modifications
adapted to our setup.

For a set $A\subset\mathbb{C}$, let $\diam_{\infty}(A)$ denote the
$L^{\infty}$-diameter of $A$.  If $D$ is a simply
connected region with piecewise smooth boundary, we let
$\mathscr{C}_D^p(\delta)$ denote the collection of (blue) clusters of
$\omega_p^{\eta}$, which are contained in $D$ and have $L^{\infty}$-diameters
at least $\delta$.  That is,
\begin{equation*}
\mathscr{C}_D^p(\delta):=\left\{\mathcal
{C}\mbox{ is a cluster in }\omega_p^{\eta}: \mathcal {C}\subset
D\mbox{ and }\diam_{\infty}(\mathcal {C})\geq\delta\right\}.
\end{equation*}
Let $\mathscr{C}_D^p$ denote the collection of all clusters of
$\omega_p^{\eta}$ which are contained in $D$, and let $\mathscr{C}^p(\delta)$ denote the collection of all
clusters of $\omega_p^{\eta}$ with $L^{\infty}$-diameter at least $\delta$.  We write
$\mathscr{C}_k^p(\delta):=\mathscr{C}_{\Lambda_k}^p(\delta)$ for
short.

Let $A,B$ be two subsets of $\mathbb{C}$.  The Hausdorff
distance between $A,B$ is defined by
\begin{equation}\label{e2}
d_H(A,B):=\inf\{\epsilon>0:A+\Lambda_{\epsilon}\supset B\mbox{ and
}B+\Lambda_{\epsilon}\supset A\},
\end{equation}
where $A+\Lambda_{\epsilon}:=\{x+y\in\mathbb{C}:x\in A,y\in
\Lambda_{\epsilon}\}$.  Denote by $\mathcal {S}_R$ the complete
separable metric space of closed connected subsets of $\Lambda_R$ with the
metric (\ref{e2}).

Recall the definition of the distance function $\Delta(\cdot,\cdot)$
in Section \ref{quad}.  The distance $D_H$ between subsets of
$\hat{\mathbb{C}}$ is defined by
\begin{equation}\label{e6}
D_H(A,B):=\inf\{\epsilon>0:\forall x\in A, \exists y\in B\mbox{ such
that }\Delta(x,y)\leq\epsilon\mbox{ and vice versa}\}.
\end{equation}
Denote by $\mathcal {S}_{\infty}$ the complete separable metric
space of closed connected subsets of $\hat{\mathbb{C}}$ with the metric
(\ref{e6}).

The distance $\widehat{\dist}$ between finite collections i.e., sets of subsets of $\mathbb{C}$, denoted by $\mathscr{S},\mathscr{S}'$, is
defined by
\begin{equation}\label{e10}
\widehat{\dist}(\mathscr{S},\mathscr{S}'):=\min_{\phi}\max_{S\in\mathscr{S}}d_H(S,\phi(S)),
\end{equation}
where the infimum is taken over all bijections
$\phi:\mathscr{S}\rightarrow\mathscr{S}'$.  In case
$|\mathscr{S}|\neq|\mathscr{S}'|$ we define the distance to be
infinite.  To account for possibly infinite collections,
$\mathscr{S}$ and $\mathscr{S}'$, of subsets of
$\mathbb{C}$, we define
\begin{equation}\label{e115}
\dist(\mathscr{S},\mathscr{S}'):=\inf\{\epsilon:\forall
A\in\mathscr{S}, \exists B\in \mathscr{S}'\mbox{ such that
}d_H(A,B)\leq\epsilon\mbox{ and vice versa}\}.
\end{equation}
Similarly, for collections $\mathscr{S}$ and $\mathscr{S}'$ of subsets of
$\hat{\mathbb{C}}$, we write
\begin{equation}\label{e12}
\Dist(\mathscr{S},\mathscr{S}'):=\inf\{\epsilon:\forall
A\in\mathscr{S}, \exists B\in \mathscr{S}'\mbox{ such that
}D_H(A,B)\leq\epsilon\mbox{ and vice versa}\}.
\end{equation}
Note that the metrics $\dist$ and $\Dist$ are
equivalent on bounded regions, and convergence in $\widehat{\dist}$ implies convergence in
$\dist$ and $\Dist$.  Moreover, the space $\Omega_R$ (resp. $\Omega_{\infty}$) of closed
subsets of $\mathcal {S}_R$ (resp. $\mathcal {S}_{\infty}$) with the
metric $\dist$ (resp. $\Dist$) is also a complete separable metric space.
We denote by $\mathcal {B}_{R}$ (resp. $\mathcal {B}_{\infty}$) its Borel $\sigma$-algebra.

\subsection{Scaling limit of the collection of clusters}\label{whole}
In this subsection we show that the collection of large clusters of $\omega_{p}^{\eta}$ has a scaling limit as $p\uparrow p_c$.

Note that Theorem \ref{t6}, combined with the fact that
$(\mathscr{H},d_{\mathscr{H}})$ is a Polish metric space (see
Section \ref{quad}), implies that there is a coupling of the
measures $(\mathbf{P}_p^{\eta})$ and $\mathbf{P}^{\infty}$ on
$(\mathscr{H},d_{\mathscr{H}})$ in which
$\omega_p^{\eta}\rightarrow\omega^{\infty}$ a.s. as $p\uparrow p_c$.

The following theorem states that in a bounded region, the
collection of clusters converges to a collection of closed connected sets
(the ``continuum clusters'') as $p\uparrow p_c$.  It is an analog of
Theorems 1 and 11 in \cite{CCK19}.
\begin{theorem}\label{t1}
Let $k>\delta>0$, and let $\mathbf{P}$ be a coupling such that
$\omega_p^{\eta}\rightarrow\omega^{\infty}$ in
$(\mathscr{H},d_{\mathscr{H}})$ a.s. as $p\uparrow p_c$.  Then, as
$p\uparrow p_c$, $\mathscr{C}_{k}^p(\delta)$ converges in
$\mathbf{P}$-probability, in the metric $\widehat{\dist}$,
to a collection of closed connected sets in the interior of
$\Lambda_k$, which we denote by $\mathscr{C}_{k}(\delta)$.
Moreover, as $p\uparrow p_c$, $\mathscr{C}_{k}^p$
converges in $\mathbf{P}$-probability, in the metric $\dist$,
to a collection of closed connected sets which we denote by
$\mathscr{C}_{k}$.  Furthermore, $\mathscr{C}_{k}(\delta)$ and
$\mathscr{C}_{k}$ are measurable functions of $\omega^{\infty}$.
\end{theorem}

The following theorem extends the above theorem to the case of the
full plane and, moreover, states that the collection of full-plane
continuum clusters is invariant under rotations and translations.
It is an analog of Theorems 3 and 4 in \cite{CCK19}.
\begin{theorem}\label{t3}
Let $\mathbb{P}_k$ denote the distribution of $\mathscr{C}_{k}$. There
exists a unique probability measure $\mathbb{P}$ on $(\Omega_{\infty},\mathcal{B}_{\infty})$
that is supported on collections of bounded (in the Euclidean metric), closed, connected
subsets of $\hat{\mathbb{C}}$, which is the full plane limit of
$\mathbb{P}_k$ in the sense that,
$\mathbb{P}|_{\Lambda_k}=\mathbb{P}_k$ for each $k\in\mathbb{N}$.
Moreover, the following statements hold:
\begin{itemize}
\item Let $\mathbf{P}$ be a coupling such that
$\omega_p^{\eta}\rightarrow\omega^{\infty}$ in
$(\mathscr{H},d_{\mathscr{H}})$ a.s. as $p\uparrow p_c$.  Then, as $p\uparrow p_c$, $\mathscr{C}^p$ converges in $\mathbf{P}$-probability, in the metric $\Dist$, to a collection of bounded, closed, connected subsets of $\hat{\mathbb{C}}$ which
we denote by $\mathscr{C}$.  Moreover, $\mathscr{C}$ is a
measurable function of $\omega^{\infty}$ and the distribution of
$\mathscr{C}$ is $\mathbb{P}$.  Furthermore, for each $k\in\mathbb{N}$ and $\delta>0$,
$\mathscr{C}(\delta)|_{\Lambda_k}$ is a.s. a finite set, and $\widehat{\dist}(\mathscr{C}^{p}(\delta)|_{\Lambda_k},\mathscr{C}(\delta)|_{\Lambda_k})$,
$\dist(\mathscr{C}^p|_{\Lambda_k},\mathscr{C}|_{\Lambda_k})\rightarrow
0$ in probability as $p\uparrow p_c$.
\item For any $k,n\in\mathbb{N}$ with $n\geq 2k+1$, the configurations $\mathscr{C}|_{\Lambda_k}$ and $\mathscr{C}|_{\Lambda_k(n)}$ are independent.
\item For $\theta\in[0,2\pi]$ and $x\in \mathbb{C}$, let $f(z):=e^{i\theta}z+x$ be a map from $\mathbb{C}$ to $\mathbb{C}$.  Set
$f(\mathscr{C}):=\{f(\mathcal {C}):\mathcal {C}\in
\mathscr{C}\}$.  Then $f(\mathscr{C})$ and $\mathscr{C}$ have
the same distribution.  That is, $\mathbb{P}$ is invariant under rotations and translations.
\end{itemize}
\end{theorem}

The proofs of Theorems \ref{t1} and \ref{t3} are analogous to those
of the corresponding results in the critical case in \cite{CCK19}.
Before moving to the proofs, we need to define the arm events that
are measurable in the Borel $\sigma$-field of the quad-crossing
topology.  These events were first introduced in \cite{GPS13}.  We
shall borrow the notation and definitions from \cite{CCK19}, with a
slight modification adapted to our setting.  (Note that for our FPP
model a site is open when it takes the value 0, while in the
standard percolation model a site is open when it takes the value 1,
so the meaning of the color sequence in the present paper is
opposite to that in \cite{CCK19}.)

For $S\subset\hat{\mathbb{C}}$, let $\partial S$ and $int(S)$ denote
the boundary and interior of $S$, respectively. The arm events are defined as follows; see Figure \ref{armdefine}.

\begin{figure}
\begin{center}
\includegraphics[height=0.35\textwidth]{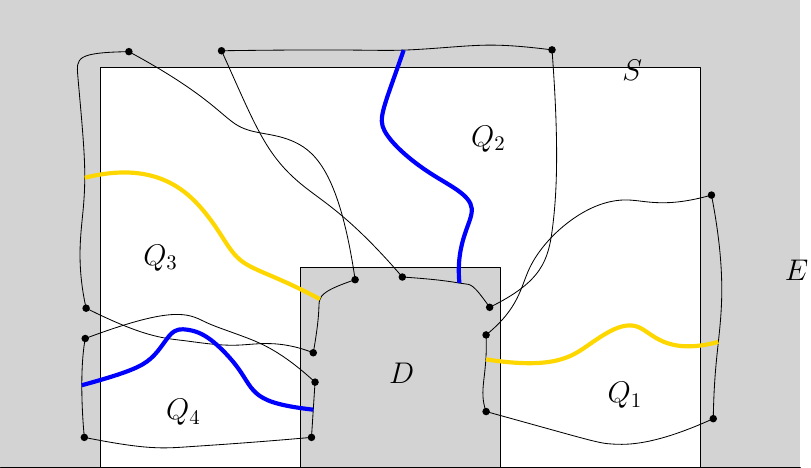}
\caption{Defining the arm event $D\xleftrightarrow{\kappa,S} E$ (in Definition \ref{def}) by using quad-crossings.  Here, $\kappa=(1010)$, $D$ is the gray square, $S$ is the upper half-plane, $E$ is the gray region outside the white region in the upper half-plane; the quads $Q_2, Q_4$ are crossed ($\kappa_2=\kappa_4=0$), shown using blue curves, while the quads $Q_1, Q_3$ are not crossed ($\kappa_1=\kappa_3=1$), shown using yellow curves (i.e., dual crossed).}\label{armdefine}
\end{center}
\end{figure}

\begin{definition}\label{def}
Let $l\in\mathbb{Z}_+$ and $\kappa\in \{0,1\}^l$.  Let $S$ be
$\hat{\mathbb{C}}$, or a simply connected subset of
$\hat{\mathbb{C}}$ with piecewise smooth boundary.  Let $D,E$ be two
disjoint simply connected subsets of $\hat{\mathbb{C}}$ with
piecewise smooth boundaries.  Let $D\xleftrightarrow{\kappa,S} E$
denote the event that there are $\delta>0$ and quads $Q_i\in
\mathcal {Q}_{int(S)}$ for $i=1,2,\ldots,l$ which satisfy the
following conditions:
\begin{enumerate}
\item $\omega\in\boxminus_{Q_i}$ for $i\in\{1,2,\ldots,l\}$ with $\kappa_i=0$ and $\omega\in\boxminus_{Q_i}^c$ for $i\in\{1,2,\ldots,l\}$ with
$\kappa_i=1$.
\item For all $i,j\in\{1,2,\ldots,l\}$ with $i\neq j$ and
$\kappa_i=\kappa_j$, the quads $Q_i$ and $Q_j$, viewed as subsets of
$\hat{\mathbb{C}}$, are disjoint, and are at distance at least
$\delta$ from each other and from the boundary of $S$.
\item $\Lambda_{\delta}+Q_i(\{0\}\times[0,1])\subset D$ and $\Lambda_{\delta}+Q_i(\{1\}\times[0,1])\subset
E$ for $i\in\{1,2,\ldots,l\}$ with $\kappa_i=0$.
\item $\Lambda_{\delta}+Q_i([0,1]\times\{0\})\subset D$ and $\Lambda_{\delta}+Q_i([0,1]\times\{1\})\subset
E$ for $i\in\{1,2,\ldots,l\}$ with $\kappa_i=1$.
\item The intersections $Q_i\cap D$, for $i=1,2,\ldots,l$, are at distance
at least $\delta$ from each other; the same holds for $Q_i\cap E$.
\item A counterclockwise order of the quads $Q_i$, for $i=1,2,\ldots,l$, is given
by ordering counterclockwise the connected components of $Q_i\cap D$
containing $Q_i(0,0)$.
\end{enumerate}
\end{definition}
We write $D\xleftrightarrow{\kappa} E$ for
$D\xleftrightarrow{\kappa,\hat{\mathbb{C}}}E$.

In the following we consider some special arm events.  For
$z\in\mathbb{C}$ and $a>0$, let
$H_1(z,a),H_2(z,a),H_3(z,a),H_4(z,a)$ denote the left, lower, right,
and upper half-planes which have the right, top, left and bottom
sides of $\Lambda_a(z)$ on their boundaries, respectively.  For
$j=1,2,3,4$, $\kappa\in\{0,1\}^l$ and $\kappa'\in\{0,1\}^{l'}$ with
$l,l'\geq 0$, we define the event $\mathcal
{A}_{\kappa,\kappa'}^j(z;a,b)$ where there are $l+l'$
disjoint arms with color sequence
$\kappa\vee\kappa':=(\kappa_1,\ldots,\kappa_l,\kappa_1',\ldots,\kappa_{l'}')$
in $A(z;a,b)$ so that the $l'$ arms, with color sequence $\kappa'$,
are in the half-plane $H_j(z,a)$.  That is,
\begin{align*}
\mathcal
{A}_{\kappa,\kappa'}^j(z;a,b):=\left\{\Lambda_a(z)\xleftrightarrow{\kappa\vee\kappa'}
\left(\hat{\mathbb{C}}\backslash\Lambda_b(z)\right)\right\}\cap\left\{\Lambda_a(z)\xleftrightarrow{\kappa',H_j(z,a)}
\left(\hat{\mathbb{C}}\backslash\Lambda_b(z)\right)\right\}.
\end{align*}
In the notation above, when $z$ is omitted, it is assumed to be 0.
When $\kappa'=\emptyset$, both the subscript $\kappa'$ and the
superscript $j$ will be omitted.  See Figure \ref{arm} for an illustration of an arm event.
\begin{figure}
\begin{center}
\includegraphics[height=0.35\textwidth]{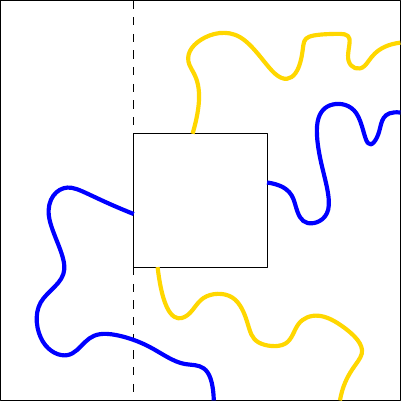}
\caption{Illustration of the event $\mathcal
{A}_{(0),(101)}^3(a,b)$.}\label{arm}
\end{center}
\end{figure}

Lemma \ref{l1} below is a near-critical analog of Lemma 10 of \cite{CCK19}
in the critical case.  As noted in \cite{CCK19}, although
Lemma 2.9 in \cite{GPS13} is a slightly weaker version of Lemma 10
in \cite{CCK19}, the proof of the former extends immediately to that
of the latter.  Similarly, Lemma 2.5 in \cite{GPS18b} (see also
Theorem 9.5 in \cite{GPS18}) is a slightly weaker version of Lemma
\ref{l1}, and its proof extends to this more general case. (The proof
of Lemma 2.9 in \cite{GPS13} works, using the stability for arm
events near criticality (\ref{e7}), together with the existence of
the near-critical scaling limit, Theorem \ref{t6}.)
\begin{lemma}\label{l1}
Under a coupling $\mathbf{P}$ of $(\mathbf{P}_p^{\eta})$ and
$\mathbf{P}^{\infty}$ on $(\mathscr{H},d_{\mathscr{H}})$ such that
$\omega_p^{\eta}\rightarrow\omega^{\infty}$ almost surely as $p\uparrow
p_c$, we have for events $\mathcal
{D}\in\{\{A\xleftrightarrow{(0),S}B\}, \{A\xleftrightarrow{(101),S}B\}, \mathcal
{A}_{\kappa,\kappa'}^j(z;a,b)\}$,
\begin{equation*}
\mathbf{1}_{\mathcal
{D}}\left(\omega_p^{\eta}\right)\rightarrow\mathbf{1}_{\mathcal
{D}}(\omega^{\infty})\qquad\mbox{in $\mathbf{P}$-probability as $p\uparrow
p_c$},
\end{equation*}
for
$(\kappa,\kappa')\in\{((0),\emptyset),((0101),\emptyset),((101010),\emptyset),(\emptyset,(101)),((0),(101))\}$, rectangle $S\subset \mathbb{C}$, $j\in\{1,2,3,4\},
0<a<b$ and $A,B$ disjoint simply connected subsets of $\mathbb{C}$
with piecewise smooth boundaries.
\end{lemma}

The following bounds on the probabilities of some special arm events will be used later.  Note that inequalities (\ref{e67}) and (\ref{e72}) imply (\ref{e83})
immediately; a combination of the right-hand inequality in
(\ref{e83}), Reimer's inequality and (\ref{e78}) gives (\ref{e84}).
\begin{lemma}\label{l2}
There exist constants $\lambda_6,\lambda_{1,3}>0$, such that for
any fixed $K\geq 1$, there exists a constant $C(K)>0$, such that for all
$p\in(p(1),p_c)$, $\eta\leq a<b\leq K$ and $j\in\{1,2,3,4\}$,
\begin{align}
&\mathbf{P}_p^{\eta}[\mathcal {A}_{(101010)}(a,b)]\leq
C\left(\frac{a}{b}\right)^{2+\lambda_6},\qquad\mathbf{P}_p^{\eta}[\mathcal
{A}_{\emptyset,(101)}^j(a,b)]\leq
C\left(\frac{a}{b}\right)^2,\label{e83}\\
&\mathbf{P}_p^{\eta}[\mathcal {A}_{(0),(101)}^j(a,b)]\leq
C\left(\frac{a}{b}\right)^{2+\lambda_{1,3}}.\label{e84}
\end{align}
\end{lemma}

Now we are ready to prove Theorems \ref{t1} and
\ref{t3}.  Since the proofs are analogous to those of the
corresponding results in the critical case in \cite{CCK19}, we only give a sketch.

\begin{proof}[Sketch of proof of Theorem \ref{t1}]
The proof is essentially the same as the proof of
Theorem 11 (see also Theorem 1 in \cite{CCK19} for its weaker version) in
\cite{CCK19}, by using estimates on near-critical arm events and the
existence of the near-critical scaling limit.  We shall give a sketch of the proof below,
and refer the reader to Sections 5 and 6 of \cite{CCK19} for the
details.

\textbf{Step 1}.  As in Section 5 of \cite{CCK19}, we construct two
approximations of blue clusters with $L^{\infty}$-diameters at least
$\delta>0$, which are completely contained in $\Lambda_k$.  The
first one relies solely on the arm events (see Definition 5 in
\cite{CCK19}), while the other is simply the union of $\epsilon$-boxes
which intersect the cluster.  It was showed in the critical case that when the mesh size is
small, with high probability these two approximations coincide; see Propositions 1 and 2 in Section 5 of \cite{CCK19}).  The
results in Section 5 of \cite{CCK19} just use estimates on
critical arm events given by Lemmas 4 and 5 in that paper.  Lemma \ref{l2} is the
corresponding lemma needed for our near-critical case.  Then we can obtain the ``clusters-approximations result'' (i.e.,
analog of Propositions 1 and 2 in \cite{CCK19}) in the near-critical
case by using Lemma \ref{l2} and the proof of Propositions 1 and 2
in \cite{CCK19}.

\textbf{Step 2}.  Following the proof of Theorem 11 given in Section
6 of \cite{CCK19}, we can derive our Theorem \ref{t1} by using the
clusters-approximations result in Step 1, Theorem \ref{t4} and Lemma
\ref{l1}.
\end{proof}

\begin{proof}[Proof of Theorem \ref{t3}]
The proof of the existence and uniqueness of the measure
$\mathbb{P}$ is analogous to the proofs of Theorem 3 in \cite{CCK19} and Theorem 6 in
\cite{CN06}.  Let $1\leq k\leq k_1<k_2$.  First, we claim that the marginal distributions
$\mathbb{P}_{k_1}|_{\Lambda_k}$ and $\mathbb{P}_{k_2}|_{\Lambda_k}$ are the same.  For this,
it suffices to show that under the coupling $\mathbf{P}$, we have $\dist(\mathscr{C}_{k_2}|_{\Lambda_{k_1}},\mathscr{C}_{k_1})=0$ with probability 1 for any integers $1\leq k_1<k_2$.  By Theorem \ref{t1}, for fixed $k\in\mathbb{N}$, $\dist(\mathscr{C}_{k}^p,\mathscr{C}_{k})\rightarrow 0$
in $\mathbf{P}$-probability as $p\uparrow p_c$.  This combined with the fact (due to the half-plane 3-arm event having
exponent larger than 1; see (\ref{e72})) that for any fixed $\delta>0$ and $k\in\mathbb{N}$, $\mathbf{P}[\mathscr{C}_{k+\epsilon}^p(\delta)\backslash\mathscr{C}_k^p(\delta)\neq\emptyset]$ tends to 0 uniformly for all $p\in(p(1),p_c)$ as $\epsilon\rightarrow 0$, implies our claim.  Hence, the consistency relations needed to apply Kolmogorov's extension
theorem (see, e.g., \cite{Dur10}) are satisfied.  Since $\Omega_k$ and $\Omega_{\infty}$ are
complete separable metric spaces, the measurable spaces
$(\Omega_k,\mathcal {B}_k)$ and $(\Omega_{\infty},\mathcal {B}_{\infty})$
are standard Borel spaces, and $\mathbb{P}_k$ is a
probability measure on $(\Omega_k,\mathcal {B}_k)$, we can apply
Kolmogorov's extension theorem and conclude that there is a
unique probability measure $\mathbb{P}$ on
$(\Omega_{\infty},\mathcal {B}_{\infty})$ such that
$\mathbb{P}|_{\Lambda_k}=\mathbb{P}_k$ for each $k\in\mathbb{N}$ and $\mathbb{P}$ is supported on collections of bounded, closed and connected
subsets of $\hat{\mathbb{C}}$.

Write $\mathscr{C}:=\bigcup_k\mathscr{C}_k$.  Then
the above argument gives that the distribution of $\mathscr{C}$ is $\mathbb{P}$.  Using (\ref{e40}), it is easy to see that for each $\epsilon>0$,
there exists $K=K(\epsilon)\geq 1$ such that for all $k\geq K$ and
$p\in(p(1),p_c)$,
\begin{equation}\label{e117}
\mathbf{P}[\Dist(\mathscr{C}_{k}^p,\mathscr{C}^p)\leq
\epsilon]\geq 1-\epsilon.
\end{equation}
Theorem \ref{t1} and (\ref{e117}) imply that $\Dist(\mathscr{C}^p,\mathscr{C})\rightarrow 0$ in probability as $p\uparrow p_c$.  Moreover, the fact that $\mathscr{C}(\delta)|_{\Lambda_k}=\mathscr{C}_k(\delta)$ a.s. and $\dist(\mathscr{C}|_{\Lambda_k},\mathscr{C}_k)=0$ a.s. (by the above argument) together with
Theorem \ref{t1} implies the following statements immediately:
$\mathscr{C}$ is a measurable function of $\omega^{\infty}$; for each $k\in\mathbb{N}$ and $\delta>0$,
$\mathscr{C}(\delta)|_{\Lambda_k}$ is a.s. a finite set, and $\widehat{\dist}(\mathscr{C}^{p}(\delta)|_{\Lambda_k},\mathscr{C}(\delta)|_{\Lambda_k})$,
$\dist(\mathscr{C}^p|_{\Lambda_k},\mathscr{C}|_{\Lambda_k})\rightarrow 0$ in probability as $p\uparrow p_c$.

It follows easily from the above statement that for any $k,n\in\mathbb{N}$,
$\dist(\mathscr{C}^p|_{\Lambda_k(n)},\mathscr{C}|_{\Lambda_k(n)})\rightarrow 0$ in probability as $p\uparrow p_c$.
It is clear that for any $k\in\mathbb{N}$, $n\geq 2k+1$ and $p\in(p(1),p_c)$, the configurations $\mathscr{C}^p|_{\Lambda_k}$ and $\mathscr{C}^p|_{\Lambda_k(n)}$ are independent, so $\mathscr{C}|_{\Lambda_k}$ and $\mathscr{C}|_{\Lambda_k(n)}$ are independent.

We give the proof of the rotational invariance of $\mathbb{P}$ below;
similar proof works for the translation invariance, and we omit it.  For $\theta\in[0,2\pi]$, let $f(z)=e^{i\theta}z$ be a
rotation of $\mathbb{C}$.  By Theorem
\ref{t4}, $f(\omega^{\infty})$ and $\omega^{\infty}$ have the same distribution.
Therefore, similarly as the proof of Theorem \ref{t1}, we can use a
coupling such that $\omega_p^{\eta}\rightarrow f(\omega^{\infty})$ in
$(\mathscr{H},d_{\mathscr{H}})$ a.s., to show that
$\mathscr{C}_{f(\Lambda_k)}^p$ converges in distribution
to $f(\mathscr{C}_{k})$ with respect to the metric $\Dist$ (note
that $f(\mathscr{C}_{k})$ is constructed by using $f(\omega^{\infty})$
and rotated boxes $f(\Lambda_{\epsilon/2}(z))$).  We deduce from
this, by letting $k\rightarrow\infty$, that
$\mathscr{C}^p$ converges in distribution to
$f(\mathscr{C})$ with respect to $\Dist$, as $p\uparrow p_c$.  Since we have proved that $\mathscr{C}^p$ also
converges in distribution to $\mathscr{C}$ with respect to
$\Dist$ as $p\uparrow p_c$, it follows that $f(\mathscr{C})$ and $\mathscr{C}$ have the same distribution.
\end{proof}
\begin{remark}\label{portion}
Consider portions of clusters in a region, that is, the connected components in the region which come from the clusters intersecting the region but not completely contained in it.  In Section 3.3 of the preprint version \cite{Yao21} of the present paper,
we constructed the scaling limit of the collection of portions of clusters in a strip as $p\uparrow p_c$, which was used to prove Theorem \ref{t2} in that preprint.

For critical site percolation, the scaling limits in the loop-ensemble and quad-crossing spaces are
equivalent in the sense that the associated $\sigma$-algebras are the same; see Section 2.3 of \cite{GPS13} for a proof that the loops
determine the quad-crossing information, and see Theorem 6.10 in \cite{HS} for the converse result.  To our knowledge, similar equivalence result has not been proved for near-critical percolation.  Moreover, this equivalence result for critical percolation implies that the scaling limit of the collection of
portions of clusters in Theorem 12 of \cite{CCK19} is not only a measurable function of the pair of quad-crossing and loop-ensemble
scaling limits, but also measurable with respect to the single quad-crossing scaling limit.
\end{remark}

\section{FPP based on the scaling limit of near-critical percolation}\label{fpp}

In this section, we define a ``continuum'' FPP on the continuum cluster ensemble $\mathscr{C}$.
We also define its discrete analog for $\mathscr{C}^{p}$ on the lattice $\eta\mathbb{T}$.  Then
we show that as $p\uparrow p_c$, some quantities of the discrete FPP
approach their corresponding quantities of the continuum FPP.
Finally, we give a law of large numbers for the ``point-to-point''
passage times of the continuum FPP.

\subsection{Definition of first-passage times}\label{DefineFPP}
We first define first-passage times for $\mathscr{C}$.  Then we
define analogous quantities for $\mathscr{C}^{p}$ with $p\in(p(1),p_c)$.

A \textbf{(continuum) chain} is a finite sequence $\Gamma=(\mathcal
{C}_0,\mathcal {C}_1,\ldots,\mathcal {C}_k)$ of distinct clusters in
$\mathscr{C}$, such that $\mathcal {C}_{j-1}\cap\mathcal
{C}_j\neq\emptyset$ for all $j\in\{1,2,\ldots,k\}$.  Let $|\Gamma|$
denote the length, or, more precisely, the number of clusters in
$\Gamma$.  The \textbf{passage time} of $\Gamma$ is defined by
$T(\Gamma):=|\Gamma|-1$.   Similarly to the third item in Theorem 2
of \cite{CN06} (which was called ``finite chaining'' in Proposition
2.7 in \cite{GPS13}) for the full-plane CLE$_6$, $\mathscr{C}$ also has the
finite chaining property: Almost surely (under $\mathbf{P}^{\infty}$),
for any pair of clusters $\mathcal {C},\mathcal {C}'\in
\mathscr{C}$, there is a chain $(\mathcal {C},\mathcal
{C}_1,\ldots,\mathcal {C}_k,\mathcal {C}')$ in $\mathscr{C}$
connecting $\mathcal {C}$ and $\mathcal {C}'$.  It is not hard to
show this property, but we will not give its proof here since we
will not use it in our proofs.

The \textbf{first-passage time} between any pair of
clusters $\mathcal {C},\mathcal {C}'\in \mathscr{C}$ is defined by
\begin{equation*}
T(\mathcal {C},\mathcal {C}'):=\inf \{T(\Gamma):\Gamma\mbox{ is a
chain in $\mathscr{C}$ that starts at $\mathcal {C}$ and ends at
$\mathcal {C}'$}\},
\end{equation*}
where we use the convention that the infimum of the empty set is
$\infty$.  It follows from the finite chaining property that the
first-passage times between all pairs of clusters in $\mathscr{C}$
are almost surely finite.
\begin{remark}
If we view each cluster of $\mathscr{C}$ as a single vertex, and
define a new edge between any pair of vertices if their
corresponding clusters touch each other, then we get a graph called
``(continuum) cluster graph''.  Observe that a chain corresponds to a
(self-avoiding) path of the cluster graph, and the first-passage
time between two clusters is just the graph distance between their
corresponding vertices in the cluster graph.  The finite chaining
property implies that the cluster graph is almost surely connected.
\end{remark}

\begin{figure}
\begin{center}
\includegraphics[height=0.35\textwidth]{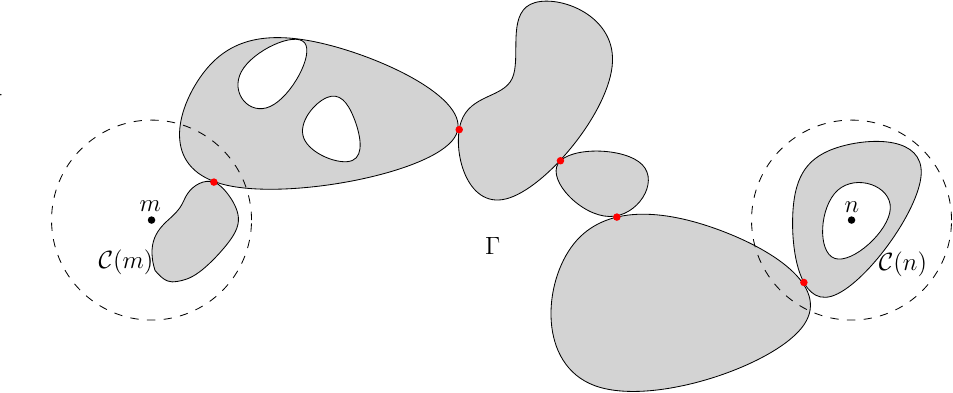}
\caption{A chain $\Gamma$ from $\mathcal
{C}(m)$ to $\mathcal
{C}(n)$ such that $T_{m,n}^{\theta=0}=T(\Gamma)=5$.}\label{chain}
\end{center}
\end{figure}

Lemma \ref{l13} below states that, for any fixed point $z\in
\mathbb{C}$, there exists almost surely a unique cluster whose Euclidean diameter is maximal
among the clusters of $\mathscr{C}$ in $\mathbb{D}_{1/2}(z)$, which is
called the \textbf{largest (continuum) cluster in $\mathbb{D}_{1/2}(z)$} and denoted by $\mathcal {C}(z)$; if there is no such cluster, then let $\mathcal
{C}(z)=\emptyset$.  Now we can define the point-to-point passage times
for $\mathscr{C}$: For $m,n\in\mathbb{Z}$ and
$\theta\in[0,2\pi]$, let
\begin{equation*}
T_{m,n}^{\theta}=T_{m,n}^{\theta}(\mathscr{C}):=T(\mathcal
{C}(me^{i\theta}),\mathcal {C}(ne^{i\theta})),
\end{equation*}
where we let $T_{m,n}^{\theta}=\infty$ if $\mathcal
{C}(me^{i\theta})$ or $\mathcal {C}(ne^{i\theta})$ is empty.  See Figure \ref{chain}.  When $\theta=0$, we write $T_{m,n}:=T_{m,n}^{0}$.  We remark that there are other ways to define
point-to-point passage times for $\mathscr{C}$.  The key point is
that for each $z\in\mathbb{C}$ one should choose an appropriate
cluster close to $z$.  Different choice may lead to different
technical arguments, but the strategy for the main result is
essentially the same.

To study $T_{m,n}^{\theta}$, we need to define its discrete version.
Suppose that $z\in\mathbb{C}$ and $p\in(p(1),p_c)$.  In the following the sites in $\eta\mathbb{T}$ are viewed as hexagons.
We let $\mathcal {C}^p(z)$ denote the cluster whose Euclidean diameter is maximal
among the clusters of $\mathscr{C}^p$ in $\mathbb{D}_{1/2}(z)$, and call it the \textbf{largest (discrete) cluster in $\mathbb{D}_{1/2}(z)$}.
When there are more than one candidates for $\mathcal {C}^p(z)$, then we choose one of them in some deterministic way; when there are no clusters of $\mathscr{C}^p$ in $\mathbb{D}_{1/2}(z)$, then we let $\mathcal {C}^p(z)$ denote the $\eta$-hexagon containing $z$ (if $z$ is on the boundary of two or three hexagons, then we choose one of them in some deterministic way).  In fact, the arguments in Section \ref{largest} imply that with high probability the above two events do not occur for $p$ close to $p_c$.  For $m,n\in\mathbb{Z}$ and $\theta\in[0,2\pi]$, let
\begin{equation*}
T_{m,n}^{p,\theta}=T_{m,n}^{p,\theta}(\mathscr{C}^p):=T(\mathcal
{C}^p(me^{i\theta}),\mathcal {C}^p(ne^{i\theta})).
\end{equation*}
When $\theta=0$, we simply drop it from the notation, writing $T_{m,n}^p:=T_{m,n}^{p,0}$.

Similarly to its continuum version above, a \textbf{(discrete)
chain} is a finite sequence $(\mathcal {C}_0,\mathcal
{C}_1,\ldots,\mathcal {C}_k)$ of distinct clusters in
$\mathscr{C}^p$, such that for each $j\in\{1,2,\ldots,k\}$, there is
a yellow $\eta$-hexagon touching both of $\mathcal {C}_{j-1}$ and
$\mathcal {C}_j$.

\subsection{Properties of large clusters}\label{largest}
This subsection gives some properties of clusters with large diameters in region.  In particular, Lemma
\ref{l13} allows us to define $\mathcal {C}(z)$, and Lemma \ref{l37} says that with high probability $\mathcal {C}(z)$ is well-approximated by $\mathcal {C}^p(z)$ for $p$ close to $p_c$, under a coupling.  As mentioned in Section \ref{DefineFPP}, one could use a different large cluster ``close'' to $z$ instead of $\mathcal {C}(z)$ used here to study the point-to-point passage times for $\mathscr{C}$.  For example, one may choose the outermost cluster surrounding $z$ in $\mathbb{D}_{1/2}(z)$, as we did in \cite{Yao21}, or the cluster having the largest ``volume'' in $\mathbb{D}_{1/2}(z)$.  However, the former choice leads to a lot of technical issues, while the latter involves counting measures associated to continuum clusters (see \cite{CCK19}), which are rather complicated.
The present definition of $\mathcal {C}(z)$ is more natural and makes the proofs cleaner.

For a set $S\subset\mathbb{C}$, let $\diam(S)$ denote the Euclidean diameter of $S$.  For convenience, we will just say diameter to mean Euclidean diameter in the rest of the paper.  Recall that the $L^{\infty}$-diameter $\diam_{\infty}(\cdot)$ is used in Section \ref{continuum}.

The following proposition will be used in the proofs of the next two lemmas.  Roughly speaking, it says that, for $p$ close $p_c$, with high probability the gaps between the diameters of clusters with large diameters are not too small.  Note that Proposition 3.2 in \cite{BC13} says that with high probability, the gaps between the volumes of critical clusters with large diameters are also not too small; similar result also holds for the near-critical case.
\begin{proposition}\label{p3}
For all $\delta,\epsilon_1\in(0,1)$ there exists $\epsilon_2=\epsilon_2(\delta,\epsilon_1)>0$ such that for all $p\in(p(5/\epsilon_2),p_c)$,
\begin{equation*}
\mathbf{P}_p^{\eta}[\mbox{there exist distinct $\mathcal{C},\mathcal{C}'\in\mathscr{C}^p(\delta)|_{\Lambda_1}$ such that $|\diam(\mathcal{C})-\diam(\mathcal{C}')|\leq\epsilon_2$}]\leq\epsilon_1.
\end{equation*}
\end{proposition}
\begin{proof}
The proof is based on upper bounds for arm probabilities and a counting argument.  Let $\mathcal{A}_{\sigma}^{\theta,\vartheta}(z;r,R)$ denote the wedge arm event for $\omega_p^{\eta}$ defined at the beginning of Section \ref{pre}.  Write $\mathcal{A}_{(0),(101)}^{\theta,\vartheta}(z;r,R):=\mathcal{A}_{(0)}(z;r,R)\circ\mathcal{A}_{(101)}^{\theta,\vartheta}(z;r,R)$, where $\circ$ means the disjoint occurrence (see e.g. Chapter 2 of \cite{Gri99}).  By (\ref{e78}) and (\ref{e23}),  we know that there exist constants $C_0>0,\vartheta_0\in(\pi,3\pi/2)$ and $0<3\beta<\lambda_1<1/2$, such that for all $p\in(p(1),p_c)$, $\eta\leq r<R\leq 1$ and $\theta\in[0,2\pi]$,
\begin{equation}\label{e30}
\mathbf{P}_p^{\eta}[\mathcal{A}_{(0)}(r,R)]\leq C_0\left(\frac{r}{R}\right)^{\lambda_1}\quad\mbox{and}\quad\mathbf{P}_p^{\eta}[\mathcal{A}_{(101)}^{\theta,\vartheta_0}(r,R)]\leq C_0\left(\frac{r}{R}\right)^{2-\beta}.
\end{equation}
Letting $\lambda:=\lambda_1-\beta$, the above two inequalities combined with Reimer's inequality implies that
\begin{equation}\label{e31}
\mathbf{P}_p^{\eta}[\mathcal{A}_{(0),(101)}^{\theta,\vartheta_0}(z;r,R)]\leq C_0^2\left(\frac{r}{R}\right)^{2+\lambda}.
\end{equation}

Let $\delta\in(0,1)$ and $\epsilon\in(0,\delta/50)$.  Given a closed subset $S$ of $\Lambda_1$ with $\diam(S)\geq\delta$, we choose a pair $\{u(S),v(S)\}$ of vertices in $\epsilon\mathbb{Z}^2$ such that $\Lambda_{\epsilon/2}(u(S))$ and $\Lambda_{\epsilon/2}(v(S))$ contain two points of $S$ whose distance is equal to the diameter of $S$.  If there are more than one candidates for this pair, we choose one of them in some deterministic way.  Write $\tilde{u}(u,v):=u+\epsilon\cdot(u-v)/|u-v|$ and $\tilde{v}(u,v):=v+\epsilon\cdot(v-u)/|v-u|$.  It is easy to check that there exists $\epsilon_0\in(0,\delta/50)$, such that for all $\epsilon\in(0,\epsilon_0)$ and all closed subset $S$ of $\Lambda_1$ with $\diam(S)\geq\delta$, the set $S+\mathbb{D}_{\epsilon/5}$ is contained in the intersection of the wedges $W(\tilde{u};\arg(v-u),\vartheta_0)$ and $W(\tilde{v};\arg(u-v),\vartheta_0)$.

In the rest of the proof, we suppose that $\delta\in(0,1),\epsilon\in(0,\epsilon_0)$ and $p\in(p(5/\epsilon),p_c)$.  Assume that $\mathcal{C},\mathcal{C}'$ are distinct clusters of $\mathscr{C}^p(\delta)$ in $\Lambda_1$ with $|\diam(\mathcal{C})-\diam(\mathcal{C}')|\leq\epsilon$.  For simplicity, we let $u,u'$ denote $u(\mathcal{C}),u(\mathcal{C}')$, respectively, and let $\tilde{u}:=\tilde{u}(u,v),\tilde{u}':=\tilde{u}(u',v')$.  We assume without loss of generality that ${|u-u'|\leq\min\{|u-v'|,|v-v'|,|v-u'|\}}$.  By considering the distance between $u$ and $u'$, we distinguish the following three cases; in each case the clusters $\mathcal {C},\mathcal {C}'$ together with their yellow site boundaries produces the corresponding independent arm events.  See Figure \ref{diameter} for an illustration.
\begin{figure}
\begin{center}
\includegraphics[height=0.45\textwidth]{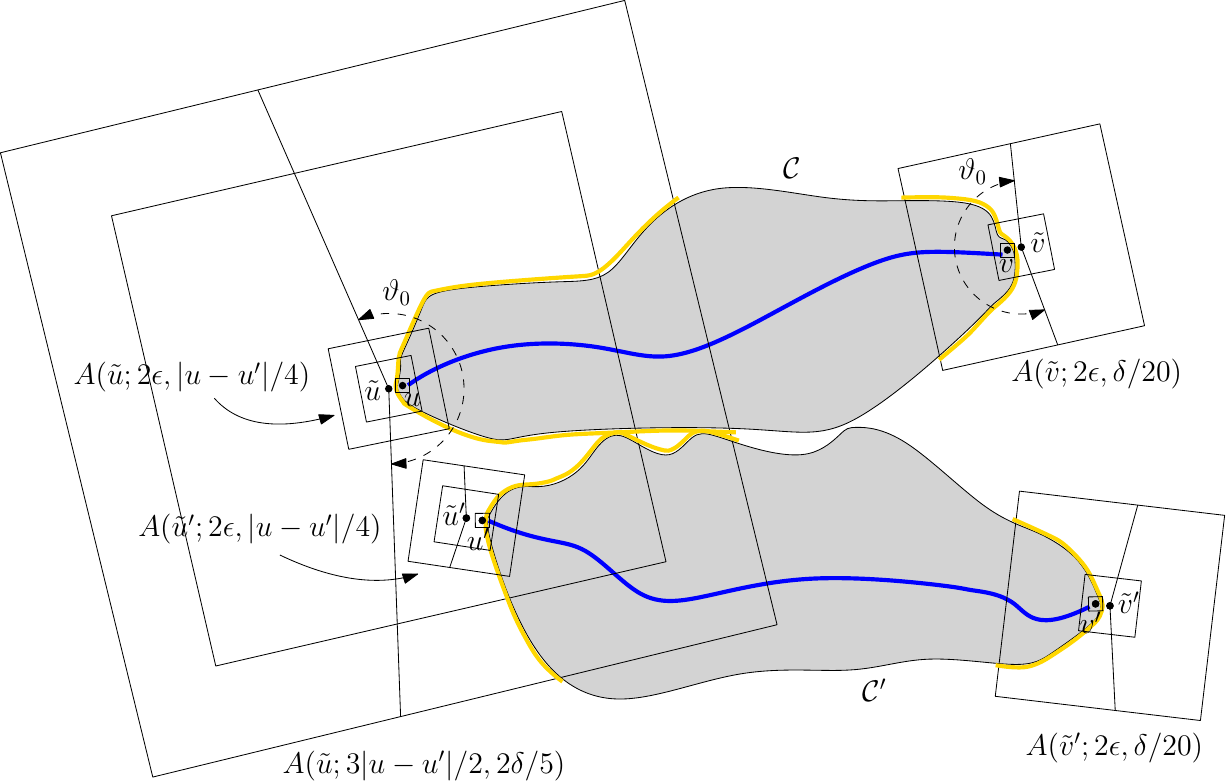}
\caption{Illustration of arm events when $|u-u'|\in[10\epsilon,\delta/5]$ and $|v-v'|>\delta/5$.  The clusters $\mathcal {C},\mathcal {C}'$ and their site boundaries provide five independent arm events $\mathcal{A}_{(101)}^{\arg(v-u),\vartheta_0}(\tilde{u};2\epsilon,|u-u'|/4)$, $\mathcal{A}_{(101)}^{\arg(v'-u'),\vartheta_0}(\tilde{u}';2\epsilon,|u-u'|/4)$, $\mathcal{A}_{(0),(101)}^{\arg(v-u),\vartheta_0}(\tilde{u};3|u-u'|/2,2\delta/5)$,
$\mathcal{A}_{(101)}^{\arg(u-v),\vartheta_0}(\tilde{v};2\epsilon,\delta/20)$ and $\mathcal{A}_{(101)}^{\arg(u'-v'),\vartheta_0}(\tilde{v}';2\epsilon,\delta/20)$.
}\label{diameter}
\end{center}
\end{figure}
\begin{itemize}
  \item If $|u-u'|\in[10\epsilon,\delta/5]$, then the events $\mathcal{A}_{(101)}^{\arg(v-u),\vartheta_0}(\tilde{u};2\epsilon,|u-u'|/4)$,
  $\mathcal{A}_{(101)}^{\arg(v'-u'),\vartheta_0}(\tilde{u}';2\epsilon,|u-u'|/4)$ and ${\mathcal{A}_{(0),(101)}^{\arg(v-u),\vartheta_0}(\tilde{u};3|u-u'|/2,2\delta/5)}$ occur.
  \item If $|u-u'|>\delta/5$, then $\mathcal{A}_{(101)}^{\arg(v-u),\vartheta_0}(\tilde{u};2\epsilon,\delta/20)$ and $\mathcal{A}_{(101)}^{\arg(v'-u'),\vartheta_0}(\tilde{u}';2\epsilon,\delta/20)$ occur.
  \item If $|u-u'|\in[0,10\epsilon)$, then $\mathcal{A}_{(0),(101)}^{\arg(v-u),\vartheta_0}(\tilde{u};12\epsilon,2\delta/5)$ occurs.
\end{itemize}
Similar statements also hold for $v$ and $v'$.  It is easy to check that the arm events introduced in each case are independent; see, for example, the five arm events shown in Figure \ref{diameter} are independent.

Note that $\diam(\mathcal{C})\in[\delta,2\sqrt{2}]$.  Now, assume that $\diam(\mathcal{C})\in[\delta_1,\delta_1+\epsilon]$ for any fixed $\delta_1\in[\delta,2\sqrt{2}]$.  Then, there are $O(\epsilon^{-3})$ choices for the pair $\{u,v\}$.  Fix $u$ and $v$.  Then there are $O(1)$ choices for $u'$ satisfying $|u-u'|\in[0,10\epsilon)$, $O(\epsilon^{-2})$ choices for $u'$ satisfying $|u-u'|>\delta/5$, and $O(4^j)$ choices for $u'$ satisfying $|u-u'|\in[10\epsilon,\delta/5]$ and $u'\in A(u;2^j\epsilon,2^{j+1}\epsilon)$ with $j\geq 3$.  Now fix $u'$ also.  Then there are $O(1)$ choices for $v'$ satisfying $|v-v'|\in[0,10\epsilon)$, $O(\epsilon^{-1})$ choices for $v'$ satisfying $|v-v'|>\delta/5$, and $O(2^k)$ choices for $v'$ satisfying $|v-v'|\in[10\epsilon,\delta/5]$ and $v'\in A(v;2^k\epsilon,2^{k+1}\epsilon)$ with $k\geq 3$.  The above arguments imply that there exist $C_1,C_2>0$ such that for all
$\delta\in(0,1),\delta_1\in[\delta,2\sqrt{2}],\epsilon\in(0,\epsilon_0)$ and $p\in(p(5/\epsilon),p_c)$,
\begin{align}
&\mathbf{P}_p^{\eta}[\mbox{there exist distinct $\mathcal{C},\mathcal{C}'\in\mathscr{C}^p(\delta)|_{\Lambda_1}$ such that $\diam(\mathcal{C})\in[\delta_1,\delta_1+\epsilon]$ and $|\diam(\mathcal{C})-\diam(\mathcal{C}')|\leq\epsilon$}]\nonumber\\
&\leq\frac{C_1}{\epsilon^3}
\left(\left(\underbrace{\left(\frac{\epsilon}{\delta}\right)^{2+\lambda}}_{(1)}+\underbrace{\sum_{j=3}^{\lceil\log_2(\delta/\epsilon)\rceil}4^j\left(\frac{\epsilon}{2^j\epsilon}\right)^{4-2\beta}\left(\frac{2^j\epsilon}{\delta}\right)^{2+\lambda}}_{(2)}\right)
\left(\underbrace{\frac{1}{\epsilon}\left(\frac{\epsilon}{\delta}\right)^{4-2\beta}}_{(3)}
+\underbrace{\sum_{k=3}^{\lceil\log_2(\delta/\epsilon)\rceil}2^k\left(\frac{\epsilon}{2^k\epsilon}\right)^{4-2\beta}\left(\frac{2^k\epsilon}{\delta}\right)^{2+\lambda}}_{(4)}\right)\right.\nonumber\\
&\left.\qquad\qquad+\underbrace{\left(\frac{\epsilon}{\delta}\right)^{4+2\lambda}}_{(5)}+\underbrace{\frac{1}{\epsilon^3}\left(\frac{\epsilon}{\delta}\right)^{8-4\beta}}_{(6)}\right)\quad\mbox{by (\ref{e30}) and (\ref{e31})}\nonumber\\
&\leq C_2\epsilon^{1+\lambda-2\beta}/\delta^8,\label{e32}
\end{align}
where the six terms (1)--(6) correspond respectively to the following six cases: (1) $|u-u'|\in[0,10\epsilon)$, (2) $|u-u'|\in[10\epsilon,\delta/5]$, (3) $|v-v'|>\delta/5$, (4) $|v-v'|\in[10\epsilon,\delta/5]$, (5) $|u-u'|,|v-v'|\in[0,10\epsilon)$, (6) $|u-u'|,|v-v'|>\delta/5$.  Hence,
\begin{align*}
&\mathbf{P}_p^{\eta}[\mbox{there exist distinct $\mathcal{C},\mathcal{C}'\in\mathscr{C}^p(\delta)|_{\Lambda_1}$ such that $|\diam(\mathcal{C})-\diam(\mathcal{C}')|\leq\epsilon$}]\\
&\leq\sum_{j=0}^{\lfloor(2\sqrt{2}-\delta)/\epsilon\rfloor}\mathbf{P}_p^{\eta}[\mbox{$\exists$ distinct $\mathcal{C},\mathcal{C}'\in\mathscr{C}^p(\delta)|_{\Lambda_1}$ s.t. $\diam(\mathcal{C})\in[\delta+j\epsilon,\delta+(j+1)\epsilon]$ and $|\diam(\mathcal{C})-\diam(\mathcal{C}')|\leq\epsilon$}]\\
&\leq 4C_2\epsilon^{\lambda-2\beta}/\delta^8\quad\mbox{by (\ref{e32})}.
\end{align*}
The lemma follows easily from this since $\lambda>2\beta$.
\end{proof}

\begin{lemma}\label{l13}
Let $z\in\mathbb{C}$.  Almost surely, there is a unique cluster whose diameter is maximal
among the clusters of $\mathscr{C}$ in $\mathbb{D}_{1/2}(z)$, which is
called the \textbf{largest (continuum) cluster in $\mathbb{D}_{1/2}(z)$} and denoted by $\mathcal {C}(z)$.
Moreover, for each $\epsilon>0$, there exists $\delta=\delta(\epsilon)>0$
such that
\begin{equation*}
\mathbf{P}^{\infty}[\diam(\mathcal{C}(z))\geq\delta]\geq 1-\epsilon.
\end{equation*}
\end{lemma}
\begin{proof}
We shall prove the lemma when $z=0$; the general case then follows since the law of $\mathscr{C}$ is invariant under translations by Theorem \ref{t3}.  Let $\mathbf{P}$ be a coupling such that $\omega_p^{\eta}\rightarrow\omega^{\infty}$ in
$(\mathscr{H},d_{\mathscr{H}})$ a.s. as $p\uparrow p_c$.  Fix any $\delta\in(0,1)$.  For each $\epsilon_1\in(0,1)$, there exist $\epsilon_2>0,p_1\in(0,p_c)$ such that
\begin{align}
&\mathbf{P}[\mbox{there exist distinct $\mathcal{C},\mathcal{C}'\in\mathscr{C}(\delta)|_{\Lambda_1}$ such that $|\diam(\mathcal{C})-\diam(\mathcal{C}')|\leq\epsilon_2$}]\nonumber\\
&\leq\mathbf{P}\left[\mbox{$\exists$ distinct $\mathcal{C},\mathcal{C}'\in\mathscr{C}^{p_1}(\delta)|_{\Lambda_1}$ s.t. $|\diam(\mathcal{C})-\diam(\mathcal{C}')|\leq 3\epsilon_2$}\right]+\mathbf{P}\left[\widehat{\dist}(\mathscr{C}(\delta)|_{\Lambda_1},\mathscr{C}^{p_1}(\delta)|_{\Lambda_1})\geq\epsilon_2/4\right]\nonumber\\
&\leq\epsilon_1\quad\mbox{by Proposition \ref{p3} and Theorem \ref{t3}.}\label{e21}
\end{align}
Letting $\epsilon_1\rightarrow 0$ gives that
\begin{equation}\label{e22}
\mathbf{P}[\mbox{there exist distinct $\mathcal{C},\mathcal{C}'\in\mathscr{C}(\delta)|_{\Lambda_1}$ such that $|\diam(\mathcal{C})-\diam(\mathcal{C}')|=0$}]=0.
\end{equation}
Therefore, for each $\epsilon>0$ there exist $\delta\in(0,1/8)$ and $p_2\in(0,p_c)$ such that
\begin{align*}
&\mathbf{P}[\mbox{$\mathscr{C}(\delta)|_{\mathbb{D}_{1/2}}\neq\emptyset$ and $\exists$ a unique cluster whose diameter is maximal among the clusters of $\mathscr{C}(\delta)|_{\mathbb{D}_{1/2}}$}]\\
&=\mathbf{P}[\mathscr{C}(\delta)|_{\mathbb{D}_{1/2}}\neq\emptyset]\quad\mbox{by (\ref{e22})}\\
&\geq\mathbf{P}[\mathscr{C}^{p_2}(2\delta)|_{\mathbb{D}_{1/4}}\neq\emptyset]-\mathbf{P}[\dist(\mathscr{C}|_{\Lambda_1},\mathscr{C}^{p_2}|_{\Lambda_1})\geq\delta/2]\\
&\geq 1-\epsilon\quad\mbox{by RSW, FKG and Theorem \ref{t3}.}
\end{align*}
The lemma follows from this immediately.
\end{proof}

\begin{lemma}\label{l37}
Let $\mathbf{P}$ be a coupling such that $\omega_p^{\eta}\rightarrow\omega^{\infty}$ in
$(\mathscr{H},d_{\mathscr{H}})$ a.s. as $p\uparrow p_c$.  For each $k\in\mathbb{N}$ and $\epsilon,\delta>0$, there exists $p_0\in(0,p_c)$,
such that for all $p\in(p_0,p_c)$ and $z\in\Lambda_k$,
\begin{equation*}
\mathbf{P}[d_H(\mathcal {C}^p(z),\mathcal {C}(z))\geq
\delta]\leq\epsilon.
\end{equation*}
\end{lemma}
\begin{proof}
Let $0<\delta_1<\delta/8<1/10$, $p\in(p(5/\delta_1),p_c)$, $k\in\mathbb{N}$ and $z\in\Lambda_k$.  Define the events
\begin{align*}
&\mathcal {E}_1=\mathcal
{E}_1(\delta_1,p,k):=\{\dist(\mathscr{C}|_{\Lambda_{k+1}},\mathscr{C}^p|_{\Lambda_{k+1}})<\delta_1\},\\
&\mathcal {E}_2=\mathcal {E}_2(z;\delta,\delta_1,p):=\{\diam(\mathcal
{C}^p(z))\geq\delta\mbox{ and }\mathscr{C}^p(\delta/2)|_{\mathbb{D}_{1/2+2\delta_1}(z)}=\mathscr{C}^p(\delta/2)|_{\mathbb{D}_{1/2-2\delta_1}(z)}\},\\
&\mathcal {E}_3=\mathcal {E}_3(z;\delta,\delta_1):=\{\mbox{for any distinct $\mathcal{C},\mathcal{C}'\in\mathscr{C}(\delta/2)|_{\Lambda_1(z)}, |\diam(\mathcal{C})-\diam(\mathcal{C}')|\geq 9\delta_1$}\}.
\end{align*}

First, we prove that $\mathcal {E}_1\cap\mathcal {E}_2\cap\mathcal
{E}_3\subset\{d_H(\mathcal {C}^p(z),\mathcal {C}(z))<\delta_1\}$.  Assume that the event $\mathcal {E}_1\cap\mathcal {E}_2\cap\mathcal
{E}_3$ holds.  Since $\mathcal {E}_1\cap\mathcal {E}_2$ occurs, there is a continuum cluster $\mathcal{C}_0$ such that $d_H(\mathcal {C}^p(z),\mathcal {C}_0)<
\delta_1$, and furthermore $\mathcal{C}^p(z)\subset\mathbb{D}_{1/2-2\delta_1}(z)$, ${\mathcal{C}_0\subset\mathbb{D}_{1/2}(z)}$ and ${\diam(\mathcal{C}_0)>\diam(\mathcal{C}^p(z))-4\delta_1>\delta/2}$.  It remains to check that $\mathcal{C}_0=\mathcal{C}(z)$.  Assume for a
contradiction that ${\mathcal{C}_0\neq\mathcal{C}(z)}$.  Since $\mathcal {E}_3$ occurs, we have $\diam(\mathcal {C}(z))\geq\diam(\mathcal {C}_0)+9\delta_1>\diam(\mathcal{C}^p(z))+5\delta_1$.  Since $\mathcal {E}_1$ holds, there is a cluster $\mathcal{C}_0^p$ of $\mathscr{C}^p$ such that $d_H(\mathcal {C}_0^p,\mathcal {C}(z))<\delta_1$, $\mathcal {C}_0^p\subset\mathbb{D}_{1/2+2\delta_1}(z)$ and $\diam(\mathcal {C}_0^p)>\diam(\mathcal {C}(z))-4\delta_1>\diam(\mathcal{C}^p(z))+\delta_1$.  Furthermore, the fact that $\mathcal {E}_2$ holds implies that $\mathcal {C}_0^p\subset\mathbb{D}_{1/2}(z)$.  This together with $\mathcal {C}_0^p>\diam(\mathcal{C}^p(z))$ contradicts the definition of $\mathcal{C}^p(z)$.

Therefore, for all $k\in\mathbb{N},\epsilon>0$ and $\delta\in(0,1/2)$, there exist $\delta_1\in(0,\delta/8)$ and $p_0\in(0,p_c)$ such that for all $p\in(p_0,p_c)$ and $z\in\Lambda_k$,
\begin{align*}
\mathbf{P}[d_H(\mathcal {C}^p(z),\mathcal {C}(z))<\delta]&\geq\mathbf{P}[d_H(\mathcal {C}^p(z),\mathcal {C}(z))<\delta_1]\\
&\geq 1-\mathbf{P}[\neg\mathcal {E}_1]-\mathbf{P}[\neg\mathcal {E}_2]-\mathbf{P}[\neg\mathcal {E}_3]\quad\mbox{by $\mathcal {E}_1\cap\mathcal {E}_2\cap\mathcal
{E}_3\subset\{d_H(\mathcal {C}^p(z),\mathcal {C}(z))<\delta_1\}$}\\
&\geq 1-\epsilon/2-\mathbf{P}[\neg\mathcal {E}_2]\quad\mbox{by Theorem \ref{t3} and (\ref{e21})}\\
&\geq 1-\epsilon\quad\mbox{by RSW, FKG and (\ref{e72}), via a standard argument}.
\end{align*}
This completes the proof of the lemma.
\end{proof}

\subsection{Properties of first-passage times}
\begin{lemma}\label{l7}
There are constants $C_1,C_2>0$, such that for each
$n\in\mathbb{Z}$, $k\in\mathbb{N}$, $\theta\in[0,2\pi]$ and $p\in(p(10),p_c)$,
\begin{equation}\label{e11}
\mathbf{P}_p^{\eta}[T_{n,n+1}^{p,\theta}\geq k]\leq C_1\exp(-C_2k).
\end{equation}
Moreover, there is $C_0>0$ such that for each $\theta\in[0,2\pi]$, $p\in(p(10),p_c)$ and $m,n\in\mathbb{Z}$ with $m<n$,
\begin{equation}\label{e48}
\mathbf{E}_p^{\eta}[T_{m,n}^{p,\theta}]\leq C_0(n-m).
\end{equation}
\end{lemma}
\begin{proof}
For simplicity, we prove Lemma \ref{l7} in the case
$\theta=0$. The proof extends immediately to the general case.

Suppose that $p\in(p(10),p_c)$.  Assume that $\mathcal {C}^p(n)$ and $\mathcal {C}^p(n+1)$ intersect boxes $\Lambda_r(u)$ and $\Lambda_r(v)$ for some $u,v\in 2r\mathbb{Z}^2$, respectively.  If $r\in [\eta,1/6]$ and $\diam(\mathcal {C}^p(n)),\diam(\mathcal {C}^p(n+1))\geq 6r$, then $\mathcal {C}^p(n)$ and $\mathcal {C}^p(n+1)$ produce 1-arm events $\mathcal{A}_{(0)}(u;r,3r)$ and
$\mathcal{A}_{(0)}(v;r,3r)$, respectively, and furthermore,
\begin{equation}\label{e86}
T_{n,n+1}^p\leq X^p(u;r,4)+X^p(v;r,4)+Y^p(n,u,v;r),
\end{equation}
where $X^p(z;r_1,r_2)$ is the annulus-crossing time defined above Lemma \ref{l10}, $Y^p(n,u,v;r):=Y^p(u;r)+Y^p(v;r)+Y^p(n)$ and
\begin{align*}
&Y^p(z;r):=\inf\{T(\gamma)(\omega_p^{\eta}):\gamma\mbox{ is a
circuit surrounding $z$ in $A(z;r,3r)$}\},\\
&Y^p(n):=\inf\{T(\gamma)(\omega_p^{\eta}):\mbox{$\gamma$ is a circuit surrounding $n+1/2$ in $A(n+1/2;1,2)$}\}.
\end{align*}
When $r=\eta$, it is easy to check that the following inequality always holds, although it could happen that ${\diam(\mathcal {C}^p(n))<6\eta}$, or $\mathcal {C}^p(n)$ is even a yellow $\eta$-hexagon.
\begin{equation}\label{e33}
T_{n,n+1}^p\leq X^p(u;\eta,4)+X^p(v;\eta,4)+Y^p(n,u,v;\eta)+6.
\end{equation}

It follows easily from Lemma \ref{l24} that there exist $C_3,C_4>0$, such
that for all ${p\in(p(10),p_c)}, {n\in\mathbb{Z}}, {k\in\mathbb{N}}, {r\in [\eta,1/6]}$ and ${u,v\in 2r\mathbb{Z}^2}$,
\begin{equation}\label{e87}
\mathbf{P}_p^{\eta}[Y^p(n,u,v;r)\geq k]\leq C_3\exp(-C_4k).
\end{equation}

For $x\geq 1$ and $n\in\mathbb{Z}$, define a set of pairs of points by
\begin{equation*}
S(n;x):=\{(u,v):u,v\in 2^{-x+1}\mathbb{Z}^2,\mbox{$\Lambda_{2^{-x}}(u)$ intersects $\mathbb{D}_{1/2}(n)$ and $\Lambda_{2^{-x}}(v)$ intersects $\mathbb{D}_{1/2}(n+1)$}\}.
\end{equation*}
It is clear that there is an absolute constant $C_5>0$ such that
\begin{equation}\label{e34}
\# S(n;x)\leq\exp(C_5x).
\end{equation}
Let $K,K_2$ be the constants from Lemma \ref{l10}.  Write
\begin{equation*}
C_6:=\min\{C_4/3,K_2/3\}\quad\mbox{and}\quad C:=\min\{1/(4K),C_6/(2C_5)\}.
\end{equation*}
Then, there is $C_7>0$ such that for all $p\in(p(10),p_c), n\in\mathbb{Z}, k\in\mathbb{N}, u,v\in 2^{-Ck+1}\mathbb{Z}^2$ and ${r^*=r^*(p,k):=\max\{2^{-Ck},\eta\}}$,
\begin{align}
&\mathbf{P}_p^{\eta}\left[X^p(u;r^*,4)+X^p(v;r^*,4)+Y^p(n,u,v;r^*)+6\geq
k\right]\nonumber\\
&\leq\mathbf{P}_p^{\eta}\left[X^p(u;r^*,4)\geq
k/3-2\right]+\mathbf{P}_p^{\eta}\left[X^p(v;r^*,4)\geq
k/3-2\right]+\mathbf{P}_p^{\eta}\left[Y^p(n,u,v;r^*)\geq k/3-2\right]\nonumber\\
&\leq C_7\exp(-C_6k)\quad\mbox{by Lemma \ref{l10} and (\ref{e87})}.\label{e91}
\end{align}

By a standard RSW and FKG argument, there exist
$C_8,C_9>0$, such that for all $n\in\mathbb{Z}$, $p\in(p(10),p_c)$, $x\geq 1$ with $2^{-x}\geq\eta$,
\begin{equation}\label{e13}
\mathbf{P}_p^{\eta}[\diam(\mathcal {C}^p(n))<6\cdot 2^{-x}\mbox{ or }\diam(\mathcal
{C}^p(n+1))<6\cdot 2^{-x}]\leq C_8\exp(-C_9x).
\end{equation}

The above arguments imply that, for $k$ satisfying $k\geq 3/C$ and $2^{-Ck}\geq\eta$,
\begin{align}
&\mathbf{P}_p^{\eta}[T_{n,n+1}^p\geq k]\nonumber\\
&\leq\mathbf{P}_p^{\eta}[\diam(\mathcal {C}^p(n))<6\cdot 2^{-Ck}\mbox{ or }\diam(\mathcal
{C}^p(n+1))<6\cdot 2^{-Ck}]\nonumber\\
&\quad+\mathbf{P}_p^{\eta}[\diam(\mathcal {C}^p(n)),\diam(\mathcal
{C}^p(n+1))\geq 6\cdot 2^{-Ck}\mbox{ and }T_{n,n+1}^p\geq k]\nonumber\\
&\leq C_8\exp(-C_9Ck)+\sum_{(u,v)\in S(n;Ck)}\mathbf{P}_p^{\eta}[X^p(u;r^*,4)+X^p(v;r^*,4)+Y^p(n,u,v;r^*)\geq k]\quad\mbox{by (\ref{e13}) and (\ref{e86})}\nonumber\\
&\leq C_8\exp(-C_9Ck)+\exp(C_5Ck)\cdot C_7\exp(-C_6k)\quad\mbox{by (\ref{e34}) and (\ref{e91})},\label{e15}
\end{align}
and for $k$ satisfying $2^{-Ck}\leq\eta$,
\begin{align}
\mathbf{P}_p^{\eta}[T_{n,n+1}^p\geq k]&\leq\sum_{(u,v)\in S(n;\log_2(1/\eta))}\mathbf{P}_p^{\eta}[X^p(u;\eta,4)+X^p(v;\eta,4)+Y^p(n,u,v;\eta)+6\geq k]\quad\mbox{by (\ref{e33})}\nonumber\\
&\leq\exp(C_5\log_2(1/\eta))\cdot C_7\exp(-C_6k)\quad\mbox{by (\ref{e34}) and (\ref{e91})}\nonumber\\
&\leq\exp(C_5Ck)\cdot C_7\exp(-C_6k).\label{e99}
\end{align}
Then (\ref{e11}) follows from (\ref{e15}) and (\ref{e99}) immediately since $C_5C<C_6$.  We get
(\ref{e48}) from (\ref{e11}) since $T_{m,n}^p\leq\sum_{j=m}^{n-1}T_{j,j+1}^p$.
\end{proof}

To study geometric properties of geodesics between clusters of $\mathscr{C}^p$,  for $z\in\mathbb{C}$, $0<4r\leq R\leq 1$ and $p\in(p(4/r),p_c)$ we define the event
\begin{equation}\label{e93}
\widetilde{\mathcal {A}}_6(z;r,R):=\mathcal
{A}_6(z;r,R-2\eta)\cup\mathcal
{A}_{(111111)}(z;r+2\eta,R)\cup\left\{
\begin{aligned}
&\mbox{there exists $x\in[r+2\eta,R-\eta]$ such that}\\
&\mbox{$\mathcal {A}_6(z;r,x-\eta)\cap\mathcal
{A}_{(111111)}(z;x,R)$ occurs}
\end{aligned}
\right\}.
\end{equation}
Roughly speaking, the above event occurs when there is a geodesic $\gamma$ between two large blue clusters
with two yellow sites of $\gamma$ ``very'' close to each other.  We shall use the event in the proof of Lemma \ref{l6} below; see Figure \ref{yellowcircuit1} for an illustration.
\begin{figure}
\begin{center}
\includegraphics[height=0.3\textwidth]{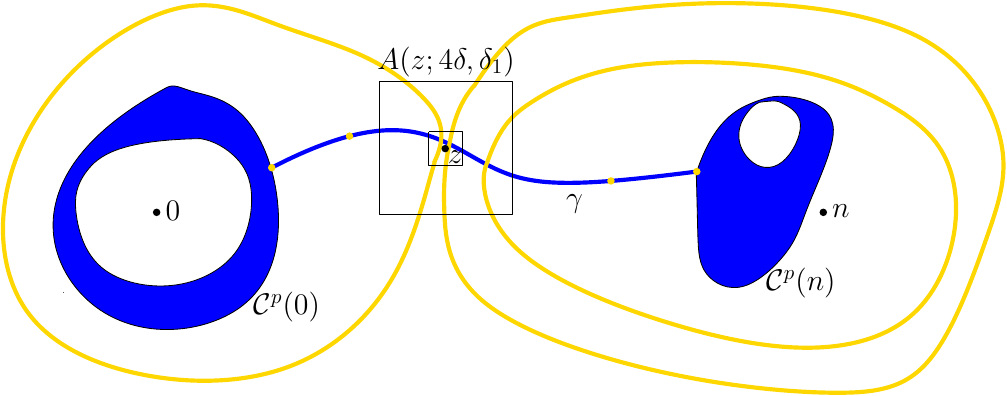}
\caption{Let $0<20\delta<\delta_1<1/4$ and $p\in(p(10/\delta),p_c)$.  Assume that $\diam(\mathcal {C}^p(0)),\diam(\mathcal {C}^p(n))\geq 4\delta_1$ and there is a geodesic $\gamma$ from $\mathcal {C}^p(0)$ to $\mathcal
{C}^p(n)$ with two yellow sites of $\gamma$ contained in $\Lambda_{3\delta}(z)$.  Then the event $\widetilde{\mathcal {A}}_6(z;4\delta,\delta_1)$ occurs.}\label{yellowcircuit1}
\end{center}
\end{figure}
\begin{lemma}\label{l8}
Let $\lambda_6>0$ be the universal constant from (\ref{e67}).  There
is a constant $C>0$ such that for all $z\in\mathbb{C}$,
$0<4r\leq R\leq 1$ and $p\in(p(4/r),p_c)$, we have
\begin{equation*}
\mathbf{P}_p^{\eta}\left[\widetilde{\mathcal
{A}}_6(z;r,R)\right]\leq C\left(\frac{r}{R}\right)^{2+\lambda_6}.
\end{equation*}
\end{lemma}
\begin{proof}
Write $J:=\lfloor\log_2(R/r)\rfloor$.  Assume that there exists
$x\in[r+2\eta,R-\eta]$ such that ${\mathcal
{A}_6(z;r,x-\eta)}\cap{\mathcal {A}_{(111111)}(z;x,R)}$ holds, and
$x$ is the smallest one satisfying this event.  Observe that if
$x\leq 2r$, then $\mathcal {A}_{(111111)}(z;2r,R)$ occurs; if $x\geq
2^{J-1}r$, then $\mathcal {A}_6(z;r,2^{J-1}r-\eta)$ occurs; if
$2^{j-1}r<x\leq 2^jr$ for $j\in\{2,\ldots,J-1\}$, then $\mathcal
{A}_6(z;r,2^{j-1}r-\eta)\cap\mathcal {A}_{(111111)}(z;2^jr,R)$
occurs.  Therefore, there are constants
$\epsilon,C,C_1>0$ such that for all $z\in\mathbb{C}$, $0<4r\leq
R\leq 1$ and $p\in(p(4/r),p_c)$,
\begin{align*}
&\mathbf{P}_p^{\eta}\left[\widetilde{\mathcal
{A}}_6(z;r,R)\right]\\
&\leq\mathbf{P}_p^{\eta}\left[\mathcal
{A}_{(111111)}(z;2r,R)\right]+\mathbf{P}_p^{\eta}\left[\mathcal {A}_6(z;r,2^{J-1}r-\eta)\right]+\sum_{j=2}^{J-1}\mathbf{P}_p^{\eta}\left[\mathcal
{A}_6(z;r,2^{j-1}r-\eta)\right]\mathbf{P}_p^{\eta}\left[\mathcal {A}_{(111111)}(z;2^jr,R)\right]\\
&\leq C_1\left(\frac{r}{R}\right)^{2+\lambda_6}+\sum_{j=2}^{J-1}C_1\left(\frac{r}{R}\right)^{2+\lambda_6}\left(\frac{2^jr}{R}\right)^{\epsilon}
\quad\mbox{by (\ref{e67}) and (\ref{e92})}\\
&\leq C\left(\frac{r}{R}\right)^{2+\lambda_6}.
\end{align*}
\end{proof}

For $n\in\mathbb{N}$, $p\in(p(1),p_c)$, $\delta\in(0,1/2)$,
$\theta\in[0,2\pi]$ and $K\geq 2$, we define the event
\begin{equation*}
\mathcal {F}_n^{p,\theta}(\delta,K):=\left\{
\begin{aligned}
&\exists\mbox{ a chain $\Gamma$ from $\mathcal {C}^p(0)$ to
$\mathcal {C}^p(ne^{i\theta})$ such that $T_{0,n}^{p,\theta}=|\Gamma|-1$},\\
&\mbox{$\diam(\mathcal {C}^p)\geq\delta$ for all $\mathcal
{C}^p\in\Gamma$, and $\Gamma\subset\Lambda_{Kn}$}
\end{aligned}
\right\}.
\end{equation*}
Write $\mathcal {F}_n^p(\delta,K):=\mathcal {F}_n^{p,0}(\delta,K)$.  Observe that for our FPP on $\eta\mathbb{T}$, if there is a geodesic between two large blue clusters with two yellow sites in this geodesic close to each other, then the blue cluster between these two yellow sites may be small and not be seen in the scaling limit.  This creates a difference between the first-passage times in the discrete and continuum cases, since the continuum FPP is defined by chains of clusters.  Nevertheless, thanks to the following lemma, for $p$ close to $p_c$ with high probability the first-passage time between two large blue clusters is realized by a discrete chain of large clusters.  The proof relies on the estimate of the ``6-arm'' event in the previous lemma.
\begin{lemma}\label{l6}
For each $R\geq 1$ and $\epsilon>0$, there exist $K=K(\epsilon)\geq 3$
(independent of  $R$), $\delta=\delta(\epsilon,R)\in(0,1/4)$, such
that for all $\theta\in[0,2\pi]$, $p\in(p(10/\delta),p_c)$ and $n\in\mathbb{N}$ with $n\leq R$,
\begin{equation*}
\mathbf{P}_p^{\eta}[\mathcal{F}_n^{p,\theta}(\delta,K)]\geq
1-\epsilon.
\end{equation*}
\end{lemma}
\begin{proof}
For simplicity, we prove Lemma \ref{l6} in the case $\theta=0$.  The proof extends immediately to the general case.

Let $C_0$ be the constant from (\ref{e48}).  We use (\ref{e48}) and
Markov's inequality to obtain that, for all $x>0$, $n\in\mathbb{N}$
and $p\in(p(10),p_c)$,
\begin{equation}\label{e17}
\mathbf{P}_p^{\eta}[T_{0,n}^p\geq xC_0n]\leq \frac{1}{x}.
\end{equation}

We claim that there are constants $\epsilon_0,C_1,C_2>0$ such that
for each $p\in(p(10),p_c)$, $r\geq 1$ and the event
\begin{equation*}
\mathcal{G}(r):=\{\exists\mbox{ a path $\gamma$ of
$\omega_p^{\eta}$ starting at a hexagon in $\mathbb{D}_{1/2}$, such
that $T(\gamma)\leq\epsilon_0r$ and }\gamma\not\subset\Lambda_r\},
\end{equation*}
we have
\begin{equation}\label{e18}
\mathbf{P}_p^{\eta}[\mathcal{G}(r)]\leq C_1\exp(-C_2r).
\end{equation}
We use a standard renormalization argument to show (\ref{e18}).  First,
we define a family of random variables $\{X_z^p: z\in
\mathbb{Z}^2\}$ as follows.  We declare a vertex $z\in\mathbb{Z}^2$
to be \textbf{good} if there is a yellow circuit (viewed as a union of yellow hexagons) of
$\omega_p^{\eta}$ surrounding the point $Mz$ in $A(Mz;M/2,M)$, where $M$ is a large
constant that will be fixed later.  The family $X^p=\{X_z^p: z\in
\mathbb{Z}^2\}$ is defined by
\begin{equation*}
X_z^p:=\left\{
\begin{aligned}
&1 &\mbox{ if $z$ is good},\\
&0 &\mbox{ otherwise}.
\end{aligned}
\right.
\end{equation*}
Note that $X^p$ is 2-dependent.  It follows by (\ref{e40}) that
\begin{equation}\label{e3}
\lim_{M\rightarrow\infty}\mathbf{P}_p^{\eta}[z \mbox{ is good}]=1\quad\mbox{uniformly in $p\in(p(10),p_c)$ and $z\in \mathbb{Z}^2$.}
\end{equation}
Then, by Theorem \ref{t8}, there is a $p_0\in(0,p_c^{site}(\mathbb{Z}^2))$ and a fixed constant $M\geq 1$, such that for
each $p\in(p(10),p_c)$, $X^p$ stochastically dominates the subcritical
Bernoulli site FPP on $\mathbb{Z}^2$ with measure
$\mathbf{P}_{\mathbb{Z}^2,p_0}^{site}$. (Recall that $\mathbf{P}_{\mathbb{Z}^2,p_0}^{site}[t(z)=0]=p_0=1-\mathbf{P}_{\mathbb{Z}^2,p_0}^{site}[t(z)=1]$ for $z\in \mathbb{Z}^2$.)

Assume that the event $\mathcal{G}(r)$ holds.  Then it is easy to
see that there exists a path $\widetilde{\gamma}$ of $\mathbb{Z}^2$
starting at the origin, such that $|\widetilde{\gamma}|\geq r/M$ and
$\Lambda_{M/2}(Mz)\cap\gamma\neq\emptyset$ for all
$z\in\widetilde{\gamma}$.  Note that when a vertex
$z\in\widetilde{\gamma}$ is good, then $\gamma$ passes through a
yellow hexagon of $\omega_p^{\eta}$ in $A(Mz;M/2,M)$.  Therefore,
\begin{equation*}
T(\gamma)\geq(\mbox{the number of good vertices in }
\widetilde{\gamma})/9.
\end{equation*}
The above argument implies that there are constants
$\epsilon_0,C_1,C_2>0$ such that for each $p\in(p(10),p_c)$ and $r\geq 1$,
\begin{align*}
\mathbf{P}_p^{\eta}[\mathcal{G}(r)]&\leq\mathbf{P}_p^{\eta}\left[
\begin{aligned}
&\mbox{$\exists$ a path $\widetilde{\gamma}$ of
$\mathbb{Z}^2$ starting at the origin, such that $|\widetilde{\gamma}|\geq r/M$}\\
&\mbox{and the number of good vertices in $\widetilde{\gamma}$ is
smaller than $10\epsilon_0r$}
\end{aligned}
\right]\\
&\leq\mathbf{P}_{\mathbb{Z}^2,p_0}^{site}\left[
\begin{aligned}
&\mbox{$\exists$ a path $\widetilde{\gamma}$ of
$\mathbb{Z}^2$ starting at the origin, such that $|\widetilde{\gamma}|\geq r/M$}\\
&\mbox{and the number of 1-valued vertices in $\widetilde{\gamma}$
is smaller than $10\epsilon_0r$}
\end{aligned}
\right]\\
&\leq C_1\exp(-C_2r)\quad\mbox{by Proposition \ref{p1}},
\end{align*}
which gives the claim (\ref{e18}).

Let $\mathcal{O}(r)$ denote the event that there exists a yellow
circuit of $\omega_p^{\eta}$ surrounding the origin in
$A(r/2,r)$.  Then for each $\epsilon\in(0,1)$ there is
$K\geq 3$ such that for all $n\in\mathbb{N}$, $p\in(p(10),p_c)$ and the event
\begin{align*}
\mathcal{H}_n(K):=\left\{
\begin{aligned}
&\mbox{for each geodesic $\gamma$ from $\mathcal {C}^p(0)$ to
$\mathcal {C}^p(n)$,
$\gamma\subset\Lambda_{Kn/2}$ and}\\
&\mbox{each blue cluster which has a site of $\gamma$ is contained
in $\Lambda_{Kn}$}
\end{aligned}
\right\},
\end{align*}
we have
\begin{align}
\mathbf{P}_p^{\eta}[\mathcal{H}_n(K)]&\geq
\mathbf{P}_p^{\eta}[\{T_{0,n}^p\leq\epsilon_0Kn/2\}\cap
\mathcal{G}(Kn/2)\cap\mathcal{O}(Kn)]\nonumber\\
&\geq 1-\epsilon/3\quad\mbox{by (\ref{e17}), (\ref{e18}) and
(\ref{e40})}.\label{e20}
\end{align}

It follows from a RSW and FKG argument that, for each
$\epsilon\in(0,1)$ there is $\delta_1\in(0,1/4)$ such that for
each $p\in(p(10/\delta_1),p_c)$, $n\in\mathbb{N}$ and
the event
\begin{equation*}
\mathcal{E}_n(\delta_1):=\{\diam(\mathcal {C}^p(0)),\diam(\mathcal {C}^p(n))\geq 4\delta_1\},
\end{equation*}
we have
\begin{equation}\label{e19}
\mathbf{P}_p^{\eta}[\mathcal{E}_n(\delta_1)]\geq 1-\epsilon/3.
\end{equation}
For $\delta\in(0,\delta_1/20)$ and $p\in(p(10/\delta),p_c)$,  write
\begin{equation*}
\mathcal{S}_n(\delta):=\left\{
\begin{aligned}
&\mbox{for each geodesic $\gamma$ from $\mathcal {C}^p(0)$ to $\mathcal {C}^p(n)$, the $L^{\infty}$ distance}\\
&\mbox{between any pair of yellow sites in $\gamma$ is larger than
$2\delta$}
\end{aligned}
\right\}.
\end{equation*}
Note that $\mathcal{F}_n^p(\delta,K)\supset\mathcal{H}_n(K)\cap
\mathcal{E}_n(\delta_1)\cap \mathcal{S}_n(\delta)$.  It remains to
bound the probability of the event $\mathcal{H}_n(K)\cap
\mathcal{E}_n(\delta_1)\cap \neg\mathcal{S}_n(\delta)$.  Assume that
this event holds.  Then there exist two yellow sites $v_1,v_2$ in a
geodesic $\gamma$ from $\mathcal {C}^p(0)$ to $\mathcal
{C}^p(n)$, such that $\|v_1-v_2\|_{\infty}\leq 2\delta$ and
$v_1,v_2\in\Lambda_{Kn/2}$.  It is clear that there is $z\in
2\delta\mathbb{Z}^2\cap\Lambda_{Kn/2}$ such that
$v_1,v_2\in\Lambda_{3\delta}(z)$.  By (\ref{i2}) of Proposition \ref{p4}, the following statements hold: If there
exist no yellow sites of $\gamma$ in $\Lambda_{\delta_1}(z)$ except
$v_1,v_2$, then the event $\mathcal {A}_6(z;4\delta,\delta_1-\eta)$ occurs.  Otherwise, there exists a yellow site
$v\in\gamma\backslash\{v_1,v_2\}$ which is contained in
$\Lambda_{\delta_1}(z)$ with minimal $L^{\infty}$ distance to $z$.
Write $L:=\|z-v\|_{\infty}$.  If $L<4\delta+2\eta$, then
$\mathcal {A}_{(111111)}(z;4\delta+2\eta,\delta_1)$ occurs; if
$L>\delta_1-\eta$, then $\mathcal
{A}_6(z;4\delta,\delta_1-2\eta)$ occurs; if
$L\in[4\delta+2\eta,\delta_1-\eta]$, then $\mathcal
{A}_6(z;4\delta,L-\eta)\cap\mathcal {A}_{(111111)}(z;L,\delta_1)$
occurs.  The above argument implies that the event
$\widetilde{\mathcal {A}}_6(z;4\delta,\delta_1)$ defined by
(\ref{e93}) occurs; see Figure \ref{yellowcircuit1}.  Therefore, there exist universal constants
$\lambda_6,C_3>0$, such that for each $R\geq 1$ and $\epsilon>0$, there
exist $\delta_1=\delta_1(\epsilon)\in(0,1/4)$ (independent of $R$),
$K=K(\epsilon)\geq 3$ (independent of $R$) and $\delta=\delta(\epsilon,R)\in(0,\delta_1/20)$, such that for all
$p\in(p(10/\delta),p_c)$ and $n\in\mathbb{N}$ with
$n\leq R$,
\begin{align*}
\mathbf{P}_p^{\eta}[\mathcal{H}_n(K)\cap
\mathcal{E}_n(\delta_1)\cap
\neg\mathcal{S}_n(\delta)]&\leq\sum_{z\in
2\delta\mathbb{Z}^2\cap\Lambda_{Kn/2}}\mathbf{P}_p^{\eta}\left[\widetilde{\mathcal
{A}}_6(z;4\delta,\delta_1)\right]\\
&\leq C_3\left(\frac{KR}{\delta}\right)^2\left(\frac{\delta}{\delta_1}\right)^{2+\lambda_6}\quad\mbox{by Lemma \ref{l8}}\\
&\leq\epsilon/3.
\end{align*}
Combining this with (\ref{e20}) and (\ref{e19}), we obtain
\begin{equation*}
\mathbf{P}_p^{\eta}[\mathcal{F}_n^p(\delta,K)]\geq\mathbf{P}_p^{\eta}[\mathcal{H}_n(K)\cap
\mathcal{E}_n(\delta_1)\cap \mathcal{S}_n(\delta)]\geq 1-\epsilon.
\end{equation*}
\end{proof}

The next lemma says that for fixed $n,\theta$ and each $p<p_c$ sufficiently close to $p_c$, with high probability $T_{0,n}^{\theta}$ equals $T^{p,\theta}_{0,n}$.
\begin{lemma}\label{l29}
Let $\mathbf{P}$ be a coupling such that
$\omega_p^{\eta}\rightarrow\omega^{\infty}$ in
$(\mathscr{H},d_{\mathscr{H}})$ a.s. as $p\uparrow p_c$.  For each
$R\geq 1$ and $\epsilon>0$, there exists $p_0=p_0(\epsilon,R)\in(0,p_c)$ such that for all $p\in(p_0,p_c)$,
$\theta\in[0,2\pi]$ and $n\in \mathbb{N}$ with $n\leq R$, we have
\begin{equation*}
\mathbf{P}[T_{0,n}^{\theta}=T_{0,n}^{p,\theta}]\geq 1-\epsilon.
\end{equation*}
\end{lemma}
\begin{proof}
We start with an upper bound on the
probability of the event that there exist two large blue clusters
in a box such that they are close to each other but do
not form a chain.  For $S,S'\subset\mathbb{C}$, let ${d_{\infty}(S,S'):=\inf\{\|z-z'\|_{\infty}:z\in S,z'\in S'\}}$ denote the $L^{\infty}$-distance between $S$ and $S'$.  There exist $C_1,\lambda_6>0$
such that for all $0<5r_1\leq r_2\leq 1$, $k\in\mathbb{N}$ and
$p\in(p(5/r_1),p_c)$,
\begin{align}
&\mathbf{P}\left[\mbox{$\exists$ $\mathcal {C}_1^p,\mathcal
{C}_2^p\in\mathscr{C}^p(r_2)$ in $\Lambda_k$ such that
$0<d_{\infty}(\mathcal {C}_1^p,\mathcal {C}_2^p)\leq r_1$ and
$(\mathcal {C}_1^p,\mathcal {C}_2^p)$ is not a chain}\right]\nonumber\\
&\leq\mathbf{P}_p^{\eta}[\mathcal
{A}_6(z;r_1,r_2/4)\mbox{ occurs for some point $z\in\Lambda_k$}]\nonumber\\
&\leq
C_1\left(\frac{k}{r_1}\right)^2\left(\frac{r_1}{r_2}\right)^{2+\lambda_6}\mbox{
by (\ref{e67})}.\label{e47}
\end{align}
For the continuum configuration the probability of the
corresponding event is bounded as follows:  Since for any fixed $r_2>0$ and $k\in\mathbb{N}$, there are a.s. finitely many clusters of $\mathscr{C}(r_2)$ in $\Lambda_k$ (by Theorem \ref{t3}), we know that for any fixed $r_2>0$ and $k\in\mathbb{N}$,
\begin{equation}\label{e77}
\mathbf{P}[\exists\mbox{ $\mathcal {C}_1,\mathcal
{C}_2\in\mathscr{C}(r_2)$ in $\Lambda_k$ such that
$0<d_{\infty}(\mathcal {C}_1,\mathcal {C}_2)\leq r_1$}]\rightarrow 0\quad\mbox{as $r_1\rightarrow 0$}.
\end{equation}

Let $\mathcal {F}_n^{p,\theta}(\delta,K)$ be the event defined above
Lemma \ref{l6}.  Write
\begin{equation*}
\mathcal {E}_n^{p,\theta}(\delta,K):=\left\{
\begin{aligned}
&\mbox{$\exists$ a chain $\Gamma^p$ from $\mathcal
{C}^p(0)$ to $\mathcal {C}^p(ne^{i\theta})$, such that $T_{0,n}^{p,\theta}=|\Gamma^p|-1$},\\
&\mbox{$\diam(\mathcal {C}^p)\geq\delta$ for all $\mathcal
{C}^p\in\Gamma^p$, and $\Gamma^p\subset\Lambda_{Kn}$;}\\
&\mbox{$\exists$ a chain $\Gamma$ from $\mathcal {C}(0)$ to
$\mathcal {C}(ne^{i\theta})$, such that
$|\Gamma|=|\Gamma^p|$},\\
&\mbox{$\diam(\mathcal {C})\geq\delta/2$ for all $\mathcal
{C}\in\Gamma$, and $\Gamma\subset\Lambda_{Kn}$}
\end{aligned}
\right\}
\end{equation*}
It is clear that $\mathcal
{E}_n^{p,\theta}(\delta,K)\subset\{T_{0,n}^{\theta}\leq
T_{0,n}^{p,\theta}\}$.  Therefore, for each $R\geq 1$ and $\epsilon>0$,
there exist $K=K(\epsilon)\geq 3$ with $K\in\mathbb{N}$,
$\delta=\delta(\epsilon,R)\in(0,1/4)$,
$\delta_1=\delta_1(\epsilon,R)\in(0,\delta/100)$ and
$p_1=p_1(\epsilon,R)\in(0,p_c)$, such that for all $p\in(p_1,p_c)$,
$\theta\in[0,2\pi]$ and $n\in\mathbb{N}$ with $n\leq R$, we have
\begin{align}
\mathbf{P}[T_{0,n}^{\theta}\leq
T_{0,n}^{p,\theta}]&\geq\mathbf{P}[\mathcal {E}_n^{p,\theta}(\delta,K)]\nonumber\\
&\geq\mathbf{P}[\mathcal
{F}_n^{p,\theta}(\delta,K)]-\mathbf{P}[d_H(\mathcal
{C}^p(0),\mathcal {C}(0))\geq\delta_1]-\mathbf{P}[d_H(\mathcal
{C}^p(ne^{i\theta}),\mathcal {C}(ne^{i\theta}))\geq\delta_1]\nonumber\\
&\quad{}-\mathbf{P}[\dist(\mathscr{C}^p|_{\Lambda_{Kn}},\mathscr{C}|_{\Lambda_{Kn}})\geq\delta_1]\nonumber\\
&\quad{}-\mathbf{P}[\exists\mbox{ $\mathcal {C}_1,\mathcal
{C}_2\in\mathscr{C}(\delta/4)$ in $\Lambda_{KR}$ such that
$0<d_{\infty}(\mathcal {C}_1,\mathcal {C}_2)\leq 3\delta_1$}]\nonumber\\
&\geq 1-\epsilon/2\quad\mbox{by (\ref{e77}), Theorem \ref{t3},
Lemmas \ref{l6} and \ref{l37}},\label{e42}
\end{align}
where the second inequality is due to the observation:
\begin{align*}
\mathcal {E}_n^{p,\theta}(\delta,K)\supset &\mathcal
{F}_n^{p,\theta}(\delta,K)\cap\{d_H(\mathcal {C}^p(0),\mathcal
{C}(0))\leq\delta_1\}\cap\{d_H(\mathcal
{C}^p(ne^{i\theta}),\mathcal {C}(ne^{i\theta}))\leq\delta_1\}\\
&\cap\{\dist(\mathscr{C}^p|_{\Lambda_{Kn}},\mathscr{C}|_{\Lambda_{Kn}})\leq\delta_1\}\cap\{\nexists\mbox{ $\mathcal {C}_1,\mathcal
{C}_2\in\mathscr{C}(\delta/4)$ in $\Lambda_{KR}$ such that
$0<d_{\infty}(\mathcal {C}_1,\mathcal {C}_2)\leq 3\delta_1$}\}.
\end{align*}

It remains to bound $\mathbf{P}[T_{0,n}^{p,\theta}\leq T_{0,n}^{\theta}]$ from below.
Combining (\ref{e42}) and Lemma \ref{l7}, we get that for each fixed
$\theta\in[0,2\pi]$ and $n\in \mathbb{N}$, $T_{0,n}^{\theta}$ is
almost surely finite.  Moreover, by Theorem \ref{t3} we known that
all the continuum clusters of $\mathscr{C}$ are almost surely bounded and the distribution of
$\mathscr{C}$ is invariant under rotations.  Therefore, for each
$R\geq 1$ and $\epsilon>0$, there is an integer $N=N(\epsilon,R)\geq
2R$ and a $\delta_0=\delta_0(\epsilon,R)\in(0,1/4)$ such that for all $\theta\in[0,2\pi]$, $n\in\mathbb{N}$ with $n\leq R$, and the event
\begin{equation*}
\mathcal {G}_n^{\theta}(\delta_0,N):=\left\{
\begin{aligned}
&\mbox{$\exists$ a chain $\Gamma$ from $\mathcal {C}(0)$
to $\mathcal {C}(ne^{i\theta})$, such that $T_{0,n}^{\theta}=|\Gamma|-1$},\\
&\mbox{$\diam(\mathcal {C})\geq\delta_0$ for all $\mathcal
{C}\in\Gamma$, and $\Gamma\subset\Lambda_N$}
\end{aligned}
\right\},
\end{equation*}
we have
\begin{equation}\label{e43}
\mathbf{P}[\mathcal {G}_n^{\theta}(\delta_0,N)]\geq 1-\epsilon/4.
\end{equation}
Then, similarly to the proof of (\ref{e42}), for each $R\geq
1$ and $\epsilon>0$ there exist $N=N(\epsilon,R)\geq 2R$ with $N\in\mathbb{N}$,
$\delta_0=\delta_0(\epsilon,R)\in(0,1/4)$,
$\delta_2=\delta_2(\epsilon,R)\in(0,\delta_0/100)$ and
$p_2=p_2(\epsilon,R)\in(0,p_c)$, such that for all $p\in(p_2,p_c)$,
$\theta\in[0,2\pi]$ and $n\in\mathbb{N}$ with $n\leq R$,
\begin{align}
&\mathbf{P}[T_{0,n}^{p,\theta}\leq T_{0,n}^{\theta}]\nonumber\\
&\geq\mathbf{P}\left[
\begin{aligned}
&\mbox{$\mathcal {G}_n^{\theta}(\delta_0,N)$ occurs, and there is a
chain $\Gamma^p$ from $\mathcal {C}^p(0)$ to $\mathcal
{C}^p(ne^{i\theta})$,}\\
&\mbox{such that $|\Gamma^p|=|\Gamma|$, $\diam(\mathcal
{C}^p)\geq\delta_0/2$ for all $\mathcal {C}^p\in\Gamma^p$, and
$\Gamma^p\subset\Lambda_N^{\theta}$}
\end{aligned}
\right]\nonumber\\
&\geq\mathbf{P}[\mathcal
{G}_n^{\theta}(\delta_0,N)]-\mathbf{P}[d_H(\mathcal
{C}^p(0),\mathcal {C}(0))\geq\delta_2]-\mathbf{P}[d_H(\mathcal
{C}^p(ne^{i\theta}),\mathcal {C}(ne^{i\theta}))\geq\delta_2]-\mathbf{P}[\dist(\mathscr{C}^p|_{\Lambda_N},\mathscr{C}|_{\Lambda_N})\geq\delta_2]\nonumber\\
&\quad{}-\mathbf{P}[\exists\mbox{ $\mathcal {C}_1^p,\mathcal
{C}_2^p\in\mathscr{C}^p(\delta_0/4)$ in $\Lambda_N$ such that
$0<d_{\infty}(\mathcal {C}_1^p,\mathcal {C}_2^p)\leq 3\delta_2$ and $(\mathcal {C}_1^p,\mathcal {C}_2^p)$ is not a chain}]\nonumber\\
&\geq 1-\epsilon/2\quad\mbox{by (\ref{e43}), Lemma \ref{l37},
Theorem \ref{t3} and (\ref{e47})}.\label{e44}
\end{align}
Then Lemma \ref{l29} follows from a combination of inequalities
(\ref{e44}) and (\ref{e42}).
\end{proof}

\begin{lemma}\label{l3}
Let $C_0>0$ be the absolute constant in Lemma \ref{l7}.  We have
\begin{equation}\label{e94}
\mathbf{E}^{\infty}[T_{0,1}]\leq C_0.
\end{equation}
Furthermore, there is a constant $C_1>0$ such that
for each $\epsilon>0$, there exists $N=N(\epsilon)>0$ such that for
all $n\in\mathbb{N}$ with $n>N$,
\begin{equation}\label{e98}
\mathbf{P}^{\infty}[T_{0,n}\geq C_1n]\geq 1-\epsilon.
\end{equation}
\end{lemma}
\begin{proof}
Combining Lemmas \ref{l7} and \ref{l29} gives (\ref{e94}).

Inequality (\ref{e18}) implies that there exists a constant
$C_1>0$ such that for each $\epsilon>0$, there exists
$N=N(\epsilon)>0$ such that for all $p\in(p(10),p_c)$ and $n\in\mathbb{N}$ with $n>N$,
\begin{equation*}
\mathbf{P}_p^{\eta}[T_{0,n}^p\geq C_1n]\geq 1-\epsilon/2.
\end{equation*}
Then (\ref{e98}) follows from this and Lemma \ref{l29}.
\end{proof}

The following is a law of large numbers for the point-to-point
passage times of our continuum FPP.
\begin{proposition}\label{p2}
Suppose $n\in\mathbb{N}$.  There exists a constant $\nu>0$ such
that for any fixed $\theta\in[0,2\pi]$,
\begin{equation*}
\lim_{n\rightarrow\infty}\frac{T_{0,n}^{\theta}}{n}=\nu\quad\mbox{
$\mathbf{P}^{\infty}$-a.s.
and}\quad\lim_{n\rightarrow\infty}\frac{\mathbf{E}^{\infty}T_{0,n}^{\theta}}{n}=\inf_n\frac{\mathbf{E}^{\infty}T_{0,n}^{\theta}}{n}=\nu.
\end{equation*}
\end{proposition}

\begin{proof}
Since the law of $\mathscr{C}$ is invariant under rotations by Theorem \ref{t3}, all the families $(T_{0,n}^{\theta})_{n\in\mathbb{N}}$ for $\theta\in[0,2\pi]$ have the same distribution.  Thus, to prove Proposition \ref{p2}, it suffices to prove it in the case $\theta=0$.  Suppose that $m,n\in\mathbb{Z}_+$ with $m<n$.  We verify that the
family of random variables $(T_{m,n})_{0\leq m<n}$ satisfies the
conditions of the subadditive ergodic theorem (see, e.g.,
\cite{Lig85} or Theorem 2.2 in \cite{ADH17}):
\begin{itemize}
\item $T_{0,n}\leq T_{0,m}+T_{m,n}$ for all $0<m<n$.\\
Note that a concatenation of chains from $\mathcal {C}(0)$ to
$\mathcal {C}(m)$ and from $\mathcal {C}(m)$ to $\mathcal {C}(n)$
yields a chain from $\mathcal {C}(0)$ to $\mathcal {C}(n)$.
Combining this observation and the definition of $T_{m,n}$ gives the
above triangle inequality.
\item The distributions of the sequences $(T_{m,m+j})_{j\geq 1}$ and $(T_{m+1,m+j+1})_{j\geq 1}$
are the same for all $m\geq 0$.\\
The law of $\mathscr{C}$ is translation invariant by Theorem
\ref{t3}, which implies this immediately.
\item $(T_{nj,(n+1)j})_{n\geq 1}$ is a stationary ergodic
sequence for each $j\geq 1$.\\
Define the horizontal shifts of the plane $\tau_j$:
$\mathbb{C}\rightarrow \mathbb{C}, z\mapsto z-j$.  Then
$T_{nj,(n+1)j}(\mathscr{C})=T_{0,j}(\tau_{nj}\mathscr{C})$.  By
Theorem \ref{t3}, the law of $\mathscr{C}$ is invariant under
translations, so $\tau_j$ is measure preserving and
$(T_{nj,(n+1)j})_{n\geq 1}$ is stationary.  Next we show that
$\tau_j$ is also mixing, which implies that $(T_{nj,(n+1)j})_{n\geq
1}$ is ergodic.  When $\mathcal {A},\mathcal {B}$ are events which
depend only on the realization of $\mathscr{C}$ inside some box
$\Lambda_k$, then $\lim_{n\rightarrow\infty}\mathbb{P}[\mathcal
{A}\cap\tau_j^{-n}\mathcal {B}]=\mathbb{P}[\mathcal
{A}]\mathbb{P}[\mathcal {B}]$ follows immediately since the events
$\mathcal {A}$ and $\tau_j^{-n}\mathcal {B}$ are independent for
large $n$ by Theorem \ref{t3}.  For arbitrary events $\mathcal {A}$
and $\mathcal {B}$ depending on $\mathscr{C}$, one approximates
$\mathcal {A}$ and $\mathcal {B}$ by events which depend only on the
realization of $\mathscr{C}$ inside $\Lambda_k$, with
$k\rightarrow\infty$.  Then the result follows easily.
\item $\mathbb{E}[T_{0,1}]<\infty$ and there is a constant $C>0$ such that for each $n$, $\mathbb{E}[T_{0,n}]>-Cn$.\\
$\mathbb{E}[T_{0,1}]<\infty$ follows from (\ref{e94}) and the
rest of this item is obvious since $T_{0,n}$ is non-negative by
definition.
\end{itemize}
Then by the subadditive ergodic theorem and (\ref{e98}) we obtain
Proposition \ref{p2}, where we use (\ref{e98}) to show that the limit $\nu$ is positive.
\end{proof}

\section{Convergence of normalized time constants}\label{global}
Let $\nu$ be as in Proposition \ref{p2}, which is the ``time constant'' for the continuum FPP.
First, in Section \ref{upper}, we show that for each
fixed $u\in\mathbb{U}$, $\nu$ is an upper bound for the upper limit of $L(p)\mu(p,u)$.  Next, in Section \ref{lower}, we show that $\nu$ is also a lower bound
for the lower limit of $L(p)\mu(p,u)$.  Finally, in Section \ref{final} we prove our main result, Theorem \ref{t2}, which is an immediate consequence of the above two results.

The basic idea is as follows.  The results in Section \ref{fpp} give that
the discrete FPP on $\eta\mathbb{T}$ in any fixed box is ``well-approximated'' by the corresponding continuum FPP, as $p\uparrow p_c$.
This, combined with appropriate renormalization arguments, implies that the approximation also holds in the whole plane.
\subsection{Upper bound}\label{upper}
The goal of this subsection is to prove:
\begin{lemma}\label{l16}
For each $u\in\mathbb{U}$, we have $\limsup_{p\uparrow
p_c}L(p)\mu(p,u)\leq\nu$.
\end{lemma}
To show this, we need the following result on passage times
of paths constrained in a box.  For $x,y\in\mathbb{C}$ and $r>0$, define
\begin{equation*}
Box(x,y;r):=\left\{z\in\mathbb{C}:\mbox{$z$ is within $L^{\infty}$-distance $r$ of the straight line segment joining $x$ to $y$}\right\}.
\end{equation*}
\begin{lemma}\label{l9}
There exists $K\in\mathbb{N}$ such that for each $\epsilon>0$, there exists $N\in\mathbb{N}$ such that for all $n\geq N$, there exists $p_0\in(0,p_c)$ such that for all $p\in(p_0,p_c)$, $\theta\in[0,2\pi]$ and $z\in\mathbb{C}$,
\begin{equation*}
\mathbf{P}_p^{\eta}\left[\mbox{there is a path $\gamma$ from $\mathcal {C}^p(z)$
to $\mathcal {C}^p(z+ne^{i\theta})$ in $Box(z,z+ne^{i\theta};Kn)$ s.t.
$T(\gamma)\leq(\nu+\epsilon)n$}\right]\geq 1-\epsilon.
\end{equation*}
\end{lemma}
\begin{proof}
Since the law of $\mathscr{C}$ is invariant under rotations by Theorem \ref{t3}, for any fixed $n$ the passage times $T_{0,n}^{\theta}$ for $\theta\in[0,2\pi]$ have the same distribution.  This and Proposition \ref{p2} imply that for each $\epsilon>0$, there exists
$N_1\in\mathbb{N}$ such that for all $n\geq
N_1$ and $\theta\in[0,2\pi]$,
\begin{equation*}
\mathbf{P}^{\infty}[T_{0,n}^{\theta}\leq(\nu+\epsilon)n]\geq 1-\epsilon/5.
\end{equation*}
Combining the above inequality and Lemma \ref{l29}, we obtain that, for
each $\epsilon>0$, there exists $N_1\in\mathbb{N}$ such
that for all $n\geq N_1$, there exists
$p_1\in(p(10),p_c)$ such that for all $p\in(p_1,p_c)$ and $\theta\in[0,2\pi]$,
\begin{equation}\label{e24}
\mathbf{P}_p^{\eta}[T_{0,n}^{p,\theta}\leq(\nu+\epsilon)n]\geq 1-\epsilon/4.
\end{equation}
By (\ref{e18}), there exists $K_1\in\mathbb{N}$ such that for each $\epsilon\in(0,1)$, there exists $N_2\in\mathbb{N}$ such that
for all $n\geq N_2$, $p\in(p(10),p_c)$ and $\theta\in[0,2\pi]$,
\begin{equation}\label{e25}
\mathbf{P}_p^{\eta}[\exists\mbox{ a path $\gamma$ starting at a site in $\mathbb{D}_{1/2}$, such
that $T(\gamma)\leq (\nu+1)n$ and }\gamma\not\subset Box(0,ne^{i\theta};K_1n)]\leq\epsilon/4.
\end{equation}
Then (\ref{e24}) and (\ref{e25}) imply that there is $K_1\in\mathbb{N}$ such that for each $\epsilon\in(0,1)$, there exists $N=\max\{N_1,N_2\}$ such that for all $n\geq N$, there exists $p_1\in(p(10),p_c)$ such that for all $p\in(p_1,p_c)$ and $\theta\in[0,2\pi]$,
\begin{equation}\label{e114}
\mathbf{P}_p^{\eta}\left[\mbox{$\exists$ a path $\gamma$ from $\mathcal {C}^p(0)$
to $\mathcal {C}^p(ne^{i\theta})$ in $Box(0,ne^{i\theta};K_1n)$ such that
$T(\gamma)\leq(\nu+\epsilon)n$}\right]\geq 1-\epsilon/2.
\end{equation}
We need to generalize this result to pairs of points $z,z+ne^{i\theta}$ for $z\in\mathbb{C}$ and $\theta\in[0,2\pi]$.  It is easy to see that for each $\epsilon>0$,
there exist $\delta\in(0,1/2)$ and $p_2\in(p(10/\delta),p_c)$ such that for all $p\in(p_2,p_c)$ and $x\in\mathbb{C}$,
\begin{align}
&\mathbf{P}_p^{\eta}[\mathcal {C}^p(y)=\mathcal {C}^p(x)\mbox{
for all $y\in\mathbb{D}_{\eta}(x)$}]\nonumber\\
&\geq\mathbf{P}_p^{\eta}\left[
\begin{aligned}
&\mbox{$\diam(\mathcal {C}^p(x))\geq\delta$ and there is a unique open cluster having diameter $\diam(\mathcal {C}^p(x))$ in $\mathbb{D}_{1/2}(x)$;}\\
&\mbox{all the open clusters which have diameters at least $\delta$ in $\mathbb{D}_{1/2+\eta}(x)$ lie in $\mathbb{D}_{1/2-\eta}(x)$}
\end{aligned}
\right]\nonumber\\
&\geq 1-\epsilon/4\quad\mbox{by FKG, RSW, Proposition \ref{p3} and
(\ref{e72}).}\label{e113}
\end{align}
For $z\in\mathbb{C}$, let $\tilde{z}$ denote the site of
$\eta\mathbb{T}$ closest to $z$.  Then, there exists $K=K_1+1\in\mathbb{N}$ such that for each $\epsilon>0$, there exists $N\in\mathbb{N}$ such that for all $n\geq N$, there exists $p_0=\max\{p_1,p_2\}\in(0,p_c)$ such that for all $p\in(p_0,p_c)$, $\theta\in[0,2\pi]$ and $z\in\mathbb{C}$,
\begin{align*}
&\mathbf{P}_p^{\eta}\left[\mbox{there is a path $\gamma$ from $\mathcal {C}^p(z)$
to $\mathcal {C}^p(z+ne^{i\theta})$ in $Box(z,z+ne^{i\theta};Kn)$ such that
$T(\gamma)\leq(\nu+\epsilon)n$}\right]\\
&\geq \mathbf{P}_p^{\eta}\left[
\begin{aligned}
&\mbox{$\mathcal {C}^p(\tilde{z})=\mathcal {C}^p(z)$, $\mathcal
{C}^p(\tilde{z}+ne^{i\theta})=\mathcal {C}^p(z+ne^{i\theta})$, and there is a path $\gamma$ from}\\
&\mbox{$\mathcal {C}^p(\tilde{z})$ to $\mathcal {C}^p(\tilde{z}+ne^{i\theta})$ in $Box(\tilde{z},\tilde{z}+ne^{i\theta};K_1n)$ such that $T(\gamma)\leq(\nu+\epsilon)n$}
\end{aligned}
\right]\\
&\geq 1-\epsilon\quad\mbox{by (\ref{e113}), (\ref{e114}) and the
symmetry of $\eta\mathbb{T}$},
\end{align*}
which concludes the proof.
\end{proof}

\begin{proof}[Proof of Lemma \ref{l16}]
We will use a renormalization argument.  For simplicity, we prove this lemma in the case $u=1$; the proof for a general $u\in\mathbb{U}$
is analogous.  By Lemma \ref{l9}, there exists $K\in\mathbb{N}$ such that for each $\epsilon>0$, we can choose $N=N(\epsilon)\in\mathbb{N}$ and $p_0=p_0(\epsilon)\in(0,p_c)$, such that for each $p\in(p_0,p_c)$
and each bond $e=(e_-,e_+)\in E(\mathbb{Z}^2)$,
\begin{equation}\label{e26}
\mathbf{P}_p^{\eta}[\mbox{$\exists$ a path $\gamma$ from $\mathcal {C}^p(Ne_-)$ to $\mathcal {C}^p(Ne_+)$ in $Box(Ne_-,Ne_+;KN)$ such that $T(\gamma)\leq(\nu+\epsilon)N$}]\geq 1-\epsilon.
\end{equation}
We declare a bond $e\in E(\mathbb{Z}^2)$ to be \textbf{$\epsilon$-good} if the
event in (\ref{e26}) occurs.  We call a (bond) path of
$\mathbb{Z}^2$ $\epsilon$-good if all the bonds of this path are $\epsilon$-good.  The
family $X^p(\epsilon)=\{X_e^p(\epsilon):e\in E(\mathbb{Z}^2)\}$ is
defined by
\begin{equation*}
X_e^p(\epsilon)(\omega_p^{\eta}):=\left\{
\begin{aligned}
&1 &\mbox{ if $e$ is $\epsilon$-good},\\
&0 &\mbox{ otherwise}.
\end{aligned}
\right.
\end{equation*}
It is easy to check that $X^p(\epsilon)$ is a $6K$-dependent family; see Figure \ref{renorm}.
\begin{figure}
\begin{center}
\includegraphics[height=0.3\textwidth]{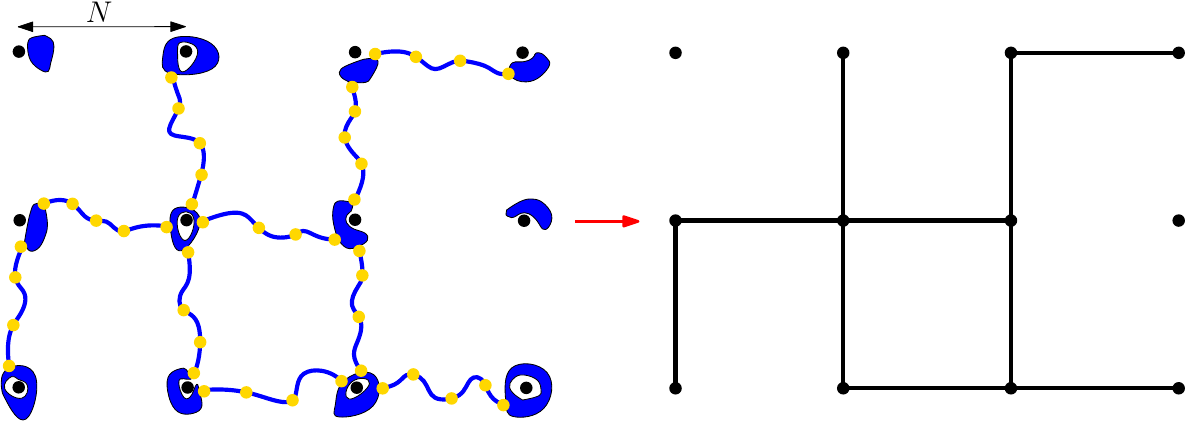}
\caption{Bernoulli site percolation configuration $\omega_p^{\eta}$ on $\eta\mathbb{T}$ induces the $6K$-dependent family $X^p(\epsilon)$, which is a $6K$-dependent bond percolation on $\mathbb{Z}^2$.
}\label{renorm}
\end{center}
\end{figure}

We have from Proposition \ref{p7} that for each
$\epsilon>0$, there is $p_1=p_1(\epsilon)\in (0,1/2)$ such that for all large $n$ (depending on $\epsilon$),
\begin{equation}\label{e27}
\mathbf{P}_{\mathbb{Z}^2,p_1}^{bond}[D(0,n)\leq (1+\epsilon)n]\geq
3/4.
\end{equation}
Our proof also works if the value $3/4$ in inequality (\ref{e27}) is
replaced with any other fixed value in (0,1).  By (\ref{e26}) and
Theorem \ref{t8}, for each $\epsilon>0$, we can choose $\epsilon_1=\epsilon_1(\epsilon)\in(0,\epsilon)$, such
that for all $p\in(p_0(\epsilon_1),p_c)$, $X^p(\epsilon_1)$
stochastically dominates the Bernoulli bond percolation on
$\mathbb{Z}^2$ with measure $\mathbf{P}_{\mathbb{Z}^2,p_1}^{bond}$. (Recall that in the present paper, each bond of $\mathbb{Z}^2$
takes the value 0 with probability $p_1$ under
$\mathbf{P}_{\mathbb{Z}^2,p_1}^{bond}$.)  Note that for the configuration $\omega_p^{\eta}$, if there is an $\epsilon_1$-good path of $\mathbb{Z}^2$
from 0 to $n$ so that the number of bonds of this path is not larger than $(1+\epsilon)n$, then
\begin{equation*}
T(\mathcal {C}^p(0),\mathcal
{C}^p(Nn))\leq(1+\epsilon)(\nu+\epsilon_1)Nn\leq(1+\epsilon)(\nu+\epsilon)Nn.
\end{equation*}
Therefore, for all $p\in(p_0(\epsilon_1),p_c)$ and all large $n$,
\begin{equation*}
\mathbf{P}_p^{\eta}[T(\mathcal {C}^p(0),\mathcal {C}^p(Nn))\leq
(1+\epsilon)(\nu+\epsilon)Nn]\geq
\mathbf{P}_{\mathbb{Z}^2,p_1}^{bond}[D(0,n)\leq (1+\epsilon)n]\geq
3/4.
\end{equation*}
Then we have
\begin{equation*}
\mathbf{P}_p^{\eta}[T(0,Nn)\leq(1+\epsilon)(\nu+\epsilon)Nn+2/\eta]\geq\mathbf{P}_p^{\eta}[T(\mathcal {C}^p(0),\mathcal {C}^p(Nn))\leq(1+\epsilon)(\nu+\epsilon)Nn]\geq 3/4.
\end{equation*}
By using this and the fact that $\mathbf{P}_p$-almost surely $a_{0,m}/m$
tends to $\mu(p)$ as $m\rightarrow\infty$ (see
(\ref{e28})), we obtain that ${L(p)\mu(p)\leq(1+\epsilon)(\nu+\epsilon)}$ for all $p\in(p_0(\epsilon_1),p_c)$.  Letting $\epsilon\rightarrow 0$, we have
$\limsup_{p\uparrow p_c}L(p)\mu(p)\leq\nu$.
\end{proof}
\begin{remark}
Lemma \ref{l9} is a key ingredient for the renormalization argument above.  The following stronger version of Lemma \ref{l9} for a fixed $\theta$ was proved in \cite{Yao21}:

Fix any $\theta\in[0,2\pi]$.  For each $\epsilon>0$, there
exist $K,N\in\mathbb{N}$, such that for all $n\geq N$, there exists $p_0=p_0(\theta,\epsilon,n)\in(0,p_c)$ such that
for all $p\in(p_0,p_c)$ and $z\in\mathbb{C}$,
\begin{equation*}
\mathbf{P}_p^{\eta}\left[\mbox{there is a path $\gamma$ from $\mathcal {C}^p(z)$
to $\mathcal {C}^p(z+ne^{i\theta})$ in $Box(z,z+ne^{i\theta};K)$ such that
$T(\gamma)\leq(\nu+\epsilon)n$}\right]\geq 1-\epsilon.
\end{equation*}
An application of this result may yield an 1-dependent bond percolation in a similar renormalization argument as above.  However, to show this version of Lemma \ref{l9} we need to construct the scaling limit of the collection of pieces of clusters contained in a strip (see Remark \ref{portion}) and study FPP in strips, which lead to many technical complications.  See \cite{Yao21} for more details.
\end{remark}

\subsection{Lower bound}\label{lower}
The goal of this subsection is to prove:
\begin{lemma}\label{l17}
For each $u\in\mathbb{U}$, we have $\liminf_{p\uparrow
p_c}L(p)\mu(p,u)\geq\nu$.
\end{lemma}
In \cite{GK84}, Grimmett and Kesten used a ``block approach'' to
obtain exponential large deviation bounds for first-passage times
for a single FPP model.  We will use their method with some
modifications to prove Lemma \ref{l17}, concerning the family
$\{\mathbf{P}_p^{\eta}\}_{p\in(0,p_c)}$ of Bernoulli FPP measures.
Similarly to the proof of Lemma \ref{l16}, the proof of Lemma
\ref{l17} also employs a type of renormalization, but it is quite
different and more complicated.

The proof of Lemma \ref{l17} is divided into two parts: estimates for line-to-line passage times
and a renormalization argument, which are given in Sections \ref{linetoline} and \ref{renormalization}, respectively.

\subsubsection{Line-to-line passage time}\label{linetoline}
Let $\mathbb{D}^{\eta}(z)$ denote the union of all hexagons of $\eta\mathbb{H}$ which intersect $\mathbb{D}(z)$.  We require the following lemma in order to study the line-to-line passage time $l_{n,m}^{p,\theta}(z)$ defined above Lemma \ref{l24}.  The main difficulty of its proof comes from the endpoints of a geodesic between $\mathbb{D}^{\eta}(x)$ and $\mathbb{D}^{\eta}(y)$.  More precisely, we need to bound the passage time from
$\mathcal {C}^p(\tilde{x})$ (resp. $\mathcal {C}^p(\tilde{y})$) to this geodesic, where $\tilde{x}$ is a well-chosen point in $\mathbb{D}(x)$.
\begin{lemma}\label{l28}
For each $\epsilon\in(0,\nu)$ and $\delta>0$, there exists an integer
$N=N(\epsilon,\delta)\geq 10$, such that for any integers $n$ and $m$ with $N\leq n\leq m$,
there exists $p_0=p_0(\epsilon,\delta,m)\in(0,p_c)$ such that for
all $p\in(p_0,p_c)$ and $x,y\in\mathbb{C}$ with
$n\leq |x-y|\leq m$,
\begin{equation*}
\mathbf{P}_p^{\eta}[T(\mathbb{D}^{\eta}(x),\mathbb{D}^{\eta}(y))\leq
(\nu-\epsilon)n]\leq\delta.
\end{equation*}
\end{lemma}
\begin{proof}
It suffices to prove the lemma in the case $y=0$.  Since the law of $\mathscr{C}$ is invariant under rotations by Theorem \ref{t3}, for any fixed $n$ the passage times $T_{0,n}^{\theta}$ for $\theta\in[0,2\pi]$ have the same distribution.  This and Proposition \ref{p2} imply that for each
$\epsilon\in(0,\nu)$ and $\delta>0$, there exists an integer
$N_1=N_1(\epsilon,\delta)\geq 10$ such that for all $\theta\in[0,2\pi]$ and all integers $n\geq N_1$,
\begin{equation}\label{e101}
\mathbf{P}^{\infty}[T_{0,n}^{\theta}\leq(\nu-\epsilon/2)n]\leq\delta/4.
\end{equation}

Combining (\ref{e101}) and Lemma \ref{l29}, we get that for each
$\epsilon\in(0,\nu)$ and $\delta>0$, there exists $N_1\geq 10$ such that
for any integers $m\geq N_1$,
there exists $p_0=p_0(\epsilon,\delta,m)\in(0,p_c)$ such that for
all $p\in(p_0,p_c)$, $\theta\in[0,2\pi]$ and $n\in\mathbb{N}$ with ${N_1\leq n\leq m}$,
\begin{equation}\label{e102}
\mathbf{P}_p^{\eta}[T_{0,n}^{p,\theta}\leq(\nu-\epsilon/2)n]\leq
\delta/2.
\end{equation}

For $x\in\mathbb{C}$ and $p\in(p(10),p_c)$, we write
\begin{align*}
&X^p(x)(\omega_p^{\eta}):=\inf\{T(\gamma):\gamma\mbox{ is a
circuit surrounding $x$ in $A(x;2,3)$}\},\\
&Y^p(x)(\omega_p^{\eta}):=\inf\{T(\gamma):\gamma\mbox{ is a path
from a site in $\mathcal {C}^p(x)$ to a site outside
$\Lambda_3(x)$}\}.
\end{align*}

It is clear that for each $x\in\mathbb{C}$ with $|x|\geq 10$,
there is $\tilde{x}\in\mathbb{C}$ such that
$\arg(\tilde{x})=\arg(x)$, $|x-\tilde{x}|\leq 1/2$ and
$|\tilde{x}|$ is an integer.  Suppose $p\in(p(10),p_c)$.  Then we have
$\Lambda_2(\tilde{x})\supset\mathbb{D}(x)\supset\mathbb{D}_{1/2}(\tilde{x})\supset\mathcal
{C}^{p}(\tilde{x})$.  Observe that for all $x\in\mathbb{C}$ with $|x|\geq
10$,
\begin{equation}\label{e103}
T_{0,|\tilde{x}|}^{p,\arg(\tilde{x})}\leq
X^p(0)+Y^p(0)+X^p(\tilde{x})+Y^p(\tilde{x})+T(\mathbb{D}^{\eta}(0),\mathbb{D}^{\eta}(x)).
\end{equation}
Applying Lemma \ref{l24}, it is easy to obtain that for each
$\delta>0$, there is $N_2=N_2(\delta)\geq 1$ such that for all
$x\in\mathbb{C}$ and $p\in(p(10),p_c)$,
\begin{equation}\label{e104}
\mathbf{P}_p^{\eta}[X^p(x)\geq N_2]\leq\delta/10.
\end{equation}
Using a similar but simpler argument as in the proof of (\ref{e11}),
one obtains that for each $\delta>0$, there is
$N_3=N_3(\delta)\geq 1$ such that for all $x\in\mathbb{C}$ and
$p\in(p(10),p_c)$,
\begin{equation}\label{e105}
\mathbf{P}_p^{\eta}[Y^p(x)\geq N_3]\leq\delta/10.
\end{equation}
Therefore, for each $\epsilon\in(0,\nu)$ and $\delta>0$, there is an
integer $N=N(\epsilon,\delta)=\max\{N_1,\lceil10N_2/\epsilon\rceil,\lceil10N_3/\epsilon\rceil\}$,
such that for any integers $n$ and $m$ with $N\leq n\leq m$,
there exists $p_0=p_0(\epsilon,\delta,m)\in(0,p_c)$ such that for
all $p\in(p_0,p_c)$ and $x\in\mathbb{C}$ with $n\leq |x|\leq m$,
\begin{align*}
&\mathbf{P}_p^{\eta}[T(\mathbb{D}^{\eta}(0),\mathbb{D}^{\eta}(x))\leq
(\nu-\epsilon)n]\\
&\leq\mathbf{P}_p^{\eta}[T_{0,|\tilde{x}|}^{p,\arg(\tilde{x})}-
X^p(0)-Y^p(0)-X^p(\tilde{x})-Y^p(\tilde{x})\leq
(\nu-\epsilon)n]\quad\mbox{by
(\ref{e103})}\\
&\leq\mathbf{P}_p^{\eta}[X^p(0)\geq\epsilon
n/10]+\mathbf{P}_p^{\eta}[Y^p(0)\geq\epsilon
n/10]+\mathbf{P}_p^{\eta}[X^p(\tilde{x})\geq\epsilon
n/10]+\mathbf{P}_p^{\eta}[Y^p(\tilde{x})\geq\epsilon
n/10]\\
&\quad{}~+\mathbf{P}_p^{\eta}[T_{0,|\tilde{x}|}^{p,\arg(\tilde{x})}\leq
(\nu-\epsilon/2)n]\\
&\leq\delta\quad\mbox{by (\ref{e104}), (\ref{e105}) and
(\ref{e102})},
\end{align*}
which ends the proof of the lemma.
\end{proof}

The following lemma gives estimates for the line-to-line passage times.  It
is similar to part (a) of Theorem 2.1 in \cite{GK84}.
\begin{lemma}\label{l23}
For each $\epsilon\in(0,\nu),\delta>0$ and $C\geq 1$, there is an integer
$N=N(\epsilon,\delta,C)\geq 3$ such that for any integers $n$ and $m$ with
$N\leq n\leq m\leq Cn$, there exists
$p_0=p_0(\epsilon,\delta,C,n)\in(0,p_c)$ such that for all $p\in
(p_0,p_c)$, $\theta\in[0,2\pi]$ and $z\in\mathbb{C}$,
\begin{equation*}
\mathbf{P}_p^{\eta}[l_{n,m}^{p,\theta}(z)\leq
n(\nu-\epsilon)]\leq\delta.
\end{equation*}
\end{lemma}
\begin{figure}
\begin{center}
\includegraphics[height=0.4\textwidth]{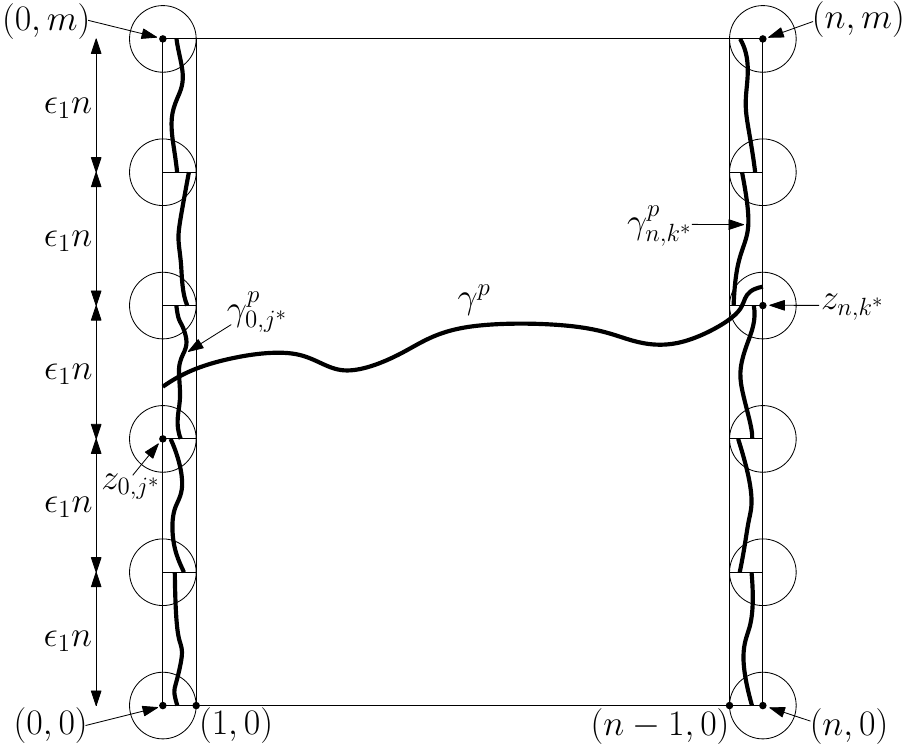}
\caption{The construction appearing in the proof of Lemma \ref{l23}.
}\label{boxcrossing}
\end{center}
\end{figure}
\begin{proof}
For simplicity, we prove Lemma \ref{l23} in the case $\theta=0$ and $z=0$; the proof extends easily to the general case.  Let us sketch the main idea.  As depicted in Figure \ref{boxcrossing}, we choose two sequence of unit balls along the left and right sides of $[0,n]\times[0,m]$, respectively.  We can control
simultaneously on the passage time between any ball in the left to any ball in the right, such that it is not too short; we can also control simultaneously on the top-bottom crossing time of any well-chosen small box which connects two consecutive balls in the left and in the right, such that it is not too long.  From this we deduce that the left-right crossing time of $[0,n]\times[0,m]$ is not too short.

Let $C\geq 1,0<\epsilon<\min\{1,\nu\}$, $\epsilon_1=\epsilon/(3K)$, $2/\epsilon_1\leq n\leq m\leq Cn$ with $n,m\in\mathbb{N}$, $p\in(p(10),p_c)$ and $j,k\in\mathbb{Z}_+$, where $K\geq 2$ is a fixed constant as in Lemma \ref{l24}.  Then, we define
\begin{align*}
X_j^p:=&\inf\{T(\gamma)(\omega_p^{\eta}):\gamma\mbox{ is a top-bottom crossing of $[0,1]\times[j\epsilon_1 n,(j+1)\epsilon_1 n]$}\},\\
Y_k^p:=&\inf\{T(\gamma)(\omega_p^{\eta}):\gamma\mbox{ is a top-bottom crossing of $[n-1,n]\times[k\epsilon_1 n,(k+1)\epsilon_1 n]$}\},\\
Z_{j,k}^p:=&T(\mathbb{D}^{\eta}(z_{0,j}),\mathbb{D}^{\eta}(z_{n,k}))(\omega_p^{\eta})\mbox{
for points $z_{0,j}:=(0,j\epsilon_1 n)$ and $z_{n,k}:=(n,k\epsilon_1
n)$.}
\end{align*}
For $\omega_p^{\eta}$, we let $\gamma^p$ be a left-right crossing of
$[0,n]\times[0,m]$ with $T(\gamma^p)=l_{n,m}^p$, $\gamma_{0,j}^p$ be a top-bottom crossing of
$[0,1]\times[j\epsilon_1 n,(j+1)\epsilon_1 n]$ with $T(\gamma_{0,j}^p)=X_j^p$, and
$\gamma_{n,k}^p$ be a top-bottom crossing of $[n-1,n]\times[k\epsilon_1 n,(k+1)\epsilon_1 n]$ with
$T(\gamma_{n,k}^p)=Y_k^p$.  It is clear that
$\mathbb{D}^{\eta}(z_{0,j})\cap\gamma_{0,j}^p\neq\emptyset$ and
$\mathbb{D}^{\eta}(z_{n,k})\cap\gamma_{n,k}^p\neq\emptyset$.
Moreover, $\gamma^p$ intersects
$\mathbb{D}^{\eta}(z_{0,j^*})$ or $\gamma_{0,j^*}^p$ for some integer
$j^*\in[0,\lceil m/(\epsilon_1 n)\rceil]$, and intersects $\mathbb{D}^{\eta}(z_{n,k^*})$ or $\gamma_{n,k^*}^p$ for
some integer $k^*\in[0,\lceil m/(\epsilon_1 n)\rceil]$.  See Figure \ref{boxcrossing} for an illustration.  Assume that the event
\begin{equation*}
\left(\bigcap_{j=0}^{\lceil m/(\epsilon_1
n)\rceil}\{X_j^p\leq\epsilon
n/3\}\right)\cap\left(\bigcap_{k=0}^{\lceil m/(\epsilon_1
n)\rceil}\{Y_k^p\leq\epsilon n/3\}\right)\cap \{l_{n,m}^p\leq
n(\nu-\epsilon)\}
\end{equation*}
holds.  Then
\begin{equation*}
Z_{j^*,k^*}^p\leq
T(\gamma^p)+T(\gamma_{0,j^*}^p)+T(\gamma_{n,k^*}^p)\leq
n(\nu-\epsilon/3).
\end{equation*}
Note that $m/(\epsilon_1 n)\leq 3CK/\epsilon$ since $n\leq m\leq Cn$
and $\epsilon_1=\epsilon/(3K)$.  Then from the above argument, we
have
\begin{align}
\mathbf{P}_p^{\eta}[l_{n,m}^p\leq
n(\nu-\epsilon)]\leq&\sum_{0\leq j\leq\lceil
3CK/\epsilon\rceil}\mathbf{P}_p^{\eta}[X_j^p\geq\epsilon
n/3]+\sum_{0\leq k\leq\lceil
3CK/\epsilon\rceil}\mathbf{P}_p^{\eta}[Y_k^p\geq \epsilon
n/3]\nonumber\\
&{}+\sum_{j,k\in[0,\lceil
3CK/\epsilon\rceil]}\mathbf{P}_p^{\eta}[Z_{j,k}^p\leq
(\nu-\epsilon/3)n].\label{e36}
\end{align}
Let us bound the three terms on the right side of (\ref{e36}).  By
Lemma \ref{l24}, there exists $N_0=N_0(\epsilon,\delta,C)\geq
2/\epsilon_1$ such that for all $p\in(p(10),p_c)$, $j,k\in\mathbb{Z}_+$ and $n\geq N_0$, we have
\begin{equation}\label{e73}
\mathbf{P}_p^{\eta}[X_j^p\geq\epsilon n/3]\leq
\delta\epsilon/(12CK)\quad\mbox{and}\quad\mathbf{P}_p^{\eta}[Y_k^p\geq
\epsilon n/3]\leq\delta\epsilon/(12CK).
\end{equation}
By Lemma \ref{l28}, there is an integer $N=N(\epsilon,\delta,C)\geq N_0$ such that for all
$n\geq N$, there exists $p_0=p_0(\epsilon,\delta,C,n)\in(0,p_c)$
such that for all $p\in(p_0,p_c)$ and $j,k\in [0,\lceil
3CK/\epsilon\rceil]$,
\begin{equation}\label{e74}
\mathbf{P}_p^{\eta}[Z_{j,k}^p\leq
(\nu-\epsilon/3)n]\leq\delta\epsilon^2/(48C^2K^2).
\end{equation}
Plugging (\ref{e73}) and (\ref{e74}) into (\ref{e36}), we obtain
$\mathbf{P}_p^{\eta}[l_{n,m}^p\leq n(\nu-\epsilon)]\leq\delta$
for all $N\leq n\leq m\leq Cn$ and all $p\in(p_0,p_c)$.
\end{proof}

\subsubsection{Renormalization}\label{renormalization}
We now follow the main lines of the proof of Theorem 3.1 in
\cite{GK84} to prove Lemma \ref{l17}.  For simplicity, we will focus on the case when $u=1$; the proof extends easily to the general case.  The idea is, roughly speaking, a path joining 0 to a point far away from 0 with a ``very short'' passage time should
cross ``many'' mesoscopic boxes which have ``very short'' box-crossing times (i.e., line-to-line passage times), and by
estimates of box-crossing times in Lemma \ref{l23} and a counting argument we show that it is unlikely that such a path exists.  Our notation in this subsection is analogous to that in \cite{GK84} but with a slight difference to fit our context.

\textbf{Step 1.} \emph{Decomposition of a path into segments.}
For each $k=(k_1,k_2)\in\mathbb{Z}^2$ and integers $M,N$ satisfying
$M>N>1$, define the boxes
\begin{align*}
&S(k):=\{z\in\mathbb{R}^2:Nk\leq
z<N(k+(1,1))\},\\
&\widehat{S}(k):=\{z\in\mathbb{R}^2:Nk-(M,M)\leq
z<N(k+(1,1))+(M,M)\}.
\end{align*}
Note that $\widehat{S}(k)$ contains $S(k)$ at its center.  Later, we
shall choose $M$ such that it is much larger than $N$, but much
smaller than $n$.  In the following we assume that $p\in(p(10),p_c)$ and view the sites of
$\eta\mathbb{T}$ as points in $\mathbb{R}^2$.  The boxes $S(k)$'s induce a partition of the sites of $\eta\mathbb{T}$,
and will play the role of ``renormalized sites''.
\begin{figure}
\begin{center}
\includegraphics[height=0.45\textwidth]{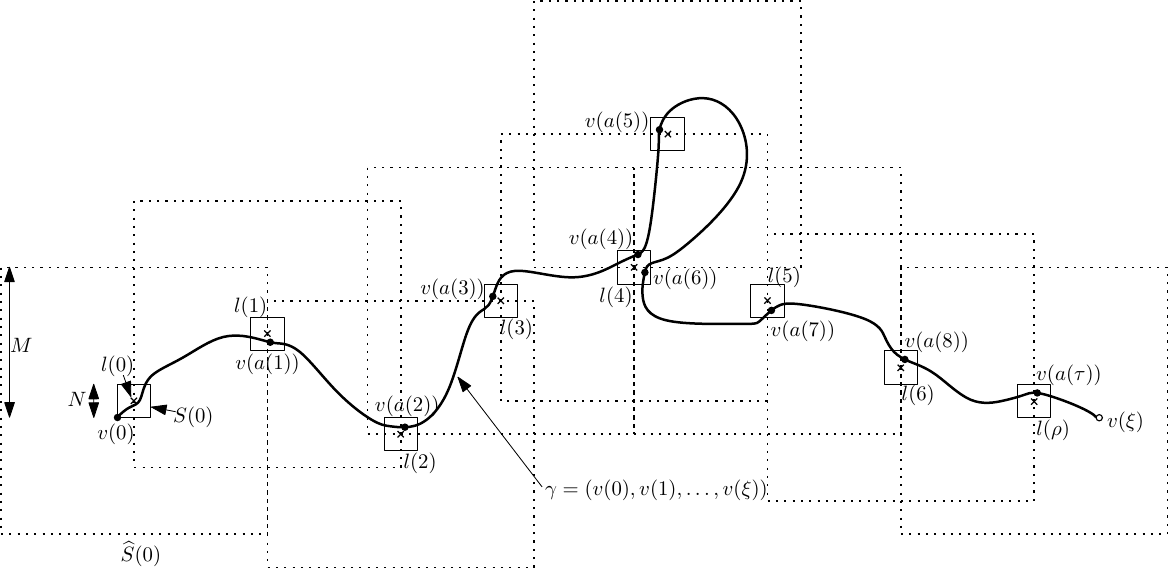}
\caption{The decomposition of a path $\gamma$ into segments in Step 1.  The black dots represent $v(a(i))$'s; the small crosses represent $k(i)$'s.  In this picture, we have $\widetilde{\sigma}=(l(0),\ldots,l(\rho))=(k(0),k(1),k(2),k(3),k(4),k(7),k(8),k(9))$, where $\widetilde{\sigma}$ is the subsequence extracted from $\sigma=(k(0),\ldots,k(\tau))$ by the loop removal process.
}\label{kesten}
\end{center}
\end{figure}

Let $\gamma=(v(0),v(1),\ldots,v(\xi))$ be a path in
$\eta\mathbb{T}$ from $v(0)=0$ to some site $v(\xi)$ in
$\{(z_1,z_2)\in\mathbb{R}^2:z_1\geq n\}$.  We associate to $\gamma$
the following two sequences.  First, let $k(0)=0$ and $a(0)=0$.  Then let
$v(a(1))$ be the first site along $\gamma$ to be outside
$\widehat{S}(k(0))$, and let $k(1)$ be the unique $k$ such that
$v(a(1))\in S(k)$. Continue recursively to find sequences
$(a(0),a(1),\ldots,a(\tau))$ and
$\sigma:=(k(0),k(1),\ldots,k(\tau))$ such that
\begin{enumerate}
\item $0=a(0)<a(1)<\cdots<a(\tau)\leq\xi$,
\item $v(a(i))\in S(k(i))$,
\item $a(i+1)$ is the smallest integer $a$ larger than $a(i)$ such
that $v(a)\notin\widehat{S}(k(i))$.
\end{enumerate}
The final terms $a(\tau)$ and $k(\tau)$ satisfy
\begin{equation*}
v(j)\in\widehat{S}(k(\tau))\quad\mbox{if }a(\tau)\leq j\leq \xi.
\end{equation*}
By the classical ``loop removal'' process described in \cite{GK84}, we obtain a subsequence $\widetilde{\sigma}$ of $\sigma$ which is free
of double points:
\begin{equation*}
\widetilde{\sigma}:=(l(0),\ldots,l(\rho)),
\end{equation*}
where $l(i)=k(j_i)$ for $i=0,1,\ldots,\rho$ and $0=j_0<j_1<\cdots<j_{\rho}\leq\tau$.  See Figure \ref{kesten} for the construction.  Note
that although $j_{\rho}$ and $\tau$ may not be equal, it is always true
that $k(j_{\rho})=k(\tau)$.  By construction,
\begin{equation*}
\|k(j+1)-k(j)\|_{\infty}\leq\frac{M}{N}+1\quad\mbox{for $j=0,1,\ldots,\tau-1$}.
\end{equation*}
This property is preserved by the loop removal process, in that
\begin{equation*}
\|l(j+1)-l(j)\|_{\infty}\leq\frac{M}{N}+1\quad\mbox{for $j=0,1,\ldots,\rho-1$}.
\end{equation*}

Consider the portion $\gamma(i):=(v(a(i-1)),\ldots,v(a(i)))$ of
$\gamma$ which stretches between $S(k(i-1))$ and $S(k(i))$, and
define
\begin{equation*}
L(i):=\|v(a(i))-v(a(i-1))\|_{\infty}.
\end{equation*}
By construction,
\begin{equation}\label{e75}
M\leq L(i)<M+N+1\quad\mbox{for }1\leq i\leq \tau.
\end{equation}
(A stronger version of
(\ref{e75}) holds: $M\leq L(i)\leq M+N+\eta$ for $1\leq i\leq
\tau$.)  Not that (\ref{e75}) is slightly different from the corresponding
equation (3.5) ``$M\leq L(i)\leq M+N$'' in \cite{GK84} for a bond FPP on $\mathbb{Z}^2$.

\textbf{Step 2.} \emph{Coloring of points and properties of white points.}
Let $0<\epsilon<\nu/2$.  For $1\leq i\leq \rho$, consider the point
$l(i)=k(j_i)\in\mathbb{Z}^2$ and the portion $\gamma(j_i)$
stretching between the two boxes $S(k(j_i-1))$ and $S(k(j_i))$.  We color $l(i)$
\textbf{white} if
\begin{equation*}
T(\gamma(j_i))\leq(\nu-2\epsilon)L(j_i);
\end{equation*}
otherwise we color $l(i)$ \textbf{black}.  Denote by $w$ the number of
white points in the sequence $(l(1),\ldots,l(\rho))$.  The next
lemma gives estimates for $w$ and $\rho$ in terms of $\epsilon,n,M,N$, and corresponds to Lemma 3.5 of \cite{GK84}.
\begin{lemma}[Number of white points associated to a fast path]\label{l25}
Suppose that $p\in(p(10),p_c)$ and $\gamma$ is
a path in $\eta\mathbb{T}$ from $0$ to some site in
$\{(z_1,z_2)\in\mathbb{R}^2:z_1\geq n\}$.  Suppose that
$\epsilon,n,M,N$ and $\gamma$ satisfy the following:
\begin{enumerate}[(i)]
\item $\epsilon$ is small: $0<\epsilon<\nu/5$,\label{item1}
\item $M/N$ is large: $M(\nu-3\epsilon)\geq
(M+N+1)(\nu-4\epsilon)$,\label{item2}
\item $n$ is large: $n\epsilon\geq
(M+2N)(\nu-4\epsilon)$,\label{item3}
\item the passage time of $\gamma$ is small: $T(\gamma)\leq
n(\nu-5\epsilon)$.\label{item4}
\end{enumerate}
Then we have
\begin{equation*}
w\geq \frac{\epsilon\rho}{2\nu}\quad\mbox{and}\quad\rho\geq
\frac{n}{M+N}-1.
\end{equation*}
\end{lemma}
\begin{proof}
Note that (\ref{item1}), (\ref{item3}) and (\ref{item4}) above are
the same as those of Lemma 3.5 in \cite{GK84}, and (\ref{item2}) is
slightly different from the corresponding item of that lemma since (\ref{e75}) is slightly
different from equation (3.5) in \cite{GK84}.  The proof of Lemma
\ref{l25} is essentially the same as that of Lemma 3.5 in
\cite{GK84}.  The details are omitted.
\end{proof}

The following lemma is an analog of Lemma 3.6 in \cite{GK84}.  It says that a point is white implies that an annulus-crossing time is ``short'', and gives an upper bound on the probability that some given points are white, using the local dependence of white points.
\begin{lemma}[Local dependence of white points]\label{l26}
Suppose that $p\in(p(10),p_c)$ and $\gamma$ is
a path in $\eta\mathbb{T}$ from $0$ to some site in
$\{(z_1,z_2)\in\mathbb{R}^2:z_1\geq n\}$.  For any $1\leq i\leq\rho$, the event
$\{l(i)\mbox{ is white}\}$ is contained in the event
\begin{equation*}
\mathcal {E}^p(i):=\left\{
\begin{aligned}
&\mbox{a site in $S(k(j_i))$ is joined to a site outside the square}\\
&\mbox{$\{z\in \mathbb{R}^2: Nk(j_i)-(M-N,M-N)\leq
z<Nk(j_i)+(M,M)\}$ by}\\
&\mbox{a path in $\eta\mathbb{T}$ with passage time less than $(M+N+1)(\nu-2\epsilon)$}
\end{aligned}
\right\}.
\end{equation*}
For any subset $\mathscr{S}$ of $\{l(1),\ldots,l(\rho)\}$ we have
\begin{equation*}
\mathbf{P}_p^{\eta}[\mbox{all points in $\mathscr{S}$ are
white}]\leq\mathbf{P}_p^{\eta}[\mathcal {E}^p(i)\mbox{ occurs
for each }l(i)\in\mathscr{S}]\leq\phi^{\alpha|\mathscr{S}|},
\end{equation*}
where
\begin{align*}
\phi=\phi(M,N,\epsilon,p):=\sup_{k\in\mathbb{Z}^2}\mathbf{P}_p^{\eta}\left[
\begin{aligned}
&\mbox{a site in $S(k)$ is joined to a site
outside the square}\\
&\mbox{$\{z\in \mathbb{R}^2: Nk-(M-N,M-N)\leq z<Nk+(M,M)\}$ by}\\
&\mbox{a path in $\eta\mathbb{T}$ with passage time less than $(M+N+1)(\nu-2\epsilon)$}
\end{aligned}
\right],
\end{align*}
and $\alpha=\alpha(M,N):=\left(\frac{N}{8(M+N)}\right)^2$.
\end{lemma}
\begin{proof}
The proof is basically the same as for Lemma 3.6 in
\cite{GK84}.  The details are omitted.
\end{proof}
The next lemma corresponds to Lemma 3.7 in \cite{GK84}, with a
slight modification adapted to our setting.  It says that, the probability that a point is white can be made arbitrarily small by means of a suitable choice of $N$ and $M$, uniformly in $p$ close to $p_c$.  This lemma is a key ingredient for the counting argument in the final proof.
\begin{lemma}[Control of the probability that a point is white]\label{l27}
For each $\epsilon\in(0,\nu/5)$, $\delta>0$ and $C>2\nu/\epsilon$,
there is an integer $N=N(\epsilon,\delta,C)\geq\nu/\epsilon$ and a
$p_0=p_0(\epsilon,\delta,C,N)\in(p(10),p_c)$, such that for each
$p\in(p_0,p_c)$ and each integer $M\in[2\nu N/\epsilon, CN]$,
\begin{equation*}
\phi=\phi(M,N,\epsilon,p)\leq\delta.
\end{equation*}
\end{lemma}
\begin{proof}
The proof is analogous to the proof of Lemma 3.7 in \cite{GK84}. Let
$B_1,B_2,B_3,B_4$ denote the following boxes:
\begin{align*}
&B_1=[-M+N,M]\times[N,M],\qquad\quad B_2=[N,M]\times[-M+N,M],\\
&B_3=[-M+N,M]\times[-M+N,0],\quad B_4=[-M+N,0]\times[-M+N,M].
\end{align*}
For each $k\in\mathbb{Z}^2$, if a site in $S(k)$ is joined to a site outside the square $\{z\in
\mathbb{R}^2: Nk-(M-N,M-N)\leq z<Nk+(M,M)\}$ by a path in
$\eta\mathbb{T}$ with passage time not exceeding
$(M+N+1)(\nu-2\epsilon)$, then one of the boxes $Nk+B_j$, $j\in\{1,2,3,4\}$, is
crossed between its longer sides by a path with passage time not
exceeding $(M+N+1)(\nu-2\epsilon)$.  Thus
\begin{equation*}
\phi(M,N,\epsilon,p)\leq
4\sup_{k\in\mathbb{Z}^2,\theta\in\{0,\pi/2\}}\mathbf{P}_p^{\eta}\left[l_{M-N,2M-N}^{p,\theta}(k)\leq
(M+N+1)(\nu-2\epsilon)\right].
\end{equation*}
If $0<\epsilon<\nu/5$, $N\geq\nu/\epsilon$ and $M\geq 2\nu
N/\epsilon$, then
\begin{equation*}
2M-N\leq 3(M-N)\quad\mbox{and}\quad(M+N+1)(\nu-2\epsilon)\leq
(M-N)(\nu-\epsilon),
\end{equation*}
giving that
\begin{equation*}
\phi(M,N,\epsilon,p)\leq
4\sup_{k\in\mathbb{Z}^2,\theta\in\{0,\pi/2\}}\mathbf{P}_p^{\eta}\left[l_{M-N,3(M-N)}^{p,\theta}(k)\leq
(M-N)(\nu-\epsilon)\right].
\end{equation*}
Applying Lemma \ref{l23} we obtain the desired result immediately.
\end{proof}
\textbf{Step 3.} \emph{A counting argument, end of the proof.}
Now we are ready to prove Lemma \ref{l17} by using the preparatory lemmas in the previous step.
\begin{proof}[Proof of Lemma \ref{l17}]
We shall prove the lemma in the case $u=1$; the
proof extends easily to the general case by using the rotated lattice $u\cdot\mathbb{Z}^2$.  We will follow the
lines of the proof of Theorem 3.1 in \cite{GK84}, which is based on
a counting argument.

Recall that each path $\gamma$ in $\eta\mathbb{T}$ from $0$ to a
site in $\{(z_1,z_2)\in\mathbb{R}^2:z_1\geq n\}$ has an associated
sequence $l(0),\ldots,l(\rho)$.  It is easy to see that the number of possible
choices for this sequence is at most $(8(\frac{M}{N}+1))^{\rho}$.
Given the sequence of $l$'s, there are $\binom{\rho}{w}$ ways of choosing a set of cardinality $w$ as the
white points.  Suppose that $\epsilon\in(0,\nu/5)$, and let $\delta$
be a fixed constant satisfying $\delta\in(0,1)$ and
$48\nu\delta^{\epsilon^3/(1152\nu^3)}<\epsilon$.  By Lemma
\ref{l27}, we can choose fixed $M,N\in\mathbb{N}$ and
$p_0\in(p(10),p_c)$ such that
\begin{equation}\label{e106}
N\geq \nu/\epsilon,\quad 2\nu N\leq M\epsilon\leq (3\nu-\epsilon)N\quad\mbox{and}\quad
\phi(M,N,\epsilon,p)\leq\delta\quad\mbox{for all }p\in(p_0,p_c).
\end{equation}
Then (\ref{item1}) and (\ref{item2}) of Lemma \ref{l25} hold.
Suppose also that $n$ is large enough for (\ref{item3}) of Lemma
\ref{l25} to hold; by (\ref{item1}) and (\ref{item3}), we have
$n\geq M+2N$, giving by Lemma \ref{l25} that $\rho\geq Kn$ where
\begin{equation*}
K=K(M,N):=\frac{N}{(M+N)(M+2N)}.
\end{equation*}
The point-to-line passage time $b_{0,n}^p$ is defined by
\begin{equation*}
b_{0,n}^p:=\inf\{T(\gamma)(\omega_p^{\eta}):\mbox{$\gamma$ is a path in
$\eta\mathbb{T}$ from 0 to a site in
$\{(z_1,z_2)\in\mathbb{R}^2:z_1\geq n\}$}\}.
\end{equation*}
Then, by (\ref{e106}), Lemmas \ref{l25} and \ref{l26}, for all
$p\in(p_0,p_c)$ and all $n$ satisfying (\ref{item3}) of Lemma
\ref{l25}, we have
\begin{equation}\label{e107}
\mathbf{P}_p^{\eta}[b_{0,n}^p<n(\nu-5\epsilon)]\leq\sum_{\rho\geq
Kn}\left(8\left(\frac{M}{N}+1\right)\right)^{\rho}\sum_{w\geq
\epsilon\rho/(2\nu)}\binom{\rho}{w}\delta^{\alpha w},
\end{equation}
where $\alpha=\alpha(M,N)=\left(\frac{N}{8(M+N)}\right)^2$.
By an easy calculation (see the details at the end of the
proof of Theorem 3.1 in \cite{GK84}), we obtain from (\ref{e107})
that there are constants $K_1>0$ and $\delta_1\in(0,1)$, independent of
$n$ and $p$, such that
\begin{equation*}
\mathbf{P}_p^{\eta}[b_{0,n}^p<n(\nu-5\epsilon)]\leq
K_1\delta_1^n.
\end{equation*}
Hence, for each $p\in(p_0,p_c)$, we have $L(p)\mu(p)\geq
\nu-5\epsilon$ since $a_{0,m}/m$ tends to $\mu(p)$
$\mathbf{P}_p$-almost surely as $m\rightarrow\infty$ by (\ref{e28}).  Letting
$\epsilon\rightarrow 0$ yields $\liminf_{p\uparrow p_c}L(p)\mu(p)\geq\nu$.
\end{proof}

\subsection{Proof of Theorem \ref{t2}}\label{final}
Finally, it is easy to derive Theorem \ref{t2} from Lemmas \ref{l16} and \ref{l17}:
\begin{proof}[Proof of Theorem \ref{t2}]
Combining Lemmas \ref{l16} and \ref{l17}, we obtain that for each
$u\in\mathbb{U}$,
\begin{equation}\label{e108}
\lim_{p\uparrow p_c}L(p)\mu(p,u)=\nu.
\end{equation}
It remains to prove that the convergence in (\ref{e108}) is uniform
in $u\in\mathbb{U}$.
By (\ref{p8}) (or (\ref{e108})) and the fact that $\mu(p,z)$ is a norm on $\mathbb{C}$ for each fixed
$p\in(0,p_c)$ (see Section \ref{s1}), there is a constant $C\geq 1$ such that
\begin{equation}\label{e110}
L(p)\mu(p,u)\leq C\quad\mbox{for all $p\in(0,p_c)$ and all
$u\in\mathbb{U}$}.
\end{equation}

Fix $\epsilon\in(0,1)$.  Let
$K=K(\epsilon):=\lceil2C\pi/\epsilon\rceil$.  For
$k\in\{1,2,\ldots,K\}$, let $u_k=u_k(\epsilon):=e^{2\pi ki/K}$.  It follows from (\ref{e108}) that
there is $p_0=p_0(\epsilon)\in(0,p_c)$ such that for all
$p\in(p_0,p_c)$ and all $k\in\{1,2,\ldots,K\}$,
\begin{equation}\label{e111}
|L(p)\mu(p,u_k)-\nu|\leq\epsilon/2.
\end{equation}
It is obvious that for each $u\in\mathbb{U}$, there is
$\tilde{u}\in\{u_1,\ldots,u_K\}$ such that
$|u-\tilde{u}|<\pi/K$.  This, combined with (\ref{e111}), (\ref{e110}) and the
fact that $\mu(p,z)$ is a norm on $\mathbb{C}$ when $p\in(0,p_c)$,
implies that for any $u\in\mathbb{U}$ with
$u\notin\{u_1,\ldots,u_k\}$ and $p\in(p_0,p_c)$, we have
\begin{align*}
&L(p)\mu(p,u)\leq L(p)\mu(p,\tilde{u})+L(p)\mu(p,u-\tilde{u})\\
&\qquad\qquad\quad\leq
\nu+\frac{\epsilon}{2}+|u-\tilde{u}|L(p)\mu\left(p,\frac{u-\tilde{u}}{|u-\tilde{u}|}\right)\leq
\nu+\epsilon,\\
&L(p)\mu(p,u)\geq L(p)\mu(p,\tilde{u})-L(p)\mu(p,\tilde{u}-u)\\
&\qquad\qquad\quad\geq
\nu-\frac{\epsilon}{2}-|\tilde{u}-u|L(p)\mu\left(p,\frac{\tilde{u}-u}{|\tilde{u}-u|}\right)\geq
\nu-\epsilon.
\end{align*}
These two inequalities combined with (\ref{e111}) implies that for
each $\epsilon\in(0,1)$, there is $p_0\in(0,p_c)$ such
that
\begin{equation*}
|L(p)\mu(p,u)-\nu|\leq\epsilon\quad\mbox{for all $p\in(p_0,p_c)$ and all $u\in\mathbb{U}$}.
\end{equation*}
Theorem \ref{t2} follows from this immediately.
\end{proof}

\begin{acks}[Acknowledgments]
The author is especially grateful to an anonymous referee for 1) pointing out that a preliminary version of the proof of Theorem \ref{t2} can be simplified, without using the scaling limit of the collection of portions of clusters in a strip, and 2) suggesting us another definition of $\mathcal{C}(z)$ which makes the proof cleaner.  The referee's suggestions make the paper less technical and much shorter than its preprint version \cite{Yao21}.  The author thanks Geoffrey Grimmett for an invitation to the Statistical Laboratory in Cambridge University, and thanks the
hospitality of the Laboratory, where this project was initiated.  The author was supported by the National Key R\&D Program of China (No. 2020YFA0712700),
the National Natural Science Foundation of China (No. 12288201 and No. 11601505) and the Key Laboratory of Random Complex Structures and Data Science, CAS (No. 2008DP173182).
\end{acks}

\end{document}